\documentclass{mod}
\usepackage{url}
\usepackage{setspace}
\usepackage{scrextend}
\usepackage{cite}
\usepackage[T1]{fontenc}
\usepackage{titlesec}

\usepackage{mathrsfs}

\usepackage{yhmath}
\usepackage{graphicx}
\usepackage{hyperref}
\usepackage[shortlabels]{enumitem}
\usepackage{amsmath}
\usepackage{amsthm}
\usepackage{amssymb}
\usepackage{titletoc}
\usepackage{mathtools}

\usepackage[all]{xy}
\usepackage[toc,page]{appendix}

\usepackage{upgreek}
\usepackage{pifont}

\usepackage{ stmaryrd }

\numberwithin{equation}{section}
\renewcommand{\theequation}{\thesection.\arabic{equation}}

\makeatletter
\def\thm@space@setup{\thm@preskip=5pt
	\thm@postskip=5pt}
\makeatother
\newtheoremstyle{mystyle}      
{} 
{} 
{\itshape} 
{} 
{\bfseries} 
{.} 
{ } 
{} 

\renewcommand{\thefootnote}{\fnsymbol{footnote}}

\usepackage{perpage}
\MakePerPage{footnote}

\makeatletter
\renewcommand\@makefntext[1]{%
	\parindent 1em%
	\noindent
	\llap{\thefootnote\enspace}%
	#1
}
\makeatother

\theoremstyle{mystyle}
\newcounter{mainthm}

\newtheorem{thm}{Theorem}[section]

\newtheorem{lem}[thm]{Lemma}
\newtheorem{prop}[thm]{Proposition}
\newtheorem{cor}[thm]{Corollary}
\newtheorem{defn-thm}[thm]{Definition-Theorem}

\newtheorem{defn-lem}[thm]{Definition-Lemma}
\newtheorem{defn-prop}[thm]{Definition-Proposition}
\newtheorem{axiom}[thm]{Axiom}
\newtheorem{situation}[thm]{Situation}

\makeatletter
\def\thm@space@setup{\thm@preskip=5pt
	\thm@postskip=5pt}
\makeatother
\newtheoremstyle{mydefinition}      
{} 
{} 
{} 
{} 
{\itshape\bfseries} 
{.} 
{ } 
{} 

\theoremstyle{mydefinition}

\newtheorem{defn}[thm]{Definition}
\newtheorem{defnn}[thm]{"Definition"}

\newtheorem{convention}[thm]{Convention}

\newtheoremstyle{rmk}
{5pt}
{5pt}
{}
{}
{\itshape}
{.}
{.5em}
{}

\theoremstyle{rmk}
\newtheorem{axiom*}[thm]{Axiom}

\newtheoremstyle{def}
{5pt}
{5pt}
{\itshape}
{}
{\itshape\sffamily}
{.}
{.5em}
{}
\theoremstyle{def}
\newtheorem{rmk}[thm]{Remark}

\newtheoremstyle{note}
{8pt}
{5pt}
{\itshape}
{10pt}
{\bfseries}
{}
{.5em}
{}

\theoremstyle{note}

\makeatletter
\renewenvironment{proof}[1][\proofname]{\par
	\vspace{-\topsep}
	\pushQED{\qed}%
	\normalfont
	\topsep0pt \partopsep0pt 
	\trivlist
	\item[\hskip\labelsep
	\itshape
	#1\@addpunct{.}]\ignorespaces
}{%
	\popQED\endtrivlist\@endpefalse
	\addvspace{6pt plus 6pt} 
}
\makeatother

\newcommand{\id}{\mathrm{id}}

\newcommand{\mathds}[1]{\text{\usefont{U}{dsrom}{m}{n}#1}}
\newcommand{\one}{\mathds {1}}

\newcommand{\UD}{\mathscr{UD}}
\newcommand{\simud}{\stackrel{\mathrm{ud}}{\sim}}

\newcommand{\ev}{\mathrm{ev}}

\newcommand{\ia}{{\mathfrak a}}
\newcommand{\e}{\mathbf e}
\newcommand{\B}{\mathcal B}
\newcommand{\mi}{\mathfrak i}

\newcommand{\mI}{\mathfrak I}

\newcommand{\Tr}{\mathscr{T}}

\newcommand{\Cint}{C^{\mathrm{int}}}
\newcommand{\Cext}{C^{\mathrm{ext}}}
\newcommand{\CC}{{CC}}
\newcommand{\tri}{{\mathrm{tri}}}

\newcommand{\m}{\mathfrak m}

\newcommand{\M}{\mathfrak M}
\newcommand{\mc}{\mathfrak c}

\newcommand{\mC}{\mathfrak C}

\newcommand{\oi}{{[0,1]}}
\newcommand{\ab}{{[a,b]}}
\newcommand{\f}{\mathfrak f}
\newcommand{\F}{\mathfrak F}

\newcommand{\mH}{ H}
\newcommand{\h}{\mathfrak h}

\newcommand{\g}{\mathfrak g}
\newcommand{\G}{\mathfrak G}

\newcommand{\HL}{H^*(L)}
\newcommand{\OL}{\Omega^*(L)}

\DeclareMathOperator{\vol}{vol}

\DeclareMathOperator{\Sp}{Sp}
\DeclareMathOperator{\incl}{Incl}
\DeclareMathOperator{\eval}{Eval}

\DeclareMathOperator{\trop}{\mathfrak {trop}}
\DeclareMathOperator{\Hom}{Hom}

\DeclareMathOperator{\Corr}{Corr}

\newcommand*{\Scale}[2][4]{\scalebox{#1}{$#2$}}%
\titlecontents{subsection}
[0em]
{\footnotesize \vspace{0pt}}%
{\qquad \quad \thecontentslabel.\enspace}
{}
{\titlerule*[0.5pc]{.}\contentspage}%

\titlecontents{subsubsection}
[0em]
{\footnotesize \vspace{1pt}}%
{\qquad \qquad \thecontentslabel.\enspace}
{}
{\titlerule*[0.5pc]{.}\contentspage}%

\titlecontents{section}
[0.72em]
{\scriptsize\bfseries\vspace{3pt}}%
{\thecontentslabel.\enspace}
{}
{\titlerule*[0.38pc]{.}\contentspage}%

\titlecontents{subsection}
[0em]
{\scriptsize \vspace{0pt}}%
{\qquad \quad \thecontentslabel.\enspace}
{}
{\titlerule*[0.5pc]{.}\contentspage}%

\titleformat{\subsection}[runin]{
	\bfseries
	\normalsize}{\thesubsection \ }{0em}{}[\mbox{ . } ]
\titlespacing{\subsection}
{0pt}
{2.25ex plus 1ex minus .2ex}
{0pt}

\titleformat{\subsubsection}{
	\sffamily
	\itshape\normalsize}{\thesubsubsection \ }{0em}{}[\vspace{0em}]
\titlespacing{\subsubsection}
{0pt}
{
	1.25ex plus 1ex minus .2ex
}
{0pt}

\begin{document}
	\setlength{\parindent}{15pt}	\setlength{\parskip}{0em}
	
	\title[{\small Non-archimedean analytic continuation of unobstructedness}]{ {
			Non-archimedean analytic continuation of unobstructedness}
	}
	\author[Hang Yuan]{Hang Yuan}
	\begin{abstract} {\sc Abstract:}  
		The Floer cohomology and the Fukaya category are not defined in general.
		Indeed, while the issue of obstructions can be theoretically addressed by introducing bounding cochains, the actual existence of even one such bounding cochain is usually unknown.
		This paper aims to deal with the problem.
		We study certain non-archimedean analytic structure that enriches the Maurer-Cartan theory of bounding cochains.
		Using this rigid structure and family Floer techniques, we prove that within a connected family of graded Lagrangian submanifolds, if any one Lagrangian is unobstructed (in a slightly stronger sense), then all remaining Lagrangians are automatically unobstructed.
	\end{abstract}
	\maketitle
	%
	%


	\tableofcontents
	
	\hypersetup{
		colorlinks=true,
		linktoc=all,
		citecolor=gray
	}

	%
	%

	\section{Introduction}

	\subsection{Main result}
	Let $(X,\omega)$ be a symplectic manifold that is closed or convex at infinity. Assume $\mathcal S$ is a connected manifold.
	Suppose $\{L_s : s\in \mathcal S\}$ is a smooth family of spin closed Lagrangian submanifolds with vanishing Maslov classes or gradings.
	Fix an arbitrary base point $s_0\in \mathcal S$.
	
	\begin{thm}
		\label{Main_thm_this_paper}
		If $L_{s_0}$ is properly unobstructed, then all $L_s$ are properly unobstructed.
	\end{thm}
	
	The unobstructedness problem is critical for general Lagrangian Floer cohomology \cite{FOOOBookOne,FOOOSpectral}, general Fukaya category \cite{FuUnobstructed}, general splitting generation criterion \cite{AFOOO}, Thomas-Yau uniqueness \cite{thomas2002special,li2022quantitative,li2022thomas,imagi2016uniqueness}, open Gromov-Witten invariants \cite{FuCounting,solomon2016point}, and so on. However, its criterion is quite limited.

	The \textit{\textbf{proper unobstructedness}} is a new concept we propose in this paper (\textit{Definition} \ref{proper_unobstructed_introduction_defn}).
	It is a sufficient condition (\textit{Proposition} \ref{sufficient_intro_prop}) of the unobstructedness in the traditional sense, i.e., the existence of (weak) bounding cochains, which are solutions to the Maurer-Cartan equations.
	Instead of asking whether these equations have a single solution, we aim to explore the "geometry" of these equations or the adic-convergent formal power series they represent.
	The primary aim of this paper is to introduce a new non-archimedean perspective on unobstructedness in Floer theory with simpler proofs than \cite{Yuan_I_FamilyFloer}.

	\subsection{Applications}
	We will soon clarify the definition and reveal that proper unobstructedness pertains solely to Maslov-0 disks, so it holds when (i) $L_{s_0}$ is monotone, or (ii) $L_{s_0}$ is \textit{tautologically unobstructed} in the sense that it does not bound any holomorphic disks of non-positive Maslov indices. 
	
	
	The work of Rizell-Goodman-Ivrii \cite[Theorem A]{dimitroglou2016lagrangian} implies that any two Lagrangian tori inside the uniruled symplectic 4-manifold $X=\mathbb R^4$, $\mathbb {CP}^2$, or $S^2\times S^2$ are Lagrangian isotopic to each other.
	It is known that there exist monotone Clifford or Chekanov tori inside $X$.
	Therefore, if we choose one of them as $L_{s_0}$, then as a simple application of Theorem \ref{Main_thm_this_paper}, we immediately obtain
	\begin{prop}
		\label{Application_4mfd_prop}
		If $X$ is a symplectic 4-manifold as above, then any Lagrangian tori $L$ in $X$ is (properly) unobstructed. In particular, the filtered $A_\infty$ algebra associated to $L$ must admit a bounding cochain.
	\end{prop}

	In the context of the Strominger-Yau-Zaslow conjecture \cite{SYZ}, a folklore speculation is that all fibers in the SYZ Lagrangian fibration are unobstructed. Due to Theorem \ref{Main_thm_this_paper}, finding a single monotone (or tautologically unobstructed) Lagrangian fiber within the fibration is sufficient to establish this.
	There is also a concrete application regarding the work \cite{AAK_blowup_toric} of Abouzaid, Auroux, Katzarkov, where they construct a singular Lagrangian fibration $\pi:X \to B$ in a blow-up of certain toric variety. Many fibers are tautologically unobstructed in the above sense, and we may pick one of them as $L_{s_0}$. The singular locus of $\pi$ is a codimensional-one amoeba-type subset in its base. The \textit{wall region}, where fibers bound Maslov-0 disks, are codimension-zero in the base. Thus, the perturbative arguments are ineffective for the unobstructedness in the \textit{interior} of the wall region. The Floer-theoretic analysis in \cite{AAK_blowup_toric} takes place outside the wall region, suggesting that the SYZ construction may not be entirely complete.
	Now, as an initial step towards a deeper comprehension, Theorem \ref{Main_thm_this_paper} enables us to draw the following conclusion:

	\begin{prop}
		\label{Application_AAK_prop_intro}
		Any Lagrangian $\pi$-fiber is (properly) unobstructed. Thus, the filtered $A_\infty$ algebra associated to it must admit a bounding cochain.
	\end{prop}
	
	
	The above situations where $L_{s_0}$ is either monotone or tautologically unobstructed should be enough for a few applications.
	In a more general setting, Solomon's anti-symplectic involution criterion \cite{Solomon_Involutions} also implies the proper unobstructedness of $L_{s_0}$, with the Maslov-0 disk obstructions canceled pairwise, if $L_{s_0}$ is preserved by an anti-symplectic involution. Such an involution is often obtained by complex conjugation; see also \cite{castano2010fixed,Solomon_Symmetry_Lag}.
	Roughly, the result of Solomon is derived from studying the orientations of moduli spaces.
	It is also inspired by the earlier works of Welschinger \cite{welschinger2005invariants,welschinger2005spinor} and Georgieva \cite{georgieva2016open} for open Gromov-Witten invariants, as explained by Solomon and Tukachinsky in \cite[p1247]{solomon2016point}.
	In light of this, the notion of proper unobstructedness introduced in this paper seems reasonable.


	\subsection{Unobstructedness beyond the Maurer-Cartan picture}
	We briefly explain our new definition of the \textit{proper unobstructedness}. A key aspect is simply to prove that this definition is indeed independent of choices and therefore valid.
	Recall that a filtered $A_\infty$ algebra associated to a Lagrangian submanifold $L$ is a collection of multilinear maps of degree $2-k-\mu(\beta)$ on the de Rham complex $\Omega^*(L)$: (see e.g. \cite[21.2]{FOOO_Kuranishi})
	\[
	\check \m_{k,\beta}: \Omega^*(L)^{\otimes k} \to \Omega^*(L)
	\]
	for each $\beta\in\pi_2(X,L)$ and $k\geqslant 0$ with the following properties:
	$\check \m_{k,0}=0$ for $k\neq 1,2$; $\check \m_{2,0}(h_1,h_2)= \pm h_1\wedge h_2$; $\check \m_{1,0}(h)=dh$; and the following $A_\infty$ associativity relation
	\[
	\textstyle \sum_{k_1+k_2=k+1} \sum_{\beta_1+\beta_2=\beta} \sum_{i=1}^{k_1+1} (-1)^\ast \check \m_{k_1,\beta_1}(h_1,\dots, \check \m_{k_2,\beta_2} (h_i, \dots, h_{i+k_2-1}) ,\dots, h_k)=0
	\]
	This collection of operators relies on a tame almost complex structure $J$ and a choice $\Xi$ of `virtual fundamental chain' on the moduli space of $J$-holomorphic disks bounded by $L$ (see \cite{FuCyclic,FOOO_Kuranishi} or Section \ref{s_Axioms}).
	Choosing the Hodge decomposition of $\Omega^*(L)$ with respect to some metric $g$ on $L$ and applying the homological perturbation to the above operators $\check \m_{k,\beta}$'s, one can further produce a new collection of multilinear maps of degree $2-k-\mu(\beta)$ on the de Rham cohomology $H^*(L)$:
	\[
	\m_{k,\beta}: H^*(L)^{\otimes k} \to H^*(L)
	\]
	so that $\m_{1,0}=0$, $\m_{2,0}(h_1,h_2)= \pm h_1\wedge h_2$, and a similar $A_\infty$ associativity holds; see \cite{Yuan_I_FamilyFloer,FuCyclic}.
	Intuitively, one should consider that these new operators count not just holomorphic disks but `holomorphic pearly trees' \cite{Sheridan15,FOOO_2009canonical_Morse}, which is also related to the `clusters' of \cite{cornea2006cluster}.

	Denote the \textit{Novikov field} as $\Lambda = \mathbb{C}((T^{\mathbb{R}}))$ with its valuation ring, the Novikov ring, as $\Lambda_0$.
	The natural non-archimedean valuation $\mathrm{val}:\Lambda\to\mathbb R\cup\{\infty\}$ gives a norm $|x|=\exp(-\mathrm{val}(x))$ and allows us to talk about the convergence of infinite series. Define $U_\Lambda$ as the unitary subgroup in $\Lambda$ consisting of elements with norm 1 (or equivalently valuation 0). 
	Let $\one$ be the generator of $H^0(L)$.

	\subsubsection{Traditional unobstructedness}

	The (weak) Maurer-Cartan equation \cite{FOOOBookOne} refers to
	\begin{equation}
		\label{MC_eq}
		\textstyle \sum_{k=0}^\infty\sum_\beta  T^{\omega(\beta)} \m_{k,\beta} (b, \dots, b) =  w \cdot \one  \qquad \text{for} \ \  b\in H^{odd}(L)\hat\otimes \Lambda_0 \ \text{and} \  w  \in \Lambda_0
	\end{equation}
	We say $L$ is (weakly) \textit{unobstructed} if there \textit{exists} a single solution $b$ to (\ref{MC_eq}), called a (weak) \textit{bounding cochain} (see e.g. \cite[Definition 3.6.14, Definition 3.6.40]{FOOOBookOne}). From now on, we will often omit saying `weak'.
	One can define a $b$-deformed noncurved $A_\infty$ algebra $\{\m_{k,\beta}^b\}$ by setting $\m_{k,\beta}^b(h_1,\dots, h_k) =\sum \m_{\bullet, \beta}(b,\dots, b, h_1, b , \dots, b, h_2,\dots, h_k, b,\dots, b)$ \cite[3.6.9]{FOOOBookOne}.
	The set $\mathcal {MC}(\m)$ of the gauge equivalence classes of bounding cochains is called the \textit{Maurer-Cartan set};
	it depends on $\m$ and is invariant only up to bijection of sets \cite[4.3.14]{FOOOBookOne}.
	Note that the concept of a bounding cochain is also attributed to Kontsevich, as indicated in \cite{FOOOBookOne}.
	
	The existence of a bounding cochain $b$ paves the way for an established theory of Lagrangian Floer cohomology in \cite{FOOOBookOne}. However, analyzing the equation (\ref{MC_eq}), especially in terms of confirming the existence of such a $b$, poses significant challenges in general. Furthermore, ensuring the convergence of equation (\ref{MC_eq}) often requires that the coefficients of $b$ lie within the maximal ideal $\Lambda_+$ of $\Lambda_0$, which adds another layer of complexity to the problem.
	Notably, the Maurer-Cartan set or the solution space of (\ref{MC_eq}) is merely a \textit{set} without any interesting structure on it.
	

	\subsubsection{Proper unobstructedness}
	We propose a new notion of the unobstructedness that involves certain intrinsic non-archimedean analytic structure regarding the Lagrangian Floer theory.
	The Maurer-Cartan theory sketched above fits well with the homological algebra results but seems less natural with geometric viewpoints. For example, viewed as an open-string version of Gromov-Witten invariants, the $A_\infty$ algebra in our context further satisfies the so-called \textit{divisor axiom}: For a degree-one class $b\in H^1(L)$, we have
	\[
	\textstyle \sum_{i=1}^{k+1} \m_{k+1,\beta} (x_1,\dots, x_{i-1}, b, x_i,\dots, x_k ) = \langle \partial\beta, b\rangle \ \m_{k,\beta}(x_1,\dots, x_k)
	\]
	It is verified in \cite{FuCyclic} and can further transfer the equation (\ref{MC_eq}) into
	$
	\textstyle  \sum_\beta T^{\omega(\beta)} \mathbf y^{\partial\beta} \m_{0,\beta} =  w \cdot \one 
	$
	where we take $\mathbf y \in H^1(L; U_\Lambda)$ so that $ \mathbf y^{\partial\beta} = e^{\partial\beta\cap b}= \sum_k \frac{1}{k!} \langle \partial\beta, b\rangle^k$.
	While the equation (\ref{MC_eq}) generically refers to all $H^{odd}(L)$, the known solutions to (\ref{MC_eq}) usually only live in $H^1(L)$; see e.g. \cite[(1.1)]{FOOOToricOne}. This perfectly fits with the divisor axiom situation.

	Instead of using a series of numbers in the Novikov field $\Lambda$ like (\ref{MC_eq}) to measure the obstruction in Floer theory, we turn to directly use the formal power series 
	$
	\textstyle 
	\sum T^{\omega(\beta)} Y^{\partial\beta} \m_{0,\beta} $
	with $Y$ being a formal variable and with $b$ being forgotten.
	We expect that similar ideas might also work for $L_\infty$ structures.
	
	The shift in perspective becomes particularly coherent when all involved Maslov indices $\mu(\beta)\geqslant 0$, regarding the assumption of Theorem \ref{Main_thm_this_paper}. Note that by \cite[Lemma 3.1]{AuTDual}, this happens if $L$ is a special or graded Lagrangian.
	(For related studies that consider negative-Maslov disks, we direct readers to the recent work by Auroux \cite{auroux2023holomorphic}.)
	
	For $\mu(\beta)\geqslant 0$, the degree condition implies $\m_{0,\beta} \in H^{2-\mu(\beta)}(L)$, so it lies in either $H^0(L)$ or $H^2(L)$.
	Therefore, given a basis $\{\Theta_1,\dots, \Theta_\ell\}$ of $H^2(L)$ and the generator $\one$ in $H^0(L)$, we can find formal power series $W, Q_1,\dots, Q_\ell$ in $\Lambda[[H_1(L)]]$ such that
	\begin{equation}
		\label{W_Q_intro_eq}
		\sum T^{\omega(\beta)} Y^{\partial\beta} \m_{0,\beta} = W \cdot \one + \sum_{i=1}^\ell Q_i \cdot \Theta_i
	\end{equation}
	Here $\Lambda[[H_1(L)]]$ is the formal power series version of the group algebra, consisting of series in the form $\sum_{i=0}^\infty c_i Y^{\alpha_i}$ with $c_i\in\Lambda$ and $\alpha_i\in H_1(L)$. Using the minimal model $A_\infty$ algebra $\m$ rather than $\check\m$ is crucial because $\OL$ is infinite-dimensional and $\HL$ is finite-dimensional.
	
	\begin{defnn}
		\label{proper_unobstructed_introduction_defn}
		We call $L$ is \textit{properly unobstructed} if all the formal power series $Q_i$'s are vanishing.
	\end{defnn}

	We highlight that the vanishing of the $Q_i$'s does \textit{not} indicate the non-existence of Maslov-0 disks but rather that the contributions from the counts of Maslov-0 disks have canceled out in the coefficient affinoid algebra.
	
	Here we include the quotation mark to emphasize that the above definition is \textit{not} valid unless the following result can be proved.
	Beware that $W$ and $Q_i$'s all depend on the choice of $\m$.
	For a different choice $\m'$, suppose $W'$ and $Q_i'$ are defined in the same way as (\ref{W_Q_intro_eq}).
	It should be noted that the (adic) convergence domain $\mathcal U'$ for $Q_i'$ may differ from the one $\mathcal U$ for $Q_i$. 
	Now, we must prove:
	
	\begin{thm}[Theorem \ref{proper_unobstructedness_invariance_thm}]
		$Q_i\equiv 0$ for all $i$ if and only if $Q_i'\equiv 0$ for all $i$.
	\end{thm}

	In symplectic geometry, the independence of an objective on various choices is always crucial, but the degree of choice-independence can range widely, depending on which structure we are seeking, such as a set, ring, module, vector space, homotopy class of $A_\infty$ homotopy equivalences, or even isomorphism classes of Berkovich analytic spaces, ... , each with its distinct level of complexity; see e.g. \cite{solomon2016point,varolgunes2021mayer,FOOOBookOne,FOOOBookTwo,groman2015floer,MS,Yuan_I_FamilyFloer,auroux2014beginner,Ritter_Smith,FOOO_Kuranishi,FOOOSpectral,Ganatra_thesis}.
	Accordingly, the set-theoretic invariance of Maurer-Cartan set of weak bounding cochains is obviously not what we want.
	Having proved the aforementioned theorem, the main Theorem \ref{Main_thm_this_paper} is essentially a combination of this result and Fukaya's trick.
	
	Moreover, our new definition is reasonable since it gives a sufficient condition of the traditional one, as explained below.
	
	\begin{prop}
		\label{sufficient_intro_prop}
		The proper unobstructedness implies the usual unobstructedness.
	\end{prop}

	\begin{proof}
		Assume all $Q_i$'s are identically zero. We claim that simply picking a common zero $\mathbf y$ of $Q_i$'s within the subset $H^1(L; U_\Lambda)\cong U_\Lambda^m$ of their common convergence domain is sufficient to produce a bounding cochain. Indeed, there exists $x_i\in \Lambda_0$ such that $e^{x_i}=y_i$, whenever $y_i\in U_\Lambda$, i.e. $|y_i|=1$.
		This is a basic property of the Novikov field; cf. \cite{Yuan_I_FamilyFloer,FOOOToricOne}. Note that set-theoretically, we have
		\begin{equation}
			\label{Lambda_0_H_1_eq}
			H^1(L;\Lambda_0)/H^1(L;2\pi \sqrt{-1}\mathbb Z)\cong H^1(L;U_\Lambda)
		\end{equation}
		Then, reversing the arguments preceding Definition \ref{proper_unobstructed_introduction_defn} implies that $b:=(x_1,\dots, x_m)\in H^1(L; \Lambda_0)$ is a bounding cochain in the Maurer-Cartan set. Thus, $L$ is unobstructed in the usual sense.
	\end{proof}

	From a geometric point of view, let $Z$ denote the zero locus of all $Q_i$'s in their common convergence domain $\mathcal U$ in $H^1(L; \Lambda^*)\cong (\Lambda^*)^m$.
	Note that $\mathcal U_0:= H^1(L; U_\Lambda) \cong (U_\Lambda)^m$ is contained in $\mathcal U$ by Gromov's compactness, and $Z\subseteq \mathcal U$ is an analytic subvariety.
	In this context, the proper unobstructedness means $Z=\mathcal U$, while the usual unobstructedness just corresponds to the condition that $Z\cap \mathcal U_0 \neq \varnothing$.
	To sum up, proper unobstructedness is much stronger, requiring the \textit{identical vanishing} of $Q_i$'s, but usual unobstructedness (see \cite[Definition 3.6.14]{FOOOBookOne}) only needs the \textit{nonemptiness} of the Maurer-Cartan set, or sufficiently, the existence of a single common root for $Q_i$'s (up to a \textit{non-analytic} transformation $e^{x_i}=y_i$). 
	Roughly speaking, the enhanced strength of proper unobstructedness leads to even more natural and robust behavior.
	But, we also note that the theory of bounding cochains has its own merits, being more general as it applies to classes of any degree in $H^*(L)$ and is particularly relevant for bulk deformations of $A_\infty$ structures, while we hope to adapt our new viewpoint to bulk deformed cases in future studies. In contrast, proper unobstructedness focuses specifically on degree-one classes in $H^1(L)$, a perspective inspired by the SYZ conjecture and mirror symmetry.

	While we can use a bounding cochain $b$ to develop Lagrangian Floer cohomology as in the literature \cite{FOOOBookOne}, the proper unobstructedness invites a possible formulation of a novel version of Lagrangian Floer cohomology, working with the formal power series $Q_i$'s directly instead of the bounding cochain $b$. 
	This viewpoint has been proven essential in developing a quantum-corrected version of a folklore conjecture \cite{Yuan_c_1} (see Sec \ref{ss_related_work}).
	This should also aid in generalizing the family Floer functor \cite{AboFamilyICM,FuFamily,AboFamilyFaithful} into an obstructed version as well. We defer this to future studies \cite{Yuan_affinoid_coeff}.

	\subsubsection{Non-archimedean analytic continuation}

	Apart from the subtle issue of choice-independence, the basic idea is natural.
	Due to Gromov's compactness and Groman-Solomon's reverse isoperimetric inequalities \cite{ReverseI,ReverseII}, the $W$ and $Q_i$'s are actually \textit{strictly convergent} formal power series in the sense that they live in the following (affinoid) algebra with additional adic-convergence conditions:
	\[
	\Lambda\langle H_1(L) \rangle 
	:=
	\left\{
	\sum_{i=0}^\infty c_i Y^{\alpha_i}  \  :  \  c_i\in\Lambda, \  \alpha_i\in H_1(L), \  \mathrm{val}(c_i) + \langle \alpha_i,\gamma\rangle  \to \infty \ \text{for $\gamma$ in a nbhd of $0$ in $H^1(L)$}
	\right\} 
	\]
	(Here we slightly abuse the terms and notations.)
	We can view $\Lambda\langle \pi_1(L)\rangle$ as the \textit{algebra of analytic functions} on an analytic open subdomain $\mathcal U$ in $H^1(L;\Lambda^*)\cong (\Lambda^*)^m$ that contains $H^1(L;U_\Lambda)\cong U_\Lambda^m$, where $\Lambda^*=\Lambda\setminus\{0\}$ and $U_\Lambda$ is the unit circle in $\Lambda$ given by $\mathrm{val}=0$. 
	The vanishings of these formal power series $Q_i$'s on sufficiently many points in $\mathcal U$ can ensure the identical vanishing of $Q_i$'s on the whole $\mathcal U$.
	To offer some intuitive insights into how non-archimedean analytic continuation works, we recall an observation in \cite{Yuan_I_FamilyFloer} below, while it may be not explicitly used in this paper.
	\begin{prop}
		If $f=\sum_{\nu\in\mathbb Z^m} c_\nu Y^\nu
		\in \Lambda [[Y_1^\pm, \dots, Y_m^\pm]]$ is convergent and vanishing on $U_\Lambda^n$, then $f$ is identically zero.
	\end{prop}
	
	\begin{proof}
		Note that $\mathrm{val}(c_\nu)\to\infty$ i.e. $|c_\nu|\to 0$. 
		Arguing by contraction, suppose the sequence $|c_\nu|$ was nonzero, say, it had a maximal value $|c_{\nu_0}|=1$ for some $\nu_0\in\mathbb Z^n$. May further assume $c_{\nu_0}=1$. 
		Then, $|c_\nu|\leqslant 1$ for all $\nu$, so $f\in \Lambda_0[[Y^\pm]]$.
		Modulo the ideal of elements with norm $<1$, we get a power series $\bar f=\sum_{\nu} \bar c_\nu Y^\nu$ over the residue field $\mathbb C$.
		As $c_\nu \to 0$, we have $|c_\nu|<1$ and $\bar c_\nu=0$ for $\nu\gg 1$.
		Hence, this $\bar f$ is just a Laurent polynomial over $\mathbb C$ with $\bar c_{\nu_0}=1$. Meanwhile, the condition also tells that $\bar f(\mathbf {\bar y})$ vanishes for all $\mathbf {\bar y}\in (\mathbb C^*)^n$; thus, $\bar f$ must be identically zero. This is a contradiction.
	\end{proof}
	
	To sum up, Theorem \ref{Main_thm_this_paper} gives a quite broad criterion so that the (proper) unobstructedness condition somehow presents ‘connectedness’ within the moduli space of Lagrangian branes.
	The critical step involves identifying the right effective definition of unobstructedness in Floer theory and demonstrating its invariance.
	The basic idea is intuitively based on an analogy with analytic continuation in complex analysis, which suggests parallels in non-archimedean analysis. Specifically, in proving Theorem \ref{Main_thm_this_paper}, we aim to interpret the (proper) unobstructedness as the vanishing of certain non-archimedean analytic functions. The vanishing of these functions at sufficiently many points implies their uniform vanishing on their convergence domains; meanwhile, Groman-Solomon’s reverse isoperimetric inequalities \cite{ReverseI} ensure that these domains are large enough to deduce the unobstructedness of adjacent Lagrangians up to ``choice-changes'' or ``wall-crossing''. 
	Everything depends on choices in an essential way, and that is where the family Floer technology in \cite{Yuan_I_FamilyFloer} enters the story.

	\subsubsection{Inspiration from family Floer mirror symmetry}
	\label{sss_inspiration_family_floer}

	The proper unobstructedness draws inspiration from the family Floer SYZ mirror construction \cite{Yuan_I_FamilyFloer}. But, notice that we do not work with Lagrangian tori. We can also largely simplify its technical aspects, as this paper does not require establishing \textit{cocycle conditions} for the structure sheaf of a mirror non-archimedean analytic space, but rather is sufficient to seek certain \textit{analytic isomorphisms}.
	Exploring these cocycle conditions is the \textit{core} of \cite{Yuan_I_FamilyFloer} and requires substantial studies for the various 2-homotopy or 2-pseudo-isotopy among two-dimensional families of fibered $A_\infty$ algebras. Notably, it is often false to directly apply the homological perturbation theory to a 2-pseudo-isotopy when there are specified 1-pseudo-isotopy boundary conditions.
	While we greatly benefit from and are deeply grateful for the ideas of Fukaya and Tu \cite{FuBerkeley,Tu,FuCyclic}, we have to indicate that the Maurer-Cartan framework may be impossible to establish the cocycle condition in the category of non-archimedean analytic spaces; cf. \cite[4.9]{Tu}.
	Even if we only seek analytic isomorphisms, the proof of Theorem \ref{Main_thm_this_paper} and the context of Definition \ref{proper_unobstructed_introduction_defn} will also present evidence for the necessity of extending beyond the Maurer-Cartan framework. For instance, the automorphism group of the algebra of analytic functions $\Lambda\langle H_1(L) \rangle$ is a vast rigid structure. An analytic automorphism on the analytic domain $\mathcal U$ can induce a bijection on the set (\ref{Lambda_0_H_1_eq}), but conversely a bijection on this set does not necessarily arise from such an analytic automorphism.
	

	This paper is essentially self-contained, with only the exception of the virtual technique and some homological algebra results derived from our previous work \cite{Yuan_I_FamilyFloer}. The rest of the discussions are basically independent of \cite{Yuan_I_FamilyFloer} and have been significantly streamlined to ensure the text is as concise as possible.
	We sincerely hope that this paper can convey the essential ideas and methodologies in \cite{Yuan_I_FamilyFloer} in a relatively accessible way.
	In fact, the main point of this paper is not complicated techniques but rather a radical new perspective that somewhat challenges traditional viewpoints.

	\subsection{Related work}
	\label{ss_related_work}

	
	The celebrated folklore conjecture by Auroux, Kontsevich, and Seidel, which relates the critical values of a B-side mirror Landau-Ginzburg superpotential to the eigenvalues of quantum multiplication by the first Chern class of the A-side space, has been extensively studied especially in contexts without Maslov-0 holomorphic disks; see e.g. \cite{castronovo2020fukaya, AuTDual,Sheridan16, Ritter_Smith}. Recent development includes the author's proof \cite{Yuan_c_1} of this conjecture with nonconstant Maslov-0 disk quantum corrections, under the assumption of \textit{proper unobstructedness}.
	
	It is interesting to note that the traditional unobstructedness of the existence of a weak bounding cochain is \textit{insufficient} for proving this quantum-corrected version of folklore conjecture in \cite{Yuan_c_1}.	
	Indeed, the limitation arises as the Maurer-Cartan set's invariance under $A_\infty$ algebra homotopy equivalences cannot detect the difference between the Hochschild cohomology rings of the $\check \m$ and $\m$. This is because the ring structure for the Hochschild cohomologies for (curved) $A_\infty$ algebras is not functorial.
	Despite the homotopy equivalence of the curved $A_\infty$ algebras $\check \m$ and $\m$, leading to identical Maurer-Cartan \textit{sets} $\mathcal M(\check \m)=\mathcal M(\m)$, their corresponding Hochschild cohomology \textit{rings} can significantly differ.

	In general, the mirror superpotential can be (non-archimedean) \textit{transcendental} and is well-defined only up to analytic automorphisms, since it counts holomorphic pearly trees and the presence of even a single Maslov-0 disk can generate an infinite number of these pearly trees. 
	Therefore, unlike monotone Lagrangians, the (local) superpotential $W$ is in general no longer a polynomial, e.g. that of Lagrangian fibers over points in the interior of the above-mentioned wall region regarding Proposition \ref{Application_AAK_prop_intro}. This explains the necessity of analytic and transcendental perspectives like Definition \ref{proper_unobstructed_introduction_defn}.

	Finally, we note that the unobstructedness is also known to play a critical role for the Thomas-Yau conjecture. In \cite{thomas2001moment}, Thomas describes a stability condition for Lagrangians, conjectured to be equivalent to the existence of a special Lagrangian in the hamiltonian deformation class of a fixed Lagrangian $L$.
	In \cite[Theorem 4.3]{thomas2002special}, Thomas and Yau relate the conjectural stability to Lagrangian mean curvature flow and argue that there can be at most one special Lagrangian in the hamiltonian deformation class of a graded spin Lagrangian $L$ \textit{provided that $L$ is unobstructed}.
	Now, our Theorem \ref{Main_thm_this_paper} also renders unobstructedness a quite mild assumption for their outcome.
	Therefore, we propose that the proper unobstructedness could be a necessary condition for the hypothetical stability condition of Lagrangians.

	\subsection*{Acknowledgment}
	The author thanks Mohammed Abouzaid, Denis Auroux, Mohammad Farajzadeh-Tehrani, Mark Gross, Kaoru Ono, Daniel Pomerleano, Paul Seidel, and Umut Varolgunes for valuable conversations.
	A key idea of this work was conceived during the 2023 Simons Math Summer Workshop, for which the author expresses gratitude to the SCGP for their hospitality. 
	The author thanks Kenji Fukaya for constant supports and lots of patient explanations for the monograph \cite{FOOO_Kuranishi}, and also thanks Mark McLean for his insightful question during a Symplectic Geometry Seminar in 2019 at Stony Brook, which inspired the investigation of this paper's topic since then.

	\section{$A_\infty$ algebras with topological labels}

	Let $C=\bigoplus_{d\in\mathbb Z} C^d$ be a $\mathbb Z$-graded vector space over $\mathbb R$ with a differential $d_C:C^\bullet \to C^{\bullet+1}$.
	We impose the condition $C^d=0$ for $d<0$ and $d\gg 1$.
	A basic example to consider is the de Rham complex or cohomology of a manifold $N$, denoted as $\Omega^*(N)$ or $H^*(N)$.
	For any element $x\in C^d$, we define its degree as $\deg x=d$. We further define the \textit{shifted degree} by
	$
	\deg'x:=\deg x-1
	$, and we write 
	\begin{equation}
		\label{sign_eq}
		x^\# =\id_\#(x)=(-1)^{\deg'x} x
	\end{equation}
	Besides, we also use
	$\id_{\# s}=\id_\#^s$ to represent the $s$-fold iteration $\id_\#\circ \cdots \circ \id_\#$ of $\id_\#$.
	For $s\in\mathbb N$ and a multi-linear map $\phi$, we write
	$\phi^{\# s}=\phi\circ \id^{s}_\#
	$.
	For $s=1$, we simply have $\phi^\#=\phi^{\# 1}$.
	The degree of $\phi$, denoted as $\deg \phi$, is its degree when considered as a homogeneously-graded $k$-multilinear operator in graded vector spaces. Its shifted degree is given by $\deg' \phi=\deg \phi+k-1$.

	\subsection{Gappedness condition}
	
	We aim to define an $A_\infty$ algebra with labels in an abelian group $\G$ accompanied by two group homomorphisms $E:\G\to \mathbb R$ and $\mu:\G \to 2\mathbb Z$.
	For the purpose of this paper, we always consider the situation $\G=H_2(X,L)$ equipped with the symplectic area $E:H_2(X,L)\to \mathbb R$ and the Maslov index $\mu:H_2(X,L)\to  2\mathbb Z$.
	The unit of $H_2(X,L)$ is denoted as $\beta=0$.
	Let's call such a triple $(H_2(X,L),E,\mu)$ a \textit{label group}.

	Define 
	\begin{equation}
		\label{CC_eq}
		\CC_{k,\beta}(C, C')=
		\Hom( C^{\otimes k}, C')
	\end{equation}
	representing the space of $k$-multilinear operators, with $\beta\in H_2(X,L)$ just serving as an additional label.
	We then define a subspace
	\[
	\CC_{H_2(X,L)}( C, C')  \, \subseteq \prod_{k\in\mathbb N} \prod_{\beta\in  H_2(X,L)} \CC_{k,\beta}( C, C')
	\]
	within the direct product, consisting of the systems $\mathfrak t=(\mathfrak t_{k,\beta})_{k\in\mathbb N,\beta\in H_2(X,L)}$ of multilinear operators $\mathfrak t_{k,\beta}$ with the following \textit{gappedness conditions}:
	\begin{itemize}
		\setlength{\itemsep}{1pt}
		\item [(a)] $\mathfrak t_{0,0}=0$; 
		\item [(b)] if $E(\beta)<0$ or $E(\beta)=0$, $\beta\neq 0$, then $\mathfrak t_\beta:=(\mathfrak t_{k,\beta})_{k\in\mathbb N}$ vanishes identically; 
		\item [(c)] for any $E_0>0$, there are only finitely many $\beta$ such that $\mathfrak t_{\beta}\neq 0$ and $E(\beta)\le E_0$.
	\end{itemize}
	
	\vspace{0.5em}
	
	For clarity, we will simply call such $\mathfrak t$ an \textit{operator system}. We often abbreviate $\CC_{H_2(X,L)}(C,C')$ to $\CC(C,C')$, $\CC_{H_2(X,L)}$, or simply $\CC$ if the context is clear.
	The gappedness condition corresponds to the Gromov's compactness in symplectic topology.

	Next, we define
	\begin{equation}
		\label{composition_Gerstenhaber_eq}
		\begin{aligned}
			(\g\diamond\f)_{k,\beta}=
			\sum_{\ell \ge 1}
			\sum_{k_1+\dots+k_\ell=k}
			\sum_{\beta_0+ \beta_1+\cdots +\beta_\ell=\beta}
			\g_{\ell,\beta_0} 
			\circ ( \f_{k_1,\beta_1}\otimes \cdots \otimes \f_{k_\ell,\beta_\ell} ) \\
			(\g\{ \h\})_{k,\beta} = \sum_{\lambda+\mu+\nu=k}\sum_{\beta'+\beta''=\beta} \g_{\lambda+\mu+1,\beta'} \circ (\id_{\#\deg'\h}^\lambda \otimes \h_{\nu,\beta''}\otimes \id^\mu)
		\end{aligned}
	\end{equation}
	Given $\f,\g, \h\in \CC$, one can check that $\g\diamond \f, \g\{\h\}\in \CC$ still satisfy the gappedness condition.
	Here we adopt Getzler's brace notation \cite{Getzler_1993cartan} for the Gerstenhaber product \cite{gerstenhaber1963cohomology}.
	
	For our purpose, we may assume that the vector space $C$ is one of the following cases:
	\begin{itemize}
		\setlength{\itemsep}{1pt}
		\item $C=H^*(N)$ or $C=\Omega^*(N)$ for a closed submanifold $N$; (we usually take $N=L$)
		\item $C=\bigoplus_{i} \Omega^*(N_i)$ is the direct sum of differential form spaces on multiple closed manifolds $N_i$'s;
		\item $C=\Omega^*(P\times N)$ where $N$ is a closed submanifold $N$ and $P$ is the convex hull of a finite subset in a Euclidean space.
	\end{itemize}

	\begin{defn}
		\label{A_infty_algebra_defn}
		An \textit{$A_\infty$ algebra with labels} in $H_2(X,L)$ is a tuple $(C,\m)$ which consists of an operator system $\m=(\m_{k,\beta})$ contained in $\CC(C,C)$, satisfying that $\deg \m_{k,\beta}=2-\mu(\beta)-k$ as a multi-linear map and the $A_\infty$ associativity relation
		\[\m\{\m\}=0\]
		Such an $A_\infty$ algebra with labels is called \textit{minimal} when $\m_{1,0}=0$.
	\end{defn}

	\begin{defn}
		An \textit{$A_\infty$ homomorphism with labels in $H_2(X,L)$} from $(C',\m')$ to $(C,\m)$ is defined as an operator system $\f=(\f_{k,\beta})$ within $\CC(C',C)$. This satisfies $\deg \f_{k,\beta}=1-\mu(\beta)-k$ and
		\[
		\m\diamond \f=\f\{\m'\}
		\]
	\end{defn}
	
	\begin{defn}
		In the above definitions, note that $\deg \m_{1,0}=1$ and the condition $\m\{\m\}=0$ implies that $\m_{1,0}\circ \m_{1,0}=0$. Hence, we naturally derive a cochain complex $(C, \m_{1,0})$.
		Furthermore, $\f$ naturally induces a cochain map $\f_{1,0}$ from $(C',\m'_{1,0})$ and $(C,\m_{1,0})$.
		When $\f_{1,0}$ is a quasi-isomorphism, $\f$ is called \textit{an $A_\infty$ homotopy equivalence}.
	\end{defn}

	\begin{rmk}
		\label{sign_A_algebra_rmk}
		Some remarks on the signs are as follows.
		Let's write $\mathfrak t_\#=\mathfrak t \diamond \id_\#$ for $\mathfrak t\in\CC_{H_2(X,L)}$. For an $A_\infty$ algebra $\m=(\m_{k,\beta})$ with labels in $H_2(X,L)$, we have $\deg'\m_{k,\beta}=\deg \m_{k,\beta}+k-1=1-\mu(\beta)=1$ (mod 2), as $\mu(\beta)$ is even. Similarly, if $\f$ is an $A_\infty$ homomorphism and $\mu(\beta)$ is even, then $\deg'\f_{k,\beta}= \deg \f_{k,\beta}+k-1=0$ (mod 2); in particular, we have $\id_\#\diamond \f_\#=\f $.
	\end{rmk}

	\begin{defn}
		An $A_\infty$ algebra with labels in $H_2(X,L)$, denoted as $(C,\m)$, is called \textit{unital} when there is a \textit{unit} $\one\in C^0$. For this unit, we require:
		(i) $\m_{1,0}(\one)=0$;
		(ii) $\m_{2,0}(\one, x)=(-1)^{\deg x}\m_{2,0}(x,\one)=x$;
		(iii) $\m_{k,\beta} (\dots, \one,\dots)=0$ for all $(k,\beta)$ except $(1,0)$ and $(2,0)$. Similarly, an $A_\infty$ homomorphism $\f:\m_1\to \m_2$ is called \textit{unital} if $\f_{1,0}(\one_1)=\one_2$, where $\one_1$ and $\one_2$ are units of $\m_1$ and $\m_2$ respectively, and $\f_{k,\beta}(\dots, \one_1,\dots)=0$ when $(k,\beta)\neq (1,0)$.
	\end{defn}

	Abusing the notations, we denote by $\one$ the (equivalence class of) constant-one function in any one of $\Omega^*(L)$, $H^*(L)$, $\Omega^*(L)_P$ and $\HL_P$. It commonly serves as the units of $A_\infty$ structures we encounter.

	\subsection{Retrieval of filtered $A_\infty$ algebras}
	The basic reason for introducing the extra label group $H_2(X,L)$ in \cite{Yuan_I_FamilyFloer} and in this paper is to simultaneously capture the topological data of pseudo-holomorphic disks when studying the $A_\infty$ algebra structures arising from their moduli spaces. 
	For example, this extra topological data is useful for disk invariants \cite{Yuan_e.g._FamilyFloer} and for explicit presentations of integral affine structures in the context of SYZ conjecture \cite{Yuan_local_SYZ,Yuan_conifold,Yuan_A_n}. It is also needed in order to describe the divisor axiom \cite{FuCyclic} and develop a divisor-axiom-preserving homotopy theory of $A_\infty$ algebras \cite{Yuan_I_FamilyFloer}.
	However, the difference is mild. Indeed, by forgetting the classes $\beta$'s and considering only the numbers $E(\beta)$'s, we can retrieve the notion of \textit{filtered $A_\infty$ algebras} in the literature \cite[Definition 3.2.20]{FOOOBookOne}.

	\subsection{Pseudo-isotopies of $A_\infty$ algebras}
	
	The pullback maps for the inclusion maps $L\to \{s\}\times L\subset P \times L$ and the projection map $P\times L\to L$ give rise to
	\begin{equation}
		\label{eval_eq}
		\eval^s: \Omega^*(P\times L)\to \Omega^*(L)
	\end{equation}
	for every $s\in P$ and 
	\[
	\incl: \Omega^*(L) \to \Omega^*(P\times L)
	\]
	Notice that an element in $\Omega^*(P\times L)$ can be written in the form $\sum_{i=1}^n \alpha_i \otimes x_i(s)$, where $\alpha_i\in \Omega^*(P)$ and $x_i(s)\in \Omega^*(L)$ smoothly varies with $s\in P$.
	Namely, there is a natural identification 
	\[
	\Omega^*(L)_P :=\Omega^*(P)\otimes_{C^\infty(P)} C^\infty(P, \Omega^*(L)) \cong \Omega^*(P\times L)
	\]
	One can similarly introduce 
	\[
	H^*(L)_P:= \Omega^*(P) \otimes_{C^\infty(P)} C^\infty(P, H^*(L))
	\]
	By specifying a basis of $H^*(L)$, this is just a direct sum of copies of $\Omega^*(P)$.

	\begin{rmk}
		\label{natural_d_rmk}
		Given the zero differential $d=0$ on $\HL$ or the exterior differential $d$ on $\OL$, there is a natural differential $d_P$ on $\HL_P$ or $\OL_P$ defined by sending $\alpha\otimes x(s)$ to $\alpha\otimes d(x(s))+ \sum (-1)^{|\alpha|} ds_k\wedge \alpha \otimes \partial_{s_k} x(s)$ where $(s_k)$ are local coordinates for $P$.
		Concisely, we may write $d_{P}=1\otimes d+\sum ds_k\otimes \partial_{s_k}$.
	\end{rmk}

	\begin{defn}
		\label{pointwise-defn}
		A \textit{$P$-pseudo-isotopy} on $\HL$ (resp. $\Omega^*(L)$) with labels is defined as a special sort of an $A_\infty$ algebra with labels $\M=(\M_{k,\beta})$ within $\CC(\HL_P,\HL_P)$ (resp. $\CC(\OL_P,\OL_P)$) such that $\M_{1,0}=d_{P}$ and every $\M_{k,\beta}$ is signed $\Omega^*(P)$-linear (see below).
		When $P=\oi$, we will simply call it a \textit{pseudo-isotopy}.
		A multi-linear operator $\M: C_P^{\otimes k} \to C'_P$ is called \textit{$P$-pointwise} or \textit{signed $\Omega^*(P)$-linear} if we have the following \textit{signed linearity condition}: for any $\sigma\in \Omega^*(P)$ we have
		\[
		\M(\eta_1 \otimes x_1,\dots, \sigma\wedge \eta_i\otimes x_i,\dots, \eta_k \otimes x_k ) = (-1)^{\deg \sigma \cdot \left( (\deg \M-1+k)
			+\sum_{a=1}^{i-1} (\deg \eta_a+\deg x_a -1) \right)} \sigma \wedge \M
		(\eta_1\otimes x_1,\dots, \eta_k\otimes x_k)
		\]
	\end{defn}

	A $P$-pseudo-isotopy $(C_P,\M)$ can be approximately viewed as a $P$-family of $A_\infty$ algebras together with certain extra data of derivatives of orders $\leqslant \dim P=:m$. More precisely, the $P$-pointwise condition allows us to express $\M$ as 
	\begin{equation}
		\label{1otimes_s_eq}
		\M=1\otimes \m^s+ \sum_{\varnothing \neq I \subseteq \{1,2,\dots, m\}} ds_I\otimes \mc^{I,s}
	\end{equation}
	where each $\m^s$ defines an $A_\infty$ algebra on $C$ and $ds_I=ds_{i_1}\wedge \cdots \wedge ds_{i_k}$ for $I=\{i_1,\dots, i_k\}$.
	The additional collections of multilinear maps \(\mc^{I,s}\) can be intuitively viewed as non-commutative derivatives of \(\m^s\), capturing its variation as $s$ moves in the parameter space $P$. We also remark that an $A_\infty$ algebra $(C,\m)$ gives rise to a {trivial $P$-pseudo-isotopy} $\M^\tri_P$ on $C_P$ by setting all $\m^s=\m$ and $\mc^{I,s}=0$.

	\begin{defn}
		\label{restriction_pseudo_isotopy_defn}
		Every $(C, \m^s)$ is an $A_\infty$ algebra with labels for each $s\in P$. Then, we call $(C,\m^s)$ the \textit{restriction} of $(C_P, \M)$ at $s\in P$, and we say that the $\M^P$ \textit{restricts} to the $\m^s$ at $s\in P$.
	\end{defn}
	
	The following result is due to Fukaya \cite{FuCyclic}. See also \cite{Yuan_I_FamilyFloer} and \cite[21.25 \& 21.29]{FOOO_Kuranishi}.
	
	\begin{lem}
		\label{pseudo_isotopy_pointwise_lem}
		When $P=\oi$, we can write $\M=1\otimes \m^s+ ds\otimes \mc^s$. Then, the $A_\infty$ associativity relation of $\M$ is equivalent to the following conditions:
		
		\begin{itemize}
			\item The $(C,\m^s)$ is an $A_\infty$ algebra for any $s$, namely, $\m^s\{\m^s\}=0$.
			\item  $\m^s_{1,0}$ is independent of $s$ and $\mc^s_{1,0}=\frac{d}{ds}$.
			\item The following relation holds
			\[
			\hspace{-2em}
			\frac{d}{ds}\m^s_{k,\beta} 
			+ 
			\sum_{\substack{i+j+\ell=k }} 
			\sum_{\substack{ \beta_1+\beta_2=\beta\\
					(i+j+1,\beta_1)\neq (1,0)}}
			\mc^s_{i+j+1,\beta_1}  (\id_\#^i \otimes \m^s_{\ell,\beta_2} \otimes \id^j) 
			\ \ -
			\sum_{\substack{i+j+\ell=k }} 
			\sum_{\substack{ \beta_1+\beta_2=\beta \\
					(\ell,\beta_2)\neq(1,0)}}
			\m^s_{i+j+1, \beta_1} 
			(\id^i \otimes 
			\mc^s_{\ell,\beta_2} \otimes \id^j) =0
			\]
			More concisely, this means 
			\[
			\mc^s\{\m^s\}-\m^s\{\mc^s\}=0
			\]
			Actually, for the reduced form $\tilde \mc^s= \mc^s-\mc_{1,0}^s$ removing the $\CC_{1,0}$-component, we have\footnote{This is because $\frac{d}{ds}\m^s=\frac{d}{ds}\circ \m^s- \sum \m^s (\id^\bullet\otimes \frac{d}{ds}\otimes \id^\bullet)=\frac{d}{ds} \{\m^s\} -\m^s\{\frac{d}{ds}\} $.}
			\[
			\tfrac{d}{ds} \m^s + \tilde \mc^s \{\m^s\} -\m^s\{\tilde \mc^s\} =0
			\]
		\end{itemize}
	\end{lem}
	
	\begin{proof}
		By Definition \ref{A_infty_algebra_defn}, the $A_\infty$ relation is read $\M\{\M\}=0$, or equivalently,
		\[
		\big(1\otimes \m^s +ds\otimes \mc^s \big) \{1\otimes \m^s +ds\otimes \mc^s\}=0
		\]
		For the degree on $\Omega^{\bullet}(\oi)$, either $\bullet=$ 0 or 1, the above equation in $\CC(C_\oi)$ is equivalent to the following two equations in $\CC(C)$:
		$
		\m^s\{\m^s\}=0 $ and
		$		\mc^s \{\m^s\}-\m^s\{\mc^s\}$.
		The negative sign on the second equation comes from Definition \ref{pointwise-defn}, as $\mc^s$ is labeled with the degree-one term $ds$ in $\Omega^*(\oi)$.
	\end{proof}
	
	\begin{rmk}
		\label{sign_pseudo_rmk}
		Be cautious that $\mc^s\{\m^s\}=\mc^s(\id_\#^\bullet \otimes \m^s\otimes \id^\bullet)$ while $\m^s\{\mc^s\}=\m^s(\id^\bullet \otimes \mc^s\otimes \id^\bullet)$ has different signs, since the degrees of $\m^s$ and $\mc^s$ differ by $1$.
	\end{rmk}
	
	\begin{rmk}
		Integrating the data of the ``derivatives'' $\mc^s$ in a Feymann diagram type connects $\m^0$ with $\m^1$ via an $A_\infty$ homotopy equivalence.
		We will study this later in Theorem \ref{from_pseudo_to_A_homo_thm}.
		Accordingly, we may often intuitively regard the $\mc^{I,s}$ in (\ref{1otimes_s_eq}) as the data of derivatives or as the ``$\partial_{I}$'' of $\m^s$.
	\end{rmk}
	
	\subsection{Category $\UD$}
	\label{sss_UD_category}
	
	Let $L$ be a Lagrangian submanifold in a symplectic manifold $(X,\omega)$ as before.
	We have a natural group homomorphism
	\begin{equation}
		\label{partial_pi_2_X_L_eq}
		\partial: H_2(X,L) \to    H_1(L)
	\end{equation}

	For our purpose, we may always assume that the base vector space $C$ is in the form of either $\HL_P$ or $\OL_P$ for the fixed $L$, while the label group is fixed to be $H_2(X,L)$.
	In these situations, there is a natural differential $d=d_C$ as discussed in Remark \ref{natural_d_rmk}.
	Let $Z^1(C)$ denote the kernel of $d_C$ restricted on the total degree-one part $C^1$ of $C$.

	For any $b\in Z^1(C)$ and $\sigma\in H_1 (L)$, we define $\sigma\cap b$ as follows.
	First, we may write 
	\[
	b=1\otimes \underline b(s)+\sum_i ds_i\otimes  b_i(s)
	\]
	where $s=(s_i)$ are local coordinates of $P$ and $\underline b(s), b_i(s)\in \HL$ or $\OL$.
	If $C=\HL_P$, then the condition $b\in Z^1(C)$ implies that $\partial_{s_i} \underline b(s)=0$, namely, $\underline b(s)\in H^*(L)$ is independent of $s$. Thus, we may define $\sigma\cap b:=\sigma \cap \underline b(s)$.
	If $C=\OL_P$, then $d(\underline b(s))=0$ and $\partial_{s_k} \underline b(s) -d_L(b_k(s))=0$. It follows that the de Rham cohomology class $\mathfrak b$ of $\underline b(s)$ is independent of $s$, so we may define 
	$
	\langle \sigma,  b\rangle :=\sigma \cap \mathfrak b
	$.
	
	\begin{defn}
		An operator system $\mathfrak t = (\mathfrak t_{k,\beta})$ in $\CC(C, C')$ is said to satisfy the \textit{divisor axiom} if for any $b\in Z^1(C)$ and $(k,\beta)\neq (0,0)$, we have
		\[
		\sum_{i=1}^{k+1}  \mathfrak t_{k+1,\beta}(x_1,\dots, x_{i-1}, b, x_i,\dots ,x_k) =
		\langle \partial \beta , b\rangle \cdot \mathfrak t_{k,\beta}(x_1,\dots,x_k)
		\]
	\end{defn}
	
	\begin{rmk}
		In general, the divisor axiom imposes a nontrivial constraint on the structure of an $A_\infty$ algebra. Geometrically, the Gromov-Witten divisor axiom (see \cite{kontsevich1994gromov}) is expected to have an open-string analog. This is established by Fukaya in \cite{FuCyclic} based on the specific features of the Kuranishi structure and virtual fundamental chain theory developed by Fukaya-Oh-Ohta-Ono \cite{FOOO_Kuranishi}. The construction uses moduli spaces of virtual dimensions beyond just 0 or 1. It requires considering not a single moduli space but multiple moduli spaces, ensuring that the virtual techniques applied to them are compatible with respect to the forgetful maps. At the time of writing, it seems that other virtual techniques may not yet be applicable to handle the necessary complexity required to establish the desired constraint.
	\end{rmk}

	\begin{defn}
		An operator system $\mathfrak t=(\mathfrak t_{k,\beta})$ is said to be \textit{cyclically unital} if, for any degree-zero $\e\in C^0$ and $(k,\beta)\neq (0,0)$, we have
		\[
		\sum_{i=1}^{k+1}\mathfrak t_{k+1,\beta} (x_1^\#,\dots, x_{i-1}^\#, \e, x_i, \dots, x_k) =0
		\]
	\end{defn}

	\begin{prop}
		\label{UD_prop}
		\emph{There is a category $\UD\equiv \UD(L)$ with the following:}
		\begin{enumerate}
			\setlength{\itemsep}{0em}
			\item[(I)] \emph{An object $(C,\m)$ in $\UD$ is an $A_\infty$ algebra with labels such that (I-1) it is a $P$-pseudo-isotopy for $C=\HL_P$ or $\OL_P$, and $\m_{1,0}$ is equal to the natural differential; (I-2) it is $P$-unital; (I-3) it is cyclically unital; (I-4) it satisfies the divisor axiom; (I-5) if $\m_{k,\beta}\neq 0$, then $\mu(\beta)\ge 0$.}
			
			\item[(II)] \emph{A morphism $\f$ in ${\UD}$ is an $A_\infty$ homomorphism with labels such that
				(II-1) it is unital with respect to the various $\one$; (II-2) it is cyclically unital; (II-3) it satisfies the divisor axiom; (II-4) for any $b\in Z^1(C)$, we have $\langle \partial \beta , \f_{1,0} (b) \rangle  = \langle \partial \beta , b\rangle $; (II-5) if $\f_{k,\beta}\neq 0$, then $\mu(\beta)\ge 0$.}
		\end{enumerate}
	\end{prop}
	
	\begin{proof}
		The above data forming a category just means that the properties listed above is closed under the composition of $A_\infty$ homorphisms.
		They are all straightforward to verify; see \cite{Yuan_I_FamilyFloer} for the full details. For example, if $\f$ and $\g$ are two $A_\infty$ homomorphisms, then it is well-known that the compoistion $\g\diamond \f$ is also an $A_\infty$ homomorphism.
		For another, if both $\f$ and $\g$ satisfy the divisor axiom, then one can easily check that $\g\diamond \f$ also satisfies the divisor axiom.
	\end{proof}

	
	We will refer to the conditions (I-5) and (II-5) as the semipositive conditions. Geometrically, the $\mu$ stands for the Maslov index, and the condition is valid for $A_\infty$ algebras associated to special or graded Lagrangian submanifolds \cite[Lemma 3.1]{AuTDual}.
	
	\begin{convention}
		\textit{Henceforth, all $A_\infty$ algebras and $A_\infty$ homomorphisms discussed in the following sections should belong to the category $\UD$, unless explicitly stated otherwise.}
	\end{convention}

	There is the notion of {ud-homotopy} among morphisms in $\UD$ \cite{Yuan_I_FamilyFloer} in the following sense. We say two morphisms $\f_0,\f_1$ from $(C',\m')$ to $(C,\m)$ in $\UD$ are \textit{ud-homotopic}, writing 
	\begin{equation}
		\label{ud_sim_eq}
		\f_0\simud\f_1
	\end{equation}
	if there is a morphism $\F$ from $(C',\m')$ to $(C_\oi, \M^{\tri})$ in $\UD$ such that $\eval^0 \F=\f_0$ and $\eval^1\F=\f_1$, where $\M^\tri$ is the trivial pseudo-isotopy induced by $\m$.
	Equivalently, $\f_0\simud\f_1$ if and only if there exists $\f_s,\h_s$ in $\CC(C',C)$, varying smoothly with $s\in\oi$, such that
	
	\begin{itemize}
		\setlength{\itemsep}{0.3em}
		\item[(a)] Every $\f_s$ is a morphism in $\UD$ from $(C',\m')$ to $(C,\m)$;
		\item[(b)] $ \frac{d}{ds} \circ \f_s = \sum \h_s\circ (\id_\#^\bullet\otimes \m \otimes \id^\bullet) + \sum \m'\circ (\f_s^\#\otimes \cdots \otimes \f_s^\#\otimes  \h_s\otimes \f_s\otimes \cdots\otimes \f_s)$;
		\item[(c)] The $\h_s$ satisfies the divisor axiom, the cyclical unitality, and $(\h_s)_{k,\beta}(\cdots \one \cdots )=0$ for all $(k,\beta)$;
		\item[(d)] $\deg (\h_s)_{k,\beta}= -k-\mu(\beta)$. For every $\beta$ with $\h_\beta\neq 0$, we have $\mu(\beta)\ge 0$.
	\end{itemize}

	\begin{prop}[Whitehead theorem]
		\label{whitehead_prop}
		Fix $\f\in \Hom_\UD ((C',\m'), (C,\m))$ such that $\f_{1,0}$ is a quasi-isomorphism of cochain complexes. Then, there exists $\g\in \Hom_\UD ( (C,\m) , (C',\m'))$, unique up to ud-homotopy, such that $\g\diamond \f\simud \id_{C'}$ and $\f\diamond \g\simud \id_C$. We call $\g$ a \emph{ud-homotopy inverse} of $\f$.
	\end{prop}
	
	\begin{proof}[Sketch of proof]
		This is entirely a result of homological algebra, and the complete details can be found in \cite[Theorem 2.34]{Yuan_I_FamilyFloer}. Besides, without requiring the divisor axiom, it is also proved in \cite[Theorem 4.2.45]{FOOOBookOne} as a natural generalization of the classic Whitehead theorem in algebraic topology. First, the assumption ensures $\g_{1,0}$'s existence, and we can construct other $\g_{k,\beta}$'s through induction. The key new challenge, unlike in \cite[Theorem 4.2.45]{FOOOBookOne}, is maintaining the divisor axiom during this induction. 
		The divisor axiom alone is insufficient, as it leads to a breakdown of the induction process.
		To address this issue, the category $\UD$ is introduced to encode the extra conditions necessary for a successful induction.
		Specifically, when constructing the homotopy inverse $\g$ inductively, we follow a process similar to that in \cite[Theorem 4.2.45]{FOOOBookOne}, but we strengthen the induction hypothesis by incorporating the conditions outlined in Proposition \ref{UD_prop} (II).
		With this strengthened hypothesis in place, the remaining steps of the induction process become relatively straightforward.
	\end{proof}

	\subsection{Homological perturbation and harmonic contractions}
	
	Every $A_\infty$ algebra $\check \m$ is also accompanied by a minimal $A_\infty$ algebra $\m$, called the \textit{minimal model} (also called the canonical model \cite{FOOOBookOne}). It is determined by the method of \textit{homological perturbation}. As mentioned in the introduction, the intuition is that if we view $\check \m$ as counts of pseudo-holomorphic disks bounded by a Lagrangian, then $\m$ counts `holomorphic pearly trees', i.e. a cluster of holomorphic disks arranged in a tree. (See the figure below)
	
	\begin{figure}[h]
		\centering
		\vspace{-0.5em}
		\includegraphics[scale=0.32]{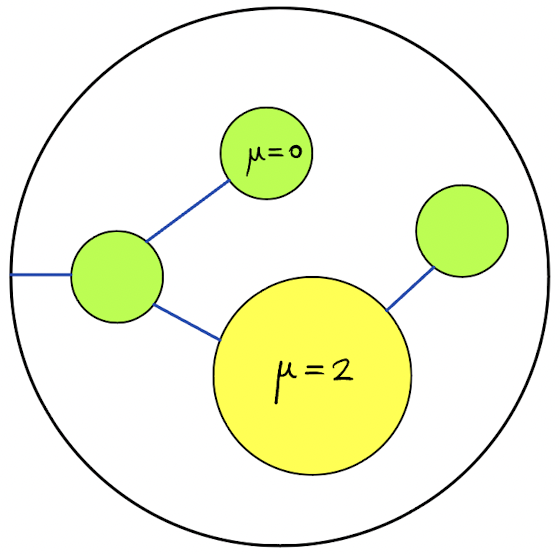}
		\vspace{-0.5em}
	\end{figure}
	
	We begin with the following concepts; see also \cite{FuCyclic,FOOO_Kuranishi,Yuan_I_FamilyFloer}.
	
	A \textit{ribbon tree} is a tree $T$ with an embedding $T \xhookrightarrow{} \mathbb D\subset \mathbb C$ such that a vertex $v$ has only one edge if and only if $v$ lies in the unit circle $\partial \mathbb D$.
	Such a vertex is called an exterior vertex, and any other vertex is called an interior vertex. The set of all exterior (resp. interior) vertices is denoted by $\Cext_0(T)$ (resp. $\Cint_0(T)$). Then, $C_0(T)=\Cext_0(T)\cup \Cint_0(T)$ is the set of all vertices.
	Besides, an edge of $T$ is called exterior if it contains an exterior vertex and is called interior otherwise. The set of all exterior edges is denoted by $\Cext_1(T)$ and that of all interior edges is denoted by $\Cint_1(T)$.
	
	A \textit{rooted ribbon tree} is a pair $(T,v_{ou})$ of a ribbon tree $T$ and an exterior vertex $v_{ou}$ therein. We call $v_{ou}$ the \textit{root}.
	By a \textit{decoration} on $(T,v_{ou})$, we mean a map $\beta(\bullet): \Cint_0(T)\to H_2(X,L)$. It is called \textit{stable} if any $v\in \Cint_0(T)$ satisfies $E(\beta(v))\ge 0$ and has at least three edges whenever $\beta(v)=0$.
	Given $k\in\mathbb N$ and $\beta$, we denote by $\Tr(k,\beta)$ the set of such decorated \textit{stable rooted ribbon trees} $(T,v_{ou},\beta(\bullet))$ so that $\#\Cext_0(T)=k+1$ and $\sum_{v\in \Cint_0(T)} \beta(v)=\beta$.
	For simplicity, we often abbreviate $T=(T,v_{ou},\beta(\bullet) )$.

	There is a natural partial order on the set of interior vertices by requiring $v<v'$ if $v\neq v'$ and there is a path of edges in $T$ from $v$ to $v_{ou}$ that passes through $v'$.
	A \textit{time allocation} of $T=(T,v_{ou},\beta(\bullet) )$ in $\Tr(k, \beta)$ is a map $\tau:\Cint_0(T)\to \mathbb R$ such that $\tau(v)\leqslant \tau(v')$ whenever $v<v'$.
	The set of all time allocations $\tau$ for $T$ such that all $\tau (v)$ are contained in $[a,b]$ is denoted by $\mathscr S_a^b(T)$.
	Similarly, a \textit{time allocation} $\tau$ of an element $(T, \mathscr B)$ in $\mathscr T(k_0,k_1,\B)$ is a rule that assigns to each interior vertex $v$ a real number $\tau(v)$ such that $v<v'$ implies that $\tau(v)\leqslant \tau(v')$.
	The set of all time allocations $\tau$ such that all $\tau(v)$ are contained in $[a,b]$ is still denoted by $\mathscr S_a^b(T)$. We also define $\mathscr S(T)=\mathscr S_0^1(T)$.

	Remark that any time allocation $\tau$ can be viewed as a point $(x_v=\tau(v))_{v\in \Cint_0(T)}$ in $[a,b]^{\#\Cint_0(T)}$. Thus, the set $\mathscr S_a^b(T)$ can be identified with a bounded polyhedron cut out by the inequalities $x_v\le x_{v'}$ for $v<v'$ and $a\le x_v\le b$. Hence, it has a natural measure induced from the Lebesgue measure.

	\begin{defn}
		\label{contraction-defn}
		Let $(C, \check \m_{1,0})$ and $(\mH,\delta)$ be two graded cochain complexes.
		A triple $(i,\pi,G)$, consisting of two maps $i: H \to C$, $\pi: C \to  H$ of degree $0$ and a map $G: C \to C$ of degree $-1$, is called a \textit{contraction} (for $H$ and $C$) if the following equations hold
		\begin{equation}\label{cochain-maps-i-pi}
			\check \m_{1,0}\circ i = i \circ \delta  \qquad
			\pi \circ\check  \m_{1,0} = \delta \circ \pi 
		\end{equation}
		\begin{equation}\label{green-operator}
			i\circ \pi-\id_{C} = \check \m_{1,0} \circ G + G\circ \check \m_{1,0}
		\end{equation}
		Further, the $(i,\pi,G)$ is called a \textit{strong contraction}, or say it is \textit{strong}, if we have the extra conditions
		$\pi\circ i - \id_{\mH}=0$, $G\circ G=0$, $G\circ i =0$, and $\pi\circ G=0$.
	\end{defn}

	We can extend the result of \cite[Theorem 5.4.2]{FOOOBookOne} incorporating the topological labels:

	\begin{thm} [Homological perturbation]
		\label{canonical_model_general_tree-thm}
		Fix an $A_\infty$ algebra $(C,\check \m)$ with labels and a graded cochain complex $(H,\delta)$.
		From a contraction $g=(i,\pi, G)$,
		there is a canonical way to construct an $A_\infty$ algebra $(H, \m)$ and an $A_\infty$ homotopy equivalence
		\[
		\mi: (H,\m)\to (C,\check  \m)
		\]
		so that 
		$
		\mi_{1,0}=i$ and $\m_{1,0}=\delta$. We call $(H,\m)$ the \emph{minimal model $A_\infty$ algebra} of $(C,\check \m)$ and $g$.
	\end{thm}
	
	\begin{proof}[Proof]
		Let $T=(T,v_{ou},\beta(\bullet) )$ be a rooted ribbon tree equipped with a $H_2(X,L)$-decoration.
		By induction on $\#\Cint_0(T)$, we construct operators $\mathfrak i_T: H^{\otimes \bullet}\to C$ and $\m_T: H^{\otimes \bullet}\to H$ as follows:
		When $\#\Cint_0(T)=0$, we set $\mi_T=i$ and $\m_T=\delta$. When $\#\Cint_0(T)=1$, we set $\mi_T=G\circ \check \m_{k,\beta} (i\otimes \cdots \otimes i)$ and $\m_T=\pi\circ \check \m_{k,\beta} (i \otimes \cdots \otimes i)$.
		Let $e$ denote the edge of the root $v_{ou}$ in $T$, and let $v$ be the other vertex of $e$. Cut all edges of $v$ except $e$, and add a pair of vertices to all resulting half lines. This produces a sequence of subtrees $T_1,\dots, T_\ell$, ordered counterclockwise. Then, we can define $\mi_T=G\circ \check \m_{\ell,\beta(v)} (\mi_{T_1} \otimes \cdots \otimes \mi_{T_\ell})$ and $\m_T=\pi \circ \check \m_{\ell,\beta(v)} (\mi_{T_1} \otimes \cdots \otimes \mi_{T_\ell})$ by induction.
		Finally, we define $\mi_{k,\beta}=\sum_{T\in\mathscr T(k,\beta)} \mi_T$ and $\m_{k,\beta}=\sum_{T\in\mathscr T(k,\beta)} \m_T$. In particular, we remark that the following induction formula holds:
		\begin{equation}
			\label{induction-tree-formula-eq}
			\begin{aligned}
				\mi_{k,\beta} &= 
				\sum_{ (\ell,\beta_0)\neq (1,0)}
				G \circ \check \m_{\ell,\beta_0} \circ (\mi_{k_1,\beta_1}\otimes \cdots \otimes \mi_{k_\ell,\beta_\ell}) \\
				\m_{k,\beta} &= 
				\sum_{ (\ell,\beta_0)\neq (1,0)}
				\pi \circ \check  \m_{\ell,\beta_0} \circ (\mi_{k_1,\beta_1}\otimes \cdots \otimes \mi_{k_\ell,\beta_\ell})
			\end{aligned}
		\end{equation}
		
		We can check that they satisfy the desired $A_\infty$ relations.
		More details can be found in \cite[Theorem 2.52]{Yuan_I_FamilyFloer}; when $\delta=0$, the proof is also established in \cite{FuCyclic}.
	\end{proof}

	In practice, we mainly study the following two specific contractions:
	
	\begin{itemize}
		\setlength{\itemsep}{1pt}
		\item 
		Consider two graded cochain complexes $(\Omega^*(L), d)$ and $(H^*(L), d=0)$. Assume $L$ is a closed or complete Lagrangian submanifold.
		Given a metric $g$ on $L$, the Hodge decomposition expresses $\Omega^*(L)$ as the direct sum of the space $\mathcal H_g$ of harmonic differential forms, and the images of the exterior derivative $d$ and its adjoint $d^*$. Identifying $\mathcal H_g$ with the cohomology $H^*(L)$, we define $i = i(g)$ as the inclusion map from $\mathcal H_g$ to $\Omega^*(L)$, $\pi = \pi(g)$ as the projection map from $\Omega^*(L)$ to $\mathcal H_g$, and $G = G(g)$ as the Green's operator associated with the metric $g$. One can verify that they form a strong contraction \cite{Yuan_I_FamilyFloer}.
		Abusing the notations, we write $g=(i(g),\pi(g),G(g))$.
		
		\item Consider two graded cochain complexes $(\Omega^*(L)_\oi, d_{L\times \oi})$ and $(\HL_\oi, d=d_{[0,1]})$.
		Given a family of metrics $\pmb g=(g_s)_{0\leqslant s \leqslant 1}$ on $L$, let $(i_s,\pi_s,G_s)=(i(g_s),\pi(g_s), G(g_s))$. One can find operators $h_s$, $k_s$, and $\sigma_s$ such that $i(\pmb g)= 1\otimes i_s +ds\otimes h_s$, $\pi(\pmb g)= 1\otimes \pi_s +ds\otimes k_s$, and $G(\pmb g)=1\otimes G_s+ds\otimes \sigma_s$ form a strong contraction as well. See \cite[Lemma 3.6]{Yuan_I_FamilyFloer}.
	\end{itemize}
	
	The following results are straightforward from the definitions; see also \cite{Yuan_I_FamilyFloer}.
	
	\begin{prop}
		\label{minimal_model_prop}
		Given an object $(\OL, \check \m)$ in the category $\UD$, utilizing Theorem \ref{canonical_model_general_tree-thm} with the strong contraction $g$ transforms $\check \m$ into another object $(\HL, \m)$ in $\UD$, and simultaneously yields $\mi:\m\to\check \m$, an $A_\infty$ homotopy equivalence in $\UD$.
	\end{prop}
	
	\begin{proof}
		Our goal is to show that the $\m$ and the $\mi$ obtained through the homological perturbation actually belong to the category $\UD$. This can be verified inductively using the induction formula (\ref{induction-tree-formula-eq}) of the homological perturbation process.
		See also \cite[Proposition 2.55 and Theorem 3.3]{Yuan_I_FamilyFloer}. 
	\end{proof}
	
	\begin{prop} 
		\label{pseudo_isotopy_minimal_model_prop}
		Likewise, for $(\OL_\oi, \check \M)$ in $\UD$, applying Theorem \ref{canonical_model_general_tree-thm} to $\check \M$ with the strong contraction $\pmb g$ results in a pseudo-isotopy $(\HL_\oi, \M)$ in $\UD$, and simultaneously yields an $A_\infty$ homotopy equivalence $\mI: \M \to \check \M$ in $\UD$.
		Moreover, by the pointwise properties, we can write $\check \M=1\otimes \check \m_s+ds\otimes \check \mc_s$, $\M=1\otimes \m_s+ds\otimes \mc_s$, and $\mI=1\otimes \mi_s+ds\otimes \mathfrak e_s$; then for each $s$, the situation of Proposition \ref{minimal_model_prop} exactly applies to $\mi_s$, $\check \m_s$, and $\m_s$.
		In particular, we have $\mi_s\diamond \eval^s =\eval^s \diamond \mI$ and $\m_s\diamond \eval^s =\eval^s \diamond \M$.
	\end{prop}
	
	\begin{proof}
		This is basically a pseudo-isotopy version of Proposition \ref{minimal_model_prop} above. The process is similar since the contraction $\pmb g$ is pointwise. See \cite[Lemma 3.8 and Theorem 3.9]{Yuan_I_FamilyFloer}.
	\end{proof}
	
	\begin{rmk}
		Analogous results do not extend to 2-dimensional pseudo-isotopy. Indeed, a 2-family of metrics $g_{s,t}$ unavoidably encounters "curvature terms" like $\partial_s \partial_t G_{s,t} - \partial_t \partial_s G_{s,t}$, which prevent us from deriving strong contractions (Definition \ref{contraction-defn}). As mentioned in the introduction, this issue is irrelevant in this paper since we do not need to prove the cocycle conditions for the structure sheaf of a mirror non-archimedean analytic space. Meanwhile, this issue explains why the cocycle condition cannot be obtained by naively applying the homological perturbation to a 2-pseudo-isotopy.
	\end{rmk}

	\subsection{Feymann-diagram-type operator integral}

	The following result is due to \cite{FuCyclic}; see also \cite{Yuan_I_FamilyFloer}.

	\begin{thm}
		\label{from_pseudo_to_A_homo_thm}
		There is a canonical way to assign to a pseudo-isotopy $( C_\ab, \M)$ an $A_\infty$ homomorphism 
		\[
		\mC=\mC^\ab: (C, \m^a ) \to (C, \m^b)
		\]
		where $\mC_{1,0}=\id_C$.
	\end{thm}
	
	Indeed, if we write $\M=1\otimes \m^s+ds\otimes \mc^s$, then we have the following inductive formula:
	\begin{equation}
		\label{inductive_formu-Fubini-yield-eq}
		\mC_{k,\beta}^\ab =
		\sum_{\ell\ge 1}
		\sum_{ \substack{ 
				\beta_0+\beta_1+\cdots+ \beta_\ell = \beta\\
				k_1+\cdots+k_\ell = k \\
				(\ell,\beta_0)\neq(1,0)}
		}
		-\int_a^b du \cdot 
		\mc^u_{\ell,\beta_0}
		\circ
		\big(
		\mC_{k_1,\beta_1}^{[a,u]} \otimes \cdots \otimes \mC_{k_\ell,\beta_\ell}^{[a,u]}
		\big)
	\end{equation}
	whenever $(k,\beta)\neq(0,0),(1,0)$. See \cite[Theorem 2.56]{Yuan_I_FamilyFloer}.

	\begin{prop}
		\label{integral_UD_prop}
		Assume a pseudo-isotopy $(C_\oi, \M)$ is an object in $\UD$. Then, its restriction $A_\infty$ algebra $\m^s$ at each $s\in \oi$ is also an object in $\UD$ (Definition \ref{restriction_pseudo_isotopy_defn}).
		Moreover, if we let $\mC: \m^0 \to \m^1$ denote the induced $A_\infty$ homotopy equivalence (Theorem \ref{from_pseudo_to_A_homo_thm}), then $\mC$ is a morphism in $\UD$.
	\end{prop}
	
	\begin{proof}
		Using the inductive formula (\ref{inductive_formu-Fubini-yield-eq}), this is straightforward from the definition.
	\end{proof}

	Remark that the relation (\ref{inductive_formu-Fubini-yield-eq}) can be also concisely interpreted as
	\begin{equation}
		\label{du_C_eq}
		\tfrac{d}{du} \mC^{[a,u]} = -\tilde\mc^u \diamond \mC^{[a,u]} \qquad (i=0,1)
	\end{equation}
	Here $\tilde \mc^u$ is the reduced form of $\mc^u$, removing the $\CC_{1,0}$-component of $\mc^u$, as introduced in Lemma \ref{pseudo_isotopy_pointwise_lem}.

	\subsection{$A_\infty$ algebras associated to Lagrangians}
	\label{s_Axioms}
	
	We review the $A_\infty$ algebras associated to Lagrangians.
	Our approach is based on the latest published monograph \cite{FOOO_Kuranishi} of Fukaya-Oh-Ohta-Ono, regarding the de Rham model as opposed to their earlier singular chain model in \cite{FOOOBookOne,FOOOBookTwo}. Note that the de Rham model is also extensively studied by Solomon and Tukachinsky \cite{solomon2016differential} \cite[Section 2]{solomon2016point}.

	\begin{defn}
		Let $\mathfrak J(X)$ be the space of smooth almost complex structures tamed by $\omega$.
		Fix $J\in\mathfrak J(X)$. Given $\beta\in H_2(X,L)$ and $k\in\mathbb N$ with $(k,\beta)\neq (0,0), (1,0)$, we denote by 
		\begin{equation}
			\label{moduli_single}
			\mathcal M_{k+1,\beta}(J,L)
		\end{equation}
		the moduli space of all equivalence classes $[\Sigma, \mathbf u, \mathbf z]$ of \textit{$(k+1)$-boundary-marked $J$-holomorphic stable maps of genus zero with one disk component bounded by $L$ in the class $\beta$}. 
		Here the $\Sigma$ is a nodal Riemann surface with one boundary component $\partial \Sigma$, the $\mathbf u: (\Sigma, \partial\Sigma)\to (X,L)$ is continuous and $J$-holomorphic restricted on each irreducible component, the $\mathbf z=(z_0,z_1,\dots,z_k)$ represents the boundary marked points ordered counter-clockwisely. We also require $\mathbf u$ is stable; see e.g. \cite{MS,Frau08,FOOOBookOne,FOOODiskOne}.
		The equivalence relation refers to a biholomorphism on the domains of two stable maps which identifies the nodal points, the marked points and the boundaries. The moduli space $\mathcal M_{k+1,\beta}(J,L)$ is first a well-defined \textit{set} and is also a compact Hausdorff \textit{topoloigcal space}, equipped with the stable map topology \cite[Theorem 7.1.43]{FOOOBookTwo}.
		Slightly abusing the notations, let $ J=\{J_t\mid t\in P\}$ denote a smooth family of $\omega$-tame almost complex structures parameterized by $P$, we consider the following union of moduli spaces:
		\begin{equation}
			\label{moduli_space_P}
			\textstyle
			\mathcal M_{k+1,\beta}(  J,L):=\bigsqcup_{t\in P} \ \{t\} \times \mathcal M_{k+1,\beta}(J_t,L)
		\end{equation}
		which is called a parameterized moduli space. Again, it is well-defined as a set or a topological space.
		Finally, we denote by 
		\begin{equation}
			\label{ev_eq}
			\ev_i:\mathcal M_{k+1,\beta}(J,L) \to L \qquad i=0,1,2,\dots, k
		\end{equation}
		the \textit{evaluation map} at the $i$-th marked point $z_i$
	\end{defn}
	
	\begin{defn}
		\label{moduli_beta_0_defn}
		For the exceptional case $\beta=0$ in $H_2(X,L)$, we define $\mathcal M_{k+1,0}(J,L)=L\times \mathbb D^{k-2}$ or $P\times L\times \mathbb D^{k-2}$ for $k\geqslant 2$. We do \textit{not} define $\mathcal M_{k+1,0}(J,L)$ for $k=0,1$. See \cite[21.6(V)]{FOOO_Kuranishi}.
	\end{defn}
	
	\begin{defn}
		We call the collection 
		\begin{equation}
			\label{moduli_space_system_eq}
			\mathbb M(J,L) =\Big\{ \mathcal M_{k+1,\beta}(J,  L) \mid (k,\beta)\in \mathbb N\times  H_2(X,L)	\Big\}
		\end{equation}
		the \textit{moduli space system} associated with $(J,L)$. Here $J$ can be a single element or a family in $\mathfrak J(X)$ according to the context.
	\end{defn}

	\begin{thm}
		\label{A_infinity_algebra_up_to_pseudo_isotopy_thm}
		We can associate to the moduli space system $\mathbb M(J,L)$ an $A_\infty$ algebra on $\Omega^*(L)$ independent of the choices up to pseudo-isotopy.
	\end{thm}
	
	It is actually an immediate consequence of the main results of \cite{FOOODiskOne,FOOODiskTwo} and \cite[Theorem 21.35 (1)]{FOOO_Kuranishi}.
	We will offer a careful exposition in Appendix \ref{s_Kuranishi}. 
	Meanwhile, it should be possible to use other virtual techniques such as \cite{Polyfold,Pardon-vir,abouzaid2021complex} for deriving Theorem \ref{A_infinity_algebra_up_to_pseudo_isotopy_thm}. Here we use Kuranishi structure theory as it is currently available at the time of writing.
	
	Notably, the specific virtual technique employed to derive the $A_\infty$ structures from the moduli space of pseudo-holomorphic curves is really \textit{irrelevant} to the development of non-archimedean techniques in this paper, as we only need the existence of $A_\infty$ structures with the desired properties. 
	Therefore, in the initial reading, it might be advisable for a reader to tentatively accept Theorem \ref{A_infinity_algebra_up_to_pseudo_isotopy_thm} or Axiom \ref{axiom_J} below to maintain a focus.
	
	A reader uncomfortable with virtual techniques may take Theorem \ref{A_infinity_algebra_up_to_pseudo_isotopy_thm} as an axiom, which can be specifically stated as follows:
	
	\begin{axiom}\label{axiom_J}~
		\begin{itemize}[itemsep=4pt]
			\item \emph{Fix $J\in \mathfrak J(X)$ and $\mathbb M(J,L)$. There exists a set $\mathbb V(J; L)$ of virtual fundamental chains in which any element $\Xi$ is associated with an $A_\infty$ algebra $\check \m^{\Xi}$ in $\CC_{H_2(X,L)}(\Omega^*(L), \Omega^*(L))$.}
			
			\item \emph{Fix $J, J'\in\mathfrak J(X)$ and $\Xi\in \mathbb V(J; L)$, $\Xi'\in\mathbb V(J' ; L)$ defining the $A_\infty$ algebras $\check \m^{\Xi}$, $\check \m^{\Xi'}$ as above.
				Take a path $\pmb J=(J_t)_{0\leqslant t \leqslant 1}$ connecting $J$ and $J'$ and the moduli system $\mathbb M(\pmb J,L)$, there exists a set $\mathbb V(\pmb J, \Xi, \Xi'; L)$ of virtual fundamental chains in which any element $\pmb \Xi$ is associated with a pseudo-isotopy $\check \M^{\pmb \Xi}$ on $\Omega^*(\oi\times L)$ such that it restricts to $\check \m^{\Xi}$ and $\check \m^{\Xi'}$ at the two ends.}
		\end{itemize}
	\end{axiom}

	Roughly speaking, the choice $\Xi$ for an $A_\infty$ algebra refers to the following packages of data: (i) Kuranishi structures $\widehat{\mathcal U}_{k,\beta}$ on all moduli spaces $\mathcal M_{k+1,\beta}:=\mathcal M_{k+1,\beta}(J,L)$ such that they are compatible with respect to the restrictions to the normalized boundaries and corners and the evaluation maps at the marked points are strongly smooth. (ii) CF-perturbation data $\widehat{\mathfrak S}_{k,\beta}$ on all $\widehat{\mathcal U}_{k,\beta}$, with induced boundary and corner compatibility conditions, such that the output evaluation map is strongly transversal.

	We remark that any virtual technique achieves must allow dependence on various choices at the chain level, as explained in \cite[Remark 7.87]{FOOO_Kuranishi}. The disk bubbling phenomenon in moduli spaces of pseudo-holomorphic curves with Lagrangian boundaries is a codimension-one event. Consequently, we have to generate virtual fundamental "chains" rather than "cycles". Thus, the construction of associated $A_\infty$ algebras is inherently choice-sensitive.
	Nevertheless, certain homological algebra can be employed to minimize the dependence on choices, where maintaining coherence among various choices and treating them as an integrated system is often essential.
	Furthermore, the non-archimedean perspective could contribute in the following sense. Our ultimate goal is to derive invariants from these $A_\infty$ structures. In this regard, the automorphism group of an analytic domain or an affinoid algebra may possess sufficient breadth to absorb or neutralize the inherent flexibility and ambiguity in the various necessary choices.

	\section{Non-archimedean analytic and tropical geometry}
	\label{S_NA}
	
	This section provides an accessible overview of the fundamentals of non-archimedean geometry for those unfamiliar with the subject. It can be omitted by readers who are already familiar with the topic.
	Let $\Bbbk$ be a normed field. Namely, there is a norm on $\Bbbk$ as a function $x\mapsto |x|$ from $\Bbbk$ to $\mathbb R_{\geqslant 0}$ such that
	\begin{itemize}
		\itemsep 0pt
		\item $|x|=0$ if and only if $x=0$;
		\item $|xy|=|x||y|$;
		\item $|x+y|\leqslant |x|+|y|$.
	\end{itemize}
	Unless otherwise specified, we will always assume that $\Bbbk$ is algebraically closed and that the norm is complete.
	A norm is called \textit{non-archimedean} if it further satisfies the strong triangle inequality (also known as ultrametric inequality) $|x+y|\leqslant \max\{|x|, |y|\}$ (see e.g. \cite{BoschBook}).
	Then, a non-archimedean norm is equivalent to a (non-archimedean) valuation $\mathrm{val}: \Bbbk\to \mathbb R\cup \{\infty\}$ such that
	\begin{itemize}
		\itemsep 0pt
		\item $\mathrm{val}(x)=\infty $ if and only if $x=0$;
		\item $\mathrm{val}(xy)=\mathrm{val}(x)+\mathrm{val}(y)$;
		\item $\mathrm{val}(x+y)\geqslant \min\{\mathrm{val}(x),\mathrm{val}(y)\}$.
	\end{itemize}
	The relation is given by $\mathrm{val}(x)=-\log|x|$.
	A non-archimedean field is a field equipped with a non-archimedean norm.
	Several examples of non-archimedean fields include the field $\mathbb C((T))$ of rational functions, the field of Puiseux series, the field of $p$-adic numbers, and so on.
	However, in this paper, the non-archimedean field we want to exclusively consider is the \textit{Novikov field} defined as follows:
	\[
	\Bbbk=\Lambda= \mathbb C(( T^{\mathbb R})) = \left\{
	\sum_{i=0}^\infty a_i T^{\lambda_i} \mid a_i\in\mathbb C, \lambda_i\in\mathbb R,\lambda_i \nearrow +\infty
	\right\}
	\]
	Its non-archimedean valuation map
	\[
	\mathrm{val}: \Lambda\to \mathbb R \cup \{\infty\}
	\]
	is defined by taking the smallest power. Namely, given $x=\sum_{i=0}^\infty a_i T^{\lambda_i}$ with $a_0\neq 0$ and $\{\lambda_i\}$ strictly increasing, we define $\mathrm{val}(x) =\lambda_0$.
	Note that the image of the norm is the whole $R_{\geqslant 0}$ and the image of the valuation is the whole $\mathbb R$.
	
	\begin{rmk}
		The Novikov field is commonly used in symplectic geometry, especially concerning the Gromov's compactness of pseudo-holomorphic curves. The powers $\lambda_i$ can be intuitively regarded as energies of pseudo-holomorphic curves.
	\end{rmk}

	The valuation ring is given by $\Lambda_0:= \mathrm{val}^{-1}[0,\infty]=\{x\mid |x| \leqslant 1 \}$, often called the \textit{Novikov ring}.
	It has the maximal ideal $\Lambda_+:=\mathrm{val}^{-1}(0,\infty]=\{x\mid |x|<1\}$.
	The multiplicative group of units in the field $\Lambda$ is denoted as 
	\[
	U_\Lambda:= \mathrm{val}^{-1}(0)=\{x\mid |x|=1\}
	\]
	We may call $U_\Lambda$ the \textit{non-archimedean unit circle}.
	Notice that $U_\Lambda= \mathbb C^*\oplus \Lambda_+$ and $\Lambda_0 =\mathbb C\oplus \Lambda_+$. The standard isomorphism $\mathbb C^*\cong \mathbb C/ 2\pi i \mathbb Z$ naturally extends to $U_\Lambda\cong \Lambda_0/ 2\pi i \mathbb Z$. In particular, for any $y\in U_\Lambda$, there exists some $x\in \Lambda_0$ with $y=\exp(x)$.
	This is a special property of Novikov field.

	\subsection{Non-archimedean analytic space}
	Let's first give an easy-to-use perspective for readers who are not familiar with non-archimedean geometry.
	Take a smooth algebraic variety $X$ over $\Bbbk$. If $\Bbbk=\mathbb C$, then $X$ can be considered a complex manifold or a complex analytic space by refining the Zariski topology with the complex norm $|\hspace{-0.2em}\cdot \hspace{-0.2em}|_{\mathbb C}$.
	For example, sequence convergence for (closed) points in $X$ can be discussed by measuring the complex norm. Likewise, if $\Bbbk$ is equipped with a non-archimedean norm $|\hspace{-0.2em}\cdot \hspace{-0.2em}|_{NA}$, sequence convergence can similarly be addressed by measuring this non-archimedean norm. For this, we often write $X^{\mathrm{an}}$ instead of $X$.
	This process is called the \textit{analytification}. We refer to \cite[Section 5.4]{BoschBook} for further reading.
	Indeed, we have the so-called GAGA-functor from the category of $\Bbbk$ schemes of locally finite type to the category of non-archimedean analytic $\Bbbk$-space (see \cite[Sec. 5.4]{BoschBook} or \cite[Sec.5.1]{TemkinIntro}).
	In a word, many concepts from complex geometry can be similarly applied to non-archimedean analytic geometry, with some necessary technical adjustments.
	For the objectives of this paper, this somewhat simplified perspective suffices for our main result.

	To delve deeper into non-archimedean geometry, let's start with some definitions.
	Due to the strong triangle inequality, a series $\sum_{\nu=0}^\infty a_\nu$ of elements $a_\nu\in \Bbbk$ is convergent with respect to the non-archimedean norm if and only if $\lim_{\nu\to\infty} |a_\nu|=0$ \cite[p10]{BoschBook}.
	
	Let 
	\[
	T_d= \Bbbk \langle z_1,\dots,z_d\rangle \subseteq  \Bbbk[[z_1,\dots,z_d]]
	\]
	be the subset of all formal power series $\sum_{\nu\in\mathbb Z^d_{\ge 0}}a_{\nu} \mathbf{z}^\nu$ so that $|a_{\nu}| \to 0$ as $|\nu| = \sum |\nu_i|\to \infty$. It form a Banach $\Bbbk$-algebra, called the $d$-th \textit{Tate algebra}. 
	Alternatively, the Tate algebra $T_d$ is the set of all formal power series that converges on the closed unit polydisk:
	\[
	\mathbb B^d = \{ x=(x_1,\dots, x_d)\in \Bbbk^d \mid |x_i|\leqslant 1 \}
	\]
	Moreover, there is a bijection between the points in $\mathbb B^n$ and the set of maximal ideals in $T_d$ given by sending a point $x$ to the ideal $\m_x=\{f\in T_n \mid f(x)=0\}$ \cite[2.2]{BoschBook}.
	In other words, the closed unit polydisk $\mathbb B^d$ is the spectrum of maximal ideals in $T_d$. 
	This observation morally leads to the definition of non-archimedean analytic space.
	
	An \textit{affinoid algebra} is defined to be a $\Bbbk$-Banach algebra $A$ with a continuous epimorphism $T_d\to A$ for some $d$. 
	An \textit{affinoid domain}, denoted by $\mathrm{Sp}(A)$, is the spectrum of maximal ideals in an affinoid algebra $A$.
	For example, the Tate algebra itself is clearly an affinoid algebra, and the closed unit polydisk $\mathbb B^d$ above is an affinoid domain correspoinding to $T^d$.
	Clearly, $\Sp(A) \neq \varnothing$ unless $A=0$. If $A= T_d / \langle f_1,\dots, f_r\rangle$ for $f_1,\dots, f_r\in A$, then $\Sp(A)$ can be identified with the zero locus of $f_1,\dots, f_r$ within the unit polydisk $\mathbb B^d$; see e.g. \cite[Proposition 3.1.8]{EKL}.
	In other wrods, we have $\Sp(A)=\{ x=(x_1,\dots, x_d)\in \mathbb B^d \mid f_1(x)=\cdots=f_r(x)=0\}$.

	A \textit{non-archimedean analytic space} is a locally ringed space whose structure sheaf is locally affinoid algebras.
	Similar to the scheme theory, the \textit{dimension} can be defined by examining the Krull dimension of local rings \cite[p300]{BGR}.
	Just as other geometric objects, non-archimedean spaces also adhere to the cocycle condition for local-to-global gluing; see e.g. \cite[5.3/5]{BoschBook} \cite[Sec 4.1.4]{TemkinIntro}.
	

	\begin{rmk}
		To be exceedingly precise, the term "non-archimedean analytic space" can refer to either Tate's rigid analytic space \cite{Tate_origin} or Berkovich's analytic space \cite{Berkovich1993etale,Berkovich_2012spectral}. Nonetheless, the differences are very mild, and delving into this distinction may exceed the intended scope of this paper.
		Roughly, given an affinoid algebra, taking the spectrum of maximal ideals of them goes to the Tate's rigid geometry; meanwhile, taking the spectrum of multiplicative semi-norms produces extra points and goes to the Berkovich geometry.
		A fully faithful functor exists between the categories of Berkovich spaces and rigid spaces; see \cite{Berkovich_2012spectral, Berkovich1993etale} and \cite[Sec C.1]{baker2010potential}).
		Intuitively, a Berkovich space $V$ can be obtained by adding `generic points' back to a rigid space $V_0$ \cite[Sec 7.1]{Fresnel_2012rigid}; in turn, the $V_0$ is everywhere dense in $V$ \cite[2.1.15]{Berkovich_2012spectral}. For instance, if $V$ is the Berkovich affine line, then $V_0$ is the points of the field $\Lambda$, while a closed \textit{disk} in $\Lambda$ is actually a sort of `generic point' in $V\setminus V_0$ \cite[1.3.1]{baker2008introduction}.  This resembles the relationship between schemes and varieties by including generic points of prime ideals, but in a different flavor.
	\end{rmk}

	\subsection{Tropicalization map}
	From now on, let's fix $\Bbbk=\Lambda$.
	Let $\Lambda^* =\Lambda\setminus\{0\}$. For our purpose, we aim to consider the \textit{tropicalization map}:
	\begin{equation}
		\label{trop_eq}
		\trop \equiv \mathrm{val}^{\times n} : (\Lambda^*)^n \to \mathbb R^n
	\end{equation}
	This is an analog of the complex logarithm map $\mathrm{Log}=\log|\hspace{-0.2em}\cdot \hspace{-0.2em}|^{\times n}: (\mathbb C^*)^n\to\mathbb R^n$.
	The total space $(\Lambda^*)^n$ is a non-archimedean analytic space.
	Let $\Delta$ be a rational convex polyhedron in $\mathbb R^n$, defined by a finite system of affine-linear inequalities
	$
	\sum_{j=1}^n b_{ij} u_j \geqslant c_i
	$
	for $i=1,\dots, N$, $c_i\in\mathbb R$, $b_{ij}\in\mathbb Z$, and $(u_1,\dots, u_n)\in\mathbb R^n$.
	We do not require $\Delta$ to have full dimension.
	The following result is due to Einsiedler, Kapranov, and Lind \cite[Proposition 3.1.5]{EKL}.

	\begin{prop}
		The preimage $\trop^{-1}(\Delta)$ is an affinoid domain; the corresponding affinoid algebra is
		\[
		\Lambda\langle \Delta \rangle  =\left\{
		\sum_{i=0}^\infty s_i Y^{\alpha_i} \in \Lambda[[\mathbb Z^n]] \mid s_i\in\Lambda, \  \alpha_i\in\mathbb Z^n , \ \mathrm{val}(s_i)+ \alpha_i \cdot \gamma \to \infty  \  \ \text{for all} \  \gamma \in\Delta 
		\right\}
		\]
		in the Laurent formal power series ring
		$\Lambda[[\mathbb Z^n]]\equiv \Lambda[[Y_1^\pm,\dots, Y_n^\pm]]$.
	\end{prop}
	
	Let's call $\Lambda\langle \Delta\rangle$ the ring of \textit{(strictly) convergent formal power series} \cite[6.1.4]{BGR} \cite{EKL}.
	In other words, the above proposition means that the spectrum $\Sp(\Lambda\langle \Delta\rangle)$ of maximal ideals in $\Lambda\langle \Delta\rangle$ is identified with the points in $\trop^{-1}(\Delta)$.
	Moreover, given $f_1,\dots, f_r\in \Lambda\langle \Delta\rangle$, we have $f_i(x)$ forms a convergent series in $\Lambda$ for any point $x\in \trop^{-1}(\Delta)$; for 
	\[
	A=\Lambda\langle \Delta\rangle / (f_1,\dots, f_r)
	\]
	the spectrum $\Sp(A)$ of the maximal ideals in $A$ is then identified with the zero locus of $f_1,\dots, f_r$ in $\trop^{-1}(\Delta)$ \cite[Proposition 3.1.8]{EKL}.
	Specifically, 
	\[
	\Sp(A)=\{ x\in\trop^{-1}(\Delta)\mid f_1(x)=\cdots =f_r(x)=0\}
	\]
	One can verify that its dimension is $n-r$.

	There is also a helpful perspective from group algebra. Recall that the Laurent polynomial ring of $n$ variables can be viewed as the (semi)group algebra associated to the abelian group $\mathbb Z^n$, namely, the set of functions $\mathbb Z^n\to \Lambda$ with finite support.
	Removing the condition of finite support yields the Laurent formal power series ring.
	Therefore, we can interpret an element in $\Lambda\langle \Delta\rangle$ as a function $f: \mathbb Z^n \to \Lambda$ with the extra convergence condition: $\mathrm{val}(f(  \nu))+   \nu \cdot \gamma\to \infty$ as $|  \nu|\to\infty$, where $\nu\in\mathbb Z^n$.

	We remark that if we replace $\Delta$ with an open subset $  U$, the preimage $\trop^{-1}(  U)$ is still an open domain in $(\Lambda^*)^n$ with respect to the non-archimedean analytic topology. But, it is no longer an affinoid domain as the cooridnate ring $\Lambda\langle   U\rangle$, defined in the same way, of analytic functions on $\trop^{-1}(U)$ is no longer an affinoid algebra. One can cover an arbitrary analytic open domain by a collection of affinoid domains.
	To illustrate this point, let's elaborate on the GAGA-principle mentioned earlier for a specific example. Fix $d\in\mathbb Z_{>0}$ and $s>1$. We consider the (scaled) Tate algebra $T_n^{(i)}:= \Bbbk\langle s^{-i} z_1, \dots,  s^{-i} z_n \rangle$ for some $i\in \mathbb N$ which consists of formal power series $\sum a^\nu \mathbf z^\nu$ so that $|a_\nu| s^{i|\nu|} \to 0$.
	Then, $\Sp T_n^{(i)}$ agrees with the closed ball with radius $s^i$.
	The analytification of the affine $n$-space $\mathbb A_{\mathbb K}^n$ is given by $\mathbb A_{\mathbb K}^{n,\mathrm{an}}=\bigcup_{i=0}^\infty \Sp T_n^{(i)}$.
	Likewise, one may describe an open polydisk $\{|x|<r\}$ as an exhausting union of closed polydisks $\{|x|\leqslant r_i\}$ similarly, where closed polydisks are affinoid domains but an open polydisk is not.
	To sum up, affinoid domains are just slightly more preliminary building blocks, just as we often take a polydisk or a ball as a local chart in complex manifolds.

	\section{Proper unobstructedness in Lagrangian Floer theory}
	\label{s_proper_unobstructed}
	
	Let's go back to the setting of sympletic geometry. Recall that $L$ is a Lagrangian submanifold in a symplectic manifold $(X,\omega)$ as before.
	Consider the formal power series version of the group algebra:
	\begin{equation}
		\label{Lambda_[[_mathcal H_1L}
		\Lambda[[H_1(L)]] =\left\{ \sum_{i=0}^\infty  c_i Y^{\alpha_i}  \mid c_i\in\Lambda,  \alpha_i\in H_1(L) \right\}
	\end{equation}
	where $Y$ is a formal symbol and $H_1(L)=H_1(L;\mathbb Z)$ is the singular homology group of $L$.

	For $\Delta\subseteq H^1(L)$, we define the following $\Lambda$-algebra
	\begin{equation*}
		\Lambda\langle H_1(L); \Delta \rangle :=
		\left\{ 
		\sum_{i=0}^\infty c_i Y^{\alpha_i} \in \Lambda[[H_1(L)]] \mid
		c_i \in\Lambda, \alpha_i\in H_1(L), \, \,  \, \,
		\mathrm{val}(c_i) + \langle \alpha_i , \gamma \rangle \to \infty \ \text {for all $\gamma\in \Delta$}
		\right\}
	\end{equation*}

	We claim that any strictly convergent formal power series $f=\sum c_i Y^{\alpha_i}$ in $\Lambda\langle H_1(L) ; \Delta\rangle$ can be naturally interpreted as a well-defined function $H^1(L;U_\Lambda)\times \Delta \to \Lambda$.
	Indeed, we denote by $\mathbf y^\alpha \in U_\Lambda$ the image of a cycle $\alpha$ in $H_1(L)$ under the map associated to $\mathbf y$ via $H^1(L;U_\Lambda) \cong \Hom (H_1(L); U_\Lambda)$.
	Given $(\mathbf y,\gamma)\in H^1(L; U_\Lambda)\times \Delta$, we want to specify a number
	$
	f(\mathbf y, \gamma)= \sum c_i \cdot T^{\langle \gamma, \alpha_i\rangle} \mathbf y^{\alpha_i}
	$ in the Novikov field $\Lambda$, where we need to check that it forms a convergent series in $\Lambda$ with respect to the adic topology induced by the non-archimedean valuation. In other words, we need to check that $\mathrm{val}(c_i \cdot T^{\langle \gamma,\alpha_i\rangle} \mathbf y^{\alpha_i})= \mathrm{val}(c_i) +\langle \gamma,\alpha_i \rangle \to\infty$. This is exactly the definition.

	\begin{rmk}
		\label{torsion_rmk}
		In general, $H_1(L)$ may have torsion parts. For instance, if $H_1(L)\cong \mathbb Z^m \oplus \mathbb Z/p\mathbb Z$, then $\Lambda\langle H_1(L)\rangle \cong \Lambda\langle Y_1^\pm,\dots, Y_m^\pm\rangle [Z]/(Z^p-1)$.
		Indeed, let $e_1,\dots, e_m$ and $\xi$ be the generators of $\mathbb Z^m$ and $\mathbb Z/p\mathbb Z$ respectively.
		Set $Y_k=Y^{e_k}$ and $Z=Y^{\xi}$, and then $Z^p=Y^{p\xi}=Y^0=1$.
		Now, an arbitrary formal power series $F$ can be expressed as $F=F_0 +F_1 Z+\cdots +F_{p-1} Z^{p-1}$ where each $F_i\in \Lambda[[Y_1^\pm,\dots, Y_m^\pm]]$.
		Since the first cohomology $H^1(L;\mathbb Z)$ is always torsion-free by the universal coefficient theorem, the cap product of a cohomology class with a torsion homology class is always zero, that is, $\langle \alpha , \xi\rangle=0$ for any $\alpha\in H_1(L)$.
		Therefore, the claim follows.
	\end{rmk}
	
	We set 
	\[\Lambda\langle H_1(L)\rangle :=\Lambda\langle H_1(L);0 \rangle
	\] 
	for $\Delta=\{0\}$ in $H^1(L)$.
	It is easy to observe that 
	$
	\Lambda\langle H_1(L);\Delta_1\cup \Delta_2\rangle =\Lambda\langle H_1(L); \Delta_1) \cap \Lambda\langle H_1(L); \Delta_2\rangle
	$.
	Thus, if $\Delta_1\supseteq \Delta_2$, then $\Lambda\langle H_1(L); \Delta_1\rangle \subseteq \Lambda\langle H_1(L); \Delta_2\rangle$.
	In particular, for any $\Delta$ that contains $0$, we have $\Lambda\langle H_1(L); \Delta\rangle \subseteq \Lambda\langle H_1(L) \rangle $ which is only `$\subsetneq$' in general.

	\subsection{Obstruction ideal}
	\label{ss_obstruction_ideal}
	
	\begin{situation}
		\label{situation_A_inf_alg_UD}
		Let $\check \m=\check \m^{\Xi}$ be an $A_\infty$ algebra on $\OL$ as in Axiom \ref{axiom_J} where $\Xi$ denotes the choice.
		By Theorem \ref{canonical_model_general_tree-thm}, applying the homological perturbation with respect to a contraction $g = (i, \pi, G)$ transforms the $A_\infty$ algebra $\check \m$ on $\Omega^*(L)$ into a \textit{minimal} $A_\infty$ algebra $\m$ on $\HL$.
		By Proposition \ref{minimal_model_prop}, we know that $\m$ also lives in the category $\UD$.
	\end{situation}
	
	From a more geometric perspective, one can interpret the homological perturbation process for the minimal $A_\infty$ algebra $(\HL,\m)$ as counting holomorphic pearly trees \cite{Sheridan15,FOOO_2009canonical_Morse}.
	For example, if a nonzero class $\beta_1$ in $H_2(X,L)$ can be represented by a nonconstant Maslov-0 holomorphic disk, then the above series actually involves infinitely many potentially nonvanishing terms $\m_{0, a\beta_1}$ for all $a\in\mathbb N$.

	After the homological perturbation, it is important to note that $\m_{0,\beta}$ lives in the finite-dimensional vector space $H^*(L)$. We consider a vector-valued Laurent formal power series
	\[
	\sum_{\beta\in H_2(X,L)} T^{E(\beta)} Y^{\partial\beta} \m_{0,\beta}
	\]
	which can be viewed as an element in $\Lambda[[H_1(L)]]\otimes H^*(L)$.
	
	Thanks to the Gromov compactness or the gappedness condition, the $\m_{0,\beta}$ is nonzero for at most countably many $\beta$'s, and we can order these $\beta$'s by making the values of $E(\beta)$ an increasing sequence.
	Moreover, if $\m_{0,\beta}\in H^{2-\mu(\beta)}(L)$ is nonzero, then recall that our assumption yields $\mu(\beta)\geqslant 0$. Therefore, $\m_{0,\beta}$ is contained either in $H^0(L)$ or $H^2(L)$, and only the terms with $\mu(\beta)=2$ or $\mu(\beta)=0$ can contribute.
	Slightly abusing the notation,  we often view those $\m_{0,\beta}$ in $H^0(L)$ as real numbers since $H^0(L)\cong \mathbb R$.

	\begin{defn}
		\label{W_Q_defn}
		We call
		\[
		W:=\sum_{\mu(\beta)=2} T^{E(\beta)} Y^{\partial\beta} \m_{0,\beta}
		\]
		the \textit{superpotential} associated to $\m$. We call
		\[
		Q:= \sum_{\mu(\beta)=0} T^{E(\beta)} Y^{\partial\beta} \m_{0,\beta}
		\]
		the \textit{obstruction series} associated to $\m$.
		Taking a basis $\{\Theta_1,\dots, \Theta_\ell\}$ of $H^2(L)$, one may write $Q= Q_1\Theta_1+\cdots + Q_\ell\Theta_\ell$ where each $Q_i\in\Lambda \langle H_1(L) \rangle$.
		Abusing the terminology, we also call $Q_1,\dots, Q_\ell$ the \textit{obstruction series} associated to $\m$.
	\end{defn}
	
	In particular, we have the following equation in $\Lambda[[H_1(L)]]\otimes H^*(L)$:
	\begin{equation}
		\label{W_Q_expand_eq}
		\sum_\beta T^{E(\beta)} Y^{\partial\beta} \m_{0,\beta} = W\cdot \one + \sum_{i=1}^\ell Q_i \cdot \Theta_i
	\end{equation}
	
	\begin{prop}
		\label{strictly_convergent_prop}
		The formal power series $W$ and $Q_i$'s are contained in $\Lambda\langle H_1(L) \rangle$.
		Moreover, for a sufficiently small $\Delta \subset H^1(L)$ that contains the origin $0$, 
		$W$ and $Q_i$'s are also contained in $\Lambda\langle H_1(L) ;  \Delta\rangle$.
	\end{prop}
	
	\begin{proof}
		Consider the countable subset of $H_2(X,L)$ consisting of $\beta$'s such that $\m_{0,\beta}\neq 0$, and then we aim to show $\mathrm{val}( T^{E(\beta)})\equiv E(\beta)\to \infty$ for the first half of the proposition. 
		This is straightforward from the Gromov compactness.
		The second half is also easy to obtain from Groman-Solomon's reverse isoperimetric inequality \cite{ReverseI}. It states that there exists a constant $c>0$ such that whenever $\beta\in H_2(X,L)$ is represented by a pseudo-holomorphic curve, we have $E(\beta)\geqslant c \ell(\partial\beta)$ where $\ell$ denotes the length.
		Specifically, our goal is to show that $E(\beta)+\langle \partial\beta ,\gamma\rangle $ is still divergent to $\infty$ for any given $\gamma \in \Delta$.
		Then, we just need to choose $\Delta$ sufficiently small so that $|\langle \partial\beta,\gamma\rangle|\le \frac{c}{2} \ell(\partial\beta)$.
		%
	\end{proof}

	\begin{defn}
		\label{obstruction_ideal_defn}
		We call the ideal
		$
		\mathfrak a = ( Q_1,\dots, Q_\ell) $ in $\Lambda \langle H_1(L) \rangle
		$
		generated by $Q_i$'s the \textit{obstruction ideal} associated to $\m$.
		We say the $A_\infty$ algebra $\m$ is \textit{properly unobstructed} if $\ia$ is vanishing.
	\end{defn}

	A pivotal question arises when considering alternative choices of $g'$, $J'$, and $\Xi'$. It is reasonable to anticipate that taking these choices differently should not affect the vanishing of the obstruction ideal $\ia$. If this holds true, it then becomes valid to introduce Definition \ref{proper_unobstructed_introduction_defn}. The exploration of this question will be undertaken in Sec \ref{ss_inv_proper_unobstructedness}. Let us begin with some necessary preliminaries.

	\subsection{Transitioning phase series}
	
	Let $(\HL,\m)$ and $(\HL,\m')$ be two objects within the category $\UD$, with $\f: \m \to \m'$ serving as a morphism between them.
	
	\begin{defn}
		\label{transitioning_phase_defn}
		We call
		\begin{equation}
			\label{P_f_transition_phase}
			P_{\f} := \sum_\beta T^{E(\beta)} Y^{\partial\beta} \f_{0,\beta} 
		\end{equation}
		the \textit{transitioning phase series} associated to $\f$. For any $\alpha\in H_1(L)$, we set 
		\[
		\langle \alpha, P_\f\rangle =\sum_\beta T^{E(\beta)} Y^{\partial\beta} \langle \alpha,\f_{0,\beta}\rangle
		\]
		For degree reasons, every nonzero $\f_{0,\beta}$ lies in $H^{1-\mu(\beta)}( L)$, further implying $\mu(\beta)=0$ and $\f_{0,\beta}\in H^1(L)$.
		Hence, the sum in (\ref{P_f_transition_phase}) is automatically taken over $\mu(\beta)=0$.
		Arguing as Proposition \ref{strictly_convergent_prop}, one can similarly show that the series $\langle \alpha, P_\f\rangle $ for any $\alpha\in H_1(L)$ is contained in the strictly convergent formal power series ring in $\Lambda\langle H_1(L) ;  \Delta\rangle$ for sufficiently small $\Delta$.
	\end{defn}

	\begin{defn}
		The assignment
		\begin{equation}
			\label{phi_defn_eq}
			Y^{\alpha} \mapsto Y^{\alpha} \exp \langle \alpha,  P_\f \rangle
		\end{equation}
		for $\alpha\in  H_1(L)$ naturally gives rise to an algebra endomorphism
		\[
		\phi=\phi_\f :  \Lambda\langle H_1(L);\Delta\rangle \to\Lambda\langle H_1(L);\Delta\rangle 
		\]
		for a sufficiently small $\Delta$.
		We call $\phi=\phi_\f$ the \emph{pre-transitioning morphism} associated to $\f$.
	\end{defn}
	
	\begin{rmk}
		As mentioned in the introduction, the motivation behind introducing $P_\f$ and $\phi=\phi_\f$ as above is that it ultimately leads to a transition map for the structure sheaf of a non-archimedean analytic space in the context of SYZ mirror construction. Nevertheless, this aspect is not explored in the current paper. In particular, $L$ is not necessarily a Lagrangian torus.
	\end{rmk}
	
	The following lemma demonstrates the relationship between the transitioning phase series of the composition of two $A_\infty$ homomorphisms and the transitioning phase series of each individual $A_\infty$ homomorphism.
	
	\begin{lem}
		\label{P_fg=_lem}
		Suppose $(\HL,\m'')$ is another object and $\g:\m''\to\m$ is a morphism in $\UD$. Then,
		\[
		P_{\f\diamond \g}= \f_{1,0}(P_\g)  + \phi_\g(P_\f)
		\]
		where $\f\diamond \g$ is the composition defined as (\ref{composition_Gerstenhaber_eq}).
	\end{lem}
	
	\begin{proof}
		We first use the formula (\ref{composition_Gerstenhaber_eq}) to expand the left side as follows:
		\begin{align*}
			P_{\f\diamond \g}
			&
			=
			\sum_\gamma T^{E(\gamma)} Y^{\partial\gamma} (\f \diamond \g)_{0,\gamma}  \\
			&
			=
			\sum_\gamma T^{E(\gamma)} Y^{\partial\gamma} \f_{1,0} (\g_{0,\gamma}) 
			+
			\sum_{(k,\beta)\neq (1,0)}  T^{E(\beta)}Y^{\partial\beta} \f_{k,\beta} \Big( \quad \sum_{\beta_1} T^{E(\beta_1)}Y^{\partial\beta_1} \g_{0,\beta_1} \quad , \quad \dots \quad , \quad \sum_{\beta_k} T^{E(\beta_k)} Y^{\partial\beta_k} \g_{0,\beta_k} \quad \Big) 
		\end{align*}
		In the above, the first sum is then given by
		\[
		\f_{1,0}\Big( \sum T^{E(\gamma)}Y^{\partial\gamma} \g_{0,\gamma} \Big) = \f_{1,0} (P_\g)
		\]
		Meanwhile, the second sum is equal to
		\begin{align*}
			&
			\
			\sum_{(k,\beta)\neq (1,0)}  T^{E(\beta)}Y^{\partial\beta} \f_{k,\beta} \Big( \quad \sum_{\beta_1} T^{E(\beta_1)}Y^{\partial\beta_1} \g_{0,\beta_1} \quad , \quad \cdots \quad , \quad \sum_{\beta_k} T^{E(\beta_k)} Y^{\partial\beta_k} \g_{0,\beta_k} \quad \Big) \\
			=
			&
			\
			\sum_{(k,\beta)\neq (1,0)} T^{E(\beta)}Y^{\partial\beta} \f_{k,\beta} \big ( \ \ P_\g , \  \cdots  \  ,  \  P_\g \  \  \big)
		\end{align*}
		Recall that by Definition \ref{transitioning_phase_defn}, it has been established that $P_\g=\sum T^{E(\beta)} Y^{\partial\beta} \g_{0,\beta}$ receives contributions solely from $\g_{0,\beta}\in H^1(L)$ where $\mu(\beta)=0$, due to considerations of degree.
		In other words, the terms in $P_\g$ are all of degree one.
		This is exactly what happens when we apply the divisor axiom.
		Further considering the definition of $\phi_\g$ in (\ref{phi_defn_eq}), the second sum becomes
		\[
		\sum_\beta T^{E(\beta)} Y^{\partial\beta} \exp\langle \partial\beta, P_\g\rangle \
		\f_{0,\beta} 
		=
		\sum_\beta T^{E(\beta)} \phi_\g( Y^{\partial\beta} )  \ \f_{0,\beta}
		=
		\phi_\g\big(\sum_\beta T^{E(\beta)}Y^{\partial\beta} \f_{0,\beta} \big)
		=
		\phi_\g(P_\f)
		\]
		Putting things together, we have proved that $P_{\f\diamond \g}=\f_{1,0}(P_\g) +\phi_\g(P_\f)$ as desired.
	\end{proof}
	
	\begin{cor}
		\label{phi_fg_cor}
		If $\f_{1,0}=\id$, then $\phi_{\f\diamond \g}= \phi_\g\circ \phi_\f$.
	\end{cor}
	
	\begin{proof}
		It is straightforward from Lemma \ref{P_fg=_lem}. We can compute as follows:
		\begin{align*}
			\phi_{\f\diamond \g} (Y^\alpha) 
			=
			Y^{\alpha} \exp\langle \alpha, P_{\f\diamond \g}\rangle
			&
			=
			Y^\alpha \exp\langle \alpha,  P_\g+\phi_\g(P_\f) \rangle  \\
			&
			=
			Y^{\alpha} \cdot \exp\langle \alpha, P_\g\rangle \cdot  \exp \langle \alpha, \phi_\g(P_\f)\rangle \\
			&
			=
			\phi_\g\big(Y^\alpha) \cdot \phi_\g \big( \exp\langle \alpha, P_\f\rangle \big)
			\\
			&
			=
			\phi_\g\big(Y^\alpha \exp\langle \alpha, P_\f\rangle \big)
			=
			\phi_\g(\phi_\f(Y^\alpha))
		\end{align*}
	\end{proof}
	
	\begin{rmk}
		In practical applications, the $A_\infty$ homomorphism $\f$ (and similarly for $\g$) is usually derived from a pseudo-isotopy of $A_\infty$ algebras, in view of Theorem \ref{A_infinity_algebra_up_to_pseudo_isotopy_thm} or Axiom \ref{axiom_J}. Furthermore, such $A_\infty$ homomorphism $\f$ always satisfies the condition $\f_{1,0}=\id$, as established in Theorem \ref{from_pseudo_to_A_homo_thm}.
	\end{rmk}

	\subsection{Wall-crossing formula}
	
	Given $\f: \m\to \m'$ as before, we have defined $P_\f$ as (\ref{P_f_transition_phase}) and $\phi=\phi_\f$ as (\ref{phi_defn_eq}).
	Recall that $\{\Theta_i: i=1,2,\dots,\ell \}$ denotes a basis of the vector space $H^2(L)$ and $\one$ is the generator of $H^0(L)$. 
	We introduce the following vector-valued series in $\Lambda[[H_1(L)]]\otimes H^*(L)$:
	\begin{equation}
		\label{R_f_notation}
		R^k := R_{\f}^k := \sum_\beta T^{E(\beta)} Y^{\partial\beta} \f_{1,\beta}(\Theta_k)
	\end{equation}
	Notice that $\deg (\f_{1,\beta}(\Theta_k) ) = \deg \Theta_k -\mu(\beta) = 2-\mu(\beta)$, namely, $\f_{1,\beta}(\Theta_k) \in H^{2-\mu(\beta)}(L) $. Similarly, we may assume it is always contained in $H^0(L)$ or $H^2(L)$ by the semipositive condition.
	Thus, we can further write
	\begin{equation}
		\label{R_f_notation_II}
		R^k =  \underline R^k \cdot \one + \sum_{i=1}^\ell R_i^k  \cdot \Theta_i 
	\end{equation}
	for some formal power series $\underline R^k$ and $R_i^k$'s in $\Lambda[[H_1(L)]]$.
	We remark that $\underline R^k$ is contributed by terms with $\mu(\beta)=2$ and $R^k_i$'s are contributed by terms with $\mu(\beta)=0$.
	Again, arguing as Proposition \ref{strictly_convergent_prop}, one can deduce that $\underline R^k$ and $R_i^k$'s are contained in the strictly convergent formal power series ring $\Lambda\langle H_1(L) ;  \Delta\rangle$ for some $\Delta$.
	
	\begin{thm}[Wall-crossing formula]
		\label{wall_cross_formula_thm}
		Let $W$ and $W'$ be the superpotentials associated to $\m$ and $\m'$ respectively; let $Q_j$'s and $Q'_j$'s be the obstruction series associated to $\m$ and $\m'$ respectively. Then,
		\begin{equation}
			\label{wall_cross_eq}
			\begin{aligned}
				\phi(W')  &=W  +  \sum_{k=1}^\ell Q_k\cdot  \underline R^k  \\
				\phi(Q'_j) & = \qquad \sum_{k=1}^\ell Q_k \cdot R^k_j
			\end{aligned}
		\end{equation}
	\end{thm}
	
	\begin{proof}
		Recall that we have chosen the basis $\{\Theta_1,\dots, \Theta_\ell\}$ of $H^2(L)$ and the basis $\{\one\}$ of $H^0(L)$.
		Let $\{\eta_i\}_{i=1}^N$ be some basis of $H^*(L)$ that includes the above $\{\Theta_i\}$ and $\{\one\}$.
		Then, an element $F$ in $\Lambda[[ H_1(L)]] \otimes H^{*}(L)$ can be viewed as a tuple $(F_1,\dots, F_N)$ where each $F_i\in \Lambda[[H_1(L)]]$ and $F=\sum_i F_i   \eta_i$. 
		Given an element $\eta$ in this basis, let $\left[ \eta, F \right]$ denote the corresponding coefficient of $F$ in $\Lambda[[H_1(L)]]$. 
		We aim to study the expression 
		\[
		\phi^\eta:= \phi \Big(\Big[\eta, W' \ \one+ \sum_j Q_j' \ \Theta_j \Big]\Big) = \phi \Big(\sum_\gamma T^{E(\gamma)} Y^{\partial\gamma} \left[\eta, \m'_{0,\gamma} \right]\Big)
		\]
		in view of (\ref{W_Q_expand_eq}).
		For instance, when $\eta$ is $\one$ or $\Theta_j$, the expression
		$\phi^\eta$ gives $\phi(W')$ or $\phi(Q_j')$ respectively.
		By definition, we first have
		\begin{align*}
			\phi^\eta
			&
			=
			\sum_\gamma T^{E(\gamma)}  \phi(Y^{\partial\gamma}) \  \left[  \eta,  \m'_{0,\gamma} \right]
			=
			\quad  \sum_\gamma T^{E(\gamma)} Y^{\partial\gamma} \exp\langle \partial\gamma, P_\f\rangle 	 \left[  \eta, \m'_{0,\gamma}  \right] 
		\end{align*}
		Then, using the divisor axiom implies that
		\begin{align*}
			\phi^\eta
			&
			=
			\Big[ \eta,  \quad  \sum_\gamma T^{E(\gamma)} Y^{\partial\gamma} \exp\langle \partial\gamma, P_\f\rangle \  \m'_{0,\gamma} \Big]  =
			\Big[ \eta , \quad  \sum_{(k,\gamma)\neq (1,0)} T^{E(\gamma)} Y^{\partial\gamma} \m'_{k,\gamma} (P_\f,\dots, P_\f) \Big]
		\end{align*}
		where we recall that $\m'_{1,0}=0$ as $\m'$ is a minimal $A_\infty$ algebra.
		Expanding the coefficients of all $P_\f$'s, employing the $A_\infty$ relation of $\f$, and using the equation (\ref{W_Q_expand_eq}), we obtain that
		\begin{align*}
			\phi^\eta
			&
			=
			\Big[ \eta, \quad \sum_{\beta} T^{E(\beta)} Y^{\partial\beta}  \sum_{\beta=\gamma+\sum\beta_i}
			\sum_k
			\m'_{k,\gamma}( \f_{0,\beta_1} , \dots, \f_{0,\beta_k} ) 
			\Big] \\
			&
			=
			\Big[
			\eta, \quad \sum_{\beta} T^{E(\beta)} Y^{\partial\beta} \sum_{\beta=\gamma_1+\gamma_2} \quad  \f_{1,\gamma_1} (\m_{0,\gamma_2})  \quad 
			\Big] \\
			&
			=
			\Big[ \eta, \quad  \sum_{\gamma_1} T^{E(\gamma_1)} Y^{\partial\gamma_1} \f_{1,\gamma_1} ( W\one + \sum_k Q_k \Theta_k )  \Big]
		\end{align*}
		Since $\f$ is unital, we have $\f_{1,\gamma_1}(\one)=0$ whenever $\gamma_1\neq 0$ and $\f_{1,0}(\one)=\one$.
		Further considering (\ref{R_f_notation}) and (\ref{R_f_notation_II}), we conclude that
		\begin{align*}
			\phi^\eta
			&
			=
			\Big[ \eta, \quad  
			W \ \one +\sum_{k=1}^\ell Q_k  \ R^k \Big] 
			=
			\Big[
			\eta, \quad  W \ \one +\sum_{k=1}^\ell Q_k \ \Big( \underline R^k \ \one + \sum_{j=1}^\ell R_j^k \ \Theta_j \Big)
			\Big] \\
			&
			=
			\Big[
			\eta, \quad \Big(W + \sum_{k=1}^\ell Q_k \ \underline{R}^k \Big) \cdot \one + 
			\sum_{j=1}^\ell
			\Big(
			\sum_{k=1}^\ell Q_k \ R_j^k \Big) \cdot \Theta_j 
			\Big]
		\end{align*}
		Choosing $\eta=\one$ or $\Theta_j$ respectively concludes the desired equations in (\ref{wall_cross_eq}).
	\end{proof}

	\begin{cor}
		\label{wall_cross_formula_cor}
		$\phi(\ia')\subseteq \ia$, and $\phi(W')$ is contained in the coset $W+\ia$
	\end{cor}
	
	\begin{proof}
		This is immediate from Theorem \ref{wall_cross_formula_thm}.
	\end{proof}
	
	The outcomes of Theorem \ref{wall_cross_formula_thm} and Corollary \ref{wall_cross_formula_cor} ensures the following definition:
	
	\begin{defn}
		\label{phi_var_defn}
		The pre-transitioning morphism $\phi=\phi_\f$ in (\ref{phi_defn_eq}) induces an algebra endomorphism
		\begin{equation}
			\label{phi_var_defn_eq}
			\varphi=\varphi_\f: 
			\Lambda\langle H_1(L);\Delta\rangle/ \ia'\to\Lambda\langle H_1(L);\Delta\rangle /\ia
		\end{equation}
		Moreover, we have $\varphi(W')=W$.
		We call $\varphi=\varphi_\f$ the \textit{transitioning morphism} associated to $\f$.
	\end{defn}

	\subsection{A canceling trick for ud-homotopy theory}
	We assume $\f_0  \simud \f_1$ for two $A_\infty$ homomorphisms $\f_0: \m'\to \m$ and $\f_1:\m'\to \m$. Namely, they are ud-homotopic to each other in the sense of (\ref{ud_sim_eq}).

	Our goal is to compare $\phi_{\f_0}$ and $\phi_{\f_1}$, so we aim to examine the difference between the two transitioning phase series $P_{\f_0}$ and $P_{\f_1}$.
	Recall that by degree reason, any nonvanishing term $(\f_i)_{0,\beta}$ lies in $H^1(L)$.
	
	\begin{thm}[Canceling Trick]
		\label{canceling_trick_thm}
		Given $\f_0\simud\f_1$ as above, 
		\[
		P_{\f_1} - P_{\f_0}= \sum_\beta  T^{E(\beta)} Y^{\partial\beta} ( (\f_1)_{0,\beta}- (\f_0)_{0,\beta})
		\]
		is contained in the obstruction ideal $\ia$ associated to the target $A_\infty$ algebra $\m$.
	\end{thm}
	
	\begin{proof}
		Since $\f_0\simud\f_1$, it follows from the definition around (\ref{ud_sim_eq}) that there are families of operator systems $\f_s$ and $\h_s$, for $s\in\oi$, in $\CC$ with the conditions (a), (b), (c), (d) therein.
		First, applying the condition (b) yields that for any $\beta$,
		\begin{align*}
			& 
			(\f_1)_{0,\beta}- (\f_0)_{0,\beta}
			=
			\int_0^1 \tfrac{d}{ds} (\f_s)_{0,\beta} ds \\
			&
			=
			\int_0^1ds \sum_{\beta_1+\beta_2=\beta}  (\h_s)_{1,\beta_1} (\m_{0,\beta_2})  + \sum_{\beta=\gamma+\alpha+\sum_i \beta_i+ \sum_j\beta'_j}
			\m'_{r+t+1,\gamma} \big(  (\f_s)_{0,\beta_1},\dots, (\f_s)_{0,\beta_r},  (\h_s)_{0,\alpha}, (\f_s)_{0,\beta_1'}, \dots, (\f_s)_{0,\beta'_t} \big)
		\end{align*}
		Using (d) implies that $(\h_s)_{0,\alpha}\in H^{-1-\mu(\alpha)}(L)$, so the second summation does not have any contribution since $\mu(\alpha)\geqslant 0$.
		The condition (c) also implies $(\h_s)_{1,\beta_1}(\one)=0$.
		By Definition \ref{W_Q_defn} and (\ref{W_Q_expand_eq}), if we let $W$ and $Q_j$'s are the superpotential and obstruction series associated to $\m$, then $\sum T^{E(\beta_2)} Y^{\partial\beta_2} \m_{0,\beta_2}= W \ \one +\ \sum Q_j \ \Theta_j$.
		Recalling the notation (\ref{R_f_notation}), we finally obtain
		\begin{align*}
			P_{\f_1}-P_{\f_0}
			&
			=
			\int_0^1 ds \sum_{\beta_1} T^{E(\beta_1)} Y^{\partial\beta_1} (\h_s)_{1,\beta_1} \big( W \ \one+ \sum_j Q_j \  \Theta_j \big)   
			=
			\sum_j Q_j \cdot  {\textstyle \int_0^1 R_{\h_s}^j ds }
		\end{align*}
		Since $Q_j$'s are the generator of the obstruction ideal, we complete the proof.
	\end{proof}

	\begin{cor}
		\label{phi_f_ud_inv_cor}
		$\phi_{\f_0}\equiv \phi_{\f_1}$ mod $\ia$. In particular, $\varphi_{\f_0}=\varphi_{\f_1}$.
	\end{cor}
	
	\begin{proof}
		In view of (\ref{phi_defn_eq}), we first consider $\langle \alpha, P_{\f_1} - P_{\f_0} \rangle$ for an arbitrary $\alpha\in H_1(L)$.
		By Theorem \ref{canceling_trick_thm}, we can write $\langle \alpha, P_{\f_1} - P_{\f_0} \rangle = \sum_j Q_j \cdot A_j$ for some strictly convergent formal power series $A_j=A^\alpha_j$ that may depend on $\alpha$. Note that the obstruction ideal $\ia$ is finitely generated by $Q_j$'s.
		Then,
		\begin{align*}
			\textstyle
			\exp\langle \alpha, P_{\f_1} -P_{\f_0}\rangle = \exp \big( \sum_{j=1}^\ell  Q_j \ A_j \big)
			=
			\prod_{j=1}^\ell \exp (Q_j A_j)
			=
			\prod_{j=1}^\ell \left(
			1+ Q_j A_j + \frac{1}{2!} Q_j^2 A_j^2 +\cdots 
			\right)
		\end{align*}
		It follows that $1 -\exp\langle \alpha, P_{\f_1} -P_{\f_0} \rangle \in \ia$.
		Thus, by the defining formula of $\phi_\f$ in (\ref{phi_defn_eq}), 
		\begin{align*}
			\phi_{\f_0}(Y^\alpha)-\phi_{\f_1}(Y^\alpha) = Y^\alpha \exp\langle\alpha, P_{\f_0}\rangle \Big( 1- \exp\langle \alpha , P_{\f_1}-P_{\f_0}\rangle \Big) \in \ia
		\end{align*}
		Finally, by (\ref{phi_var_defn_eq}), we also conclude $\varphi_{\f_0}=\varphi_{\f_1}$.
	\end{proof}

	\subsection{Invariance of proper unobstructedness}
	\label{ss_inv_proper_unobstructedness}
	
	Now, let's go back to the context of Situation \ref{situation_A_inf_alg_UD}. Assume we opt for an alternative almost complex structure $J'$, a different metric $g'$, and a distinct virtual fundamental chain $\Xi'$ on the moduli system $\mathbb M(J', L)$. This similarly leads to the derivation of $\check \m' = \check \m^{\Xi'}$ and $\m' = \m^{g',  \Xi'}$, which may differ from the original $\check \m = \check \m^{\Xi}$ and $\m = \m^{g,  \Xi}$.

	Choosing a path $\pmb J$ from $J$ to $J'$ and applying the second bullet of Axiom \ref{axiom_J}, there exists a pseudo-isotopy $\check \M$ on $\OL_\oi$ that restricts to $\check \m $ and $\check\m' $ at the two ends.
	Further choosing a path $\pmb g$ from $g$ to $g'$ and employing Proposition \ref{pseudo_isotopy_minimal_model_prop} yields a pseudo-isotopy $\M$ on $\HL_\oi$, connecting $\m $ to $\m' $.
	By Theorem \ref{from_pseudo_to_A_homo_thm}, the canonical construction associates the pseudo-isotopy $\M$ with an $A_\infty$ homotopy equivalence 
	\[
	\mC: \m\to \m'
	\]
	Recall that $\mC_{0,0}=0$ and $\mC_{1,0}=\id$ by construction.
	It lies in category $\UD$ by Proposition \ref{integral_UD_prop}. 
	By Proposition \ref{whitehead_prop}, there exists a ud-homotopy inverse within $\UD$, denoted by $\mC^{-1}: \m' \to \m$.

	Let $\mathfrak a$ and $\mathfrak a'$ be the obstruction ideals associated to $\m$ and $\m'$ respectively.

	\begin{thm}
		\label{proper_unobstructedness_invariance_thm}
		$\mathfrak a = 0$ if and only if $\mathfrak a' = 0$ within the affinoid algebra $\Lambda\langle H_1(L); \Delta \rangle$.
	\end{thm}

	\begin{proof}
		Let $A$ denote this affinoid algebra $\Lambda\langle H_1(L);\Delta\rangle$.
		Invoking Definition \ref{phi_var_defn}---applicable thanks to Theorem \ref{wall_cross_formula_thm}---we can find algebra morphisms $\varphi_{\mC}: A/\ia' \to A/\ia$ and $\varphi_{\mC^{-1}}: A/\ia \to A/\ia'$ associated to $\mC$ and $\mC^{-1}$ respectively.
		Since $\mC^{-1}\diamond \mC\simud \id$ and $\mC\diamond \mC^{-1}\simud \id$, using Corollary \ref{phi_f_ud_inv_cor} implies that $\varphi_{\mC\diamond\mC^{-1}}=\varphi_\id =\id$ and $\varphi_{\mC^{-1}\diamond \mC}=\varphi_{\id}=\id$.
		Moreover, using Corollary \ref{phi_fg_cor} implies that $\varphi_{\mC} \circ \varphi_{\mC^{-1}}=\varphi_{\mC\diamond \mC^{-1}} =\id $ and $\varphi_{\mC^{-1}} \circ \varphi_{\mC}=\varphi_{\mC^{-1}\diamond \mC}=\id$.
		Therefore, $\varphi_\mC$ is an algebra isomorphism from $A/\ia'$ to $A/\ia$ with the inverse morphism $\varphi_{\mC^{-1}}$.
		Finally, based on the standard knowledge (e.g. \cite[p300]{BGR} and \cite{EKL}) about the dimensions of the affinoid algebras $A/\ia$ and $A/\ia'$ or about that of the corresponding affinoid spaces, namely, the zero loci of $\ia$ and $\ia'$, one can deduce that $\ia=0$ if and only if $\ia'=0$.
	\end{proof}
	
	\begin{rmk}
		A basic fact is that affinoid algebras are Noetherian. Accordingly, the dimension of an affinoid algebra, defined as its Krull dimension or via Noether Normalization (similar to classical algebraic geometry), is always finite. Two affinoid algebras are isomorphic if there exists an algebra isomorphism between them.
		It’s not hard to verify that dimension is preserved under such isomorphisms. We aim to apply this to our case $A/\ia\cong A/\ia'$. However, it is crucial that this isomorphism preserves multiplications and holds within the category of affinoid algebras. This is precisely why we’ve taken care to lay the groundwork beforehand, extending beyond the \emph{set-theoretic} Maurer-Cartan framework for (weak) bounding cochains; see \cite[Corollary 4.3.14]{FOOOBookOne}.
	\end{rmk}

	Accordingly, the following definition makes sense.

	\begin{defn}[Definition \ref{proper_unobstructed_introduction_defn}]
		\label{proper_unobstructed_L_defn}
		We say a Lagrangian submanifold $L$ is \textit{properly unobstructed} if some (or any) $A_\infty$ algebra $\m=\m^{g, \Xi}$ on $H^*(L)$ as in Situation \ref{situation_A_inf_alg_UD} is properly unobstructed in the sense of Definition \ref{obstruction_ideal_defn}.
		It is well-defined due to Theorem \ref{proper_unobstructedness_invariance_thm}.
	\end{defn}

	\begin{rmk}
		\label{cocycle_condition_rmk}
		Let's briefly explain the claim about the cocycle condition and 2-pseudo-isotopies, as mentioned in Section \ref{sss_inspiration_family_floer}. Roughly, the cocycle condition needs a simultaneous investigation of three distinct $A_\infty$ algebras, say $\m_i$, $\m_j$, and $\m_k$. Similar to the above approach, we can employ (1-)pseudo isotopies to derive $A_\infty$ homotopy equivalences, say $\mC_{ij}:\m_i\to \m_j$, $\mC_{jk}:\m_j\to \m_k$, and $\mC_{ik}:\m_i\to\m_k$.
		Denote the corresponding algebra morphisms by $\varphi_{ij}$, $\varphi_{jk}$, and $\varphi_{ik}$. It is initially unclear whether the equation $\varphi_{jk}\circ \varphi_{ij}=\varphi_{ik}$ is valid. This uncertainty arises from not knowing if the composition $\mC_{jk}\circ \mC_{ij}$ is ud-homotopic to $\mC_{ik}$, which is a prerequisite for applying Corollary \ref{phi_f_ud_inv_cor}. To address this issue, it becomes necessary to delve into the study of 2-pseudo-isotopies with boundary conditions specified by the 1-pseudo-isotopies for those $\mC_{ij}$'s. This introduces a substantial degree of complexity, as elaborated in \cite{Yuan_I_FamilyFloer}. Fortunately, for the scope of this paper, it is not required to engage in this intricate problem.
	\end{rmk}

	\subsection{Analytic continuation of proper unobstructedness}
	Suppose $\{L_s\}_{s\in\mathcal S}$ is a smooth family of embedded Lagrangian submanifolds in $(X,\omega)$ parameterized by a connected manifold $\mathcal S$. By Definition \ref{proper_unobstructed_L_defn}, the purpose of this section is the proof of our main result:
	
	\begin{thm}[Theorem \ref{Main_thm_this_paper}]
		\label{analytic_continuation_unobstr_thm}
		If $L_{s_0}$ is properly unobstructed for a fixed $s_0$, then $L_s$ is properly unobstructed for every $s\in\mathcal S$.
	\end{thm}
	
	In principle, the key is the invariance of proper unobstructedness established in Theorem \ref{proper_unobstructedness_invariance_thm}.
	Then, the proof follows immediately by combining it with the Fukaya's trick.
	Let's begin with a quick review of the latter; see also \cite{FuCyclic,Yuan_I_FamilyFloer}.

	Take two sufficiently adjacent Lagrangians $L$ and $\tilde L$, and suppose there is a small non-Hamiltonian isotopy $F$ within $X$ such that $F(L)=\tilde L$.
	We may assume that $F_*J$ is also $\omega$-tame since the tameness is an open condition.
	Taking a Weinstein neighborhood $U$ of $L$ that is symplectomorphic to a neighborhood of the zero section in $T^*L$, we may assume $\tilde L$ is the graph of a small closed one-form $\xi\in Z^1(L)$.
	For any $\beta\in H_2(X,L)$, we denote $\tilde\beta=F_*\beta\in H_2(X,\tilde L)$.
	By \cite[Lemma 13.5]{FuCyclic}, we have
	\[
	E(\tilde\beta)=E(\beta) + \langle \partial\beta , \xi \rangle 
	\]
	By taking $F$-related bases, we can identify $H^1(L)\cong\mathbb R^{m}$ and $H^1(\tilde L) \cong  \mathbb R^{m}$ such that the induced map $F^*:H^1(\tilde L)\to H^1(L)$ corresponds to the identity map on $\mathbb R^m$.

	Given $\Delta \subseteq \mathbb R^m$ that contains $[\xi]$, we set $\tilde \Delta= \Delta-[\xi]$.
	Then, there is a natural isomorphism
	\begin{equation}
		\label{isom_affinoid_with_Delta_eq}
		pt:   \Lambda\langle H_1(L) ;\Delta \rangle  \cong  \Lambda\langle H_1(\tilde L) ; \tilde \Delta \rangle
	\end{equation}
	defined by
	$
	Y^\alpha \xleftrightarrow{}   T^{\langle \alpha, \xi \rangle} Y^{\tilde \alpha}$ for $\alpha\in H_1(L)$.
	Here ``$pt$'' stands for ``parallel translation''.

	\begin{defn}
		Given $(\HL, \m)$ as above, there is a natural $F$-induced $A_\infty$ algebra $(H^*(\tilde L), \m^F)$ such that
		\[
		F^* \m^F_{k,\tilde\beta} (x_1,\dots, x_k)=   \m_{k,\beta}(F^* x_1,\dots, F^* x_k)
		\]
		for $x_1,\dots, x_k\in H^*(\tilde L)$.
		Clearly, $\m^F$ is an object in $\UD$ if so is $\m$. In particular, when $k=0$, we have
		\[
		F^*\m_{0,\beta}^F= \m_{0,\beta}
		\]
	\end{defn}
	
	\begin{rmk}
		The definition of $\m^F$ seems almost algebraic, but it also has specific meanings from moduli space geometry. This is the content of Fukaya's trick \cite{FuCyclic}. We also refer to \cite{Yuan_I_FamilyFloer} for a detailed discussion.
		In brief, it transfers a small isotopy of Lagrangian submanifolds to a small perturbation of $J$, as there is a canonical identification
		\[ \mathcal M_{k+1,\beta}(J,L) \cong \mathcal M_{k+1,\tilde \beta} ( F_*J, \tilde L)
		\]
		Given any $J$-holomorphic stable map $u$ bounded by $L$, the map $F\circ u$ is $F_*J$-holomorphic and bounded by $\tilde L$.
		In view of Axiom \ref{axiom_J}, this means that a virtual fundamental chain $\Xi$ for the moduli space system $\mathbb M(J,L)$ induces a virtual fundamental chain $F_*\Xi$ for the moduli space system $\mathbb M(F_*J,\tilde L)$.
	\end{rmk}

	Let $(\Theta_i)$ be a basis of $H^2(L)$ as before.
	Let $W$ and $Q_i$'s be the superpotential and obstruction series associated with $\m$ as in Definition \ref{W_Q_defn} and \ref{obstruction_ideal_defn}.
	The map $F$ clearly relates the constant-one generators of $H^0(L)$ and $H^0(\tilde L)$ and also induces a basis $(\tilde \Theta_i)$ of $H^2(\tilde L)$.
	Let $W^F$ and $Q^F_i$'s be the superpotential and obstruction series associated with $\m^F$.
	The obstruction ideals $\ia$ and $\ia^F$ associated with $\m$ and $\m^F$ are generated by $Q_i$'s and $Q_i^F$'s respectively.
	
	\begin{prop}
		\label{proper_unobstructedness_Fuk_trick_prop}
		$W^F=  pt (W) $ and $Q^F_i=  pt  (Q_i)$. In particular, $\ia=0$ if and only if $\ia^F=0$.
	\end{prop}
	\begin{proof}
		This is straightforward.
		We only check it for $Q^F_i$.
		\begin{align*}
			Q_i^F 
			&
			=
			\sum T^{E(\tilde\beta)} Y^{\partial\tilde\beta} \Big[ \tilde\Theta_i \, ,  \,   \m^F_{0,\tilde\beta} \Big]
			=
			\sum T^{E(\beta)} \cdot T^{\langle \partial\beta, \xi\rangle} Y^{\partial\tilde\beta} \cdot
			\Big[ \Theta_i \, ,  \,   F^* \m^F_{0,\tilde\beta} \Big] \\
			&
			=
			\sum T^{E(\beta)} \cdot  {pt} (Y^{\partial\beta}) \cdot 
			\Big[
			\Theta_i,  \m_{0,\beta} 
			\Big]
			=
			pt (Q_i)
		\end{align*}
	\end{proof}

	\vspace{1em}

	\begin{proof}[Proof of Theorem \ref{analytic_continuation_unobstr_thm}]
		We consider the subset
		\[
		\mathcal A=\{ s\in\mathcal S \mid L_s \ \text{is properly unobstructed} \ \}
		\]
		Since $\mathcal S$ is connected and $\mathcal A$ is non-empty, it suffices to verify that $\mathcal A$ is both open and closed in $\mathcal S$.
		
		We first show the openness. Given $s_0\in\mathcal A$, we take a small Weinstein neighborhood $U$ of $L_{s_0}$. Then, we can take a small open neighborhood $V$ of $s_0$ in $\mathcal S$ such that for any $s\in V$, $L_{s}\subseteq U$ can be realized as the graph of a closed one-form $\xi_{s}$.
		Let $c_{s_0}$ denote the estimate constant for the reverse isoperimetric inequality for $L_{s_0}$. Namely, a holomorphic curve $u$ bounded by $L_{s_0}$ satisfies that $E(u)\ge c_{s_0} \ell(\partial u)$ for the length $\ell(\partial u)$ of $\partial u$ (cf. Proposition \ref{strictly_convergent_prop}).
		Shrinking $V$ if necessary, we may assume that the norm of $\xi_{s}$ for $s\in V$ is controlled by this estimate constant $c_{s_0}$.
		Given $s\in V$, we can choose a small isotopy $F$ such that $F(L_{s_0})=L_s$.
		Since $L_{s_0}$ is properly unobstructed, using Proposition \ref{proper_unobstructedness_Fuk_trick_prop} yields that $L_s$ is also properly unobstructed. Hence, $V\subseteq \mathcal A$, and $\mathcal A$ is open.
		
		We next show the closedness. This is basically the consequence of the uniform version of reverse isoperimetric inequalities, which is proved in \cite{Yuan_I_FamilyFloer} by modifying Duval's arguments in \cite{ReverseII} (see also Remark \ref{reverse_rmk_proof} below). Namely, we can make the aforementioned estimate constant $c_s$ uniform if $s$ ranges over a compact subset in $\mathcal S$.
		Specifically, let $(s_n)_{n\ge 0}$ be a sequence in $\mathcal A$ convergent to a point $s_\infty$ in $\mathcal S$.
		Our goal is to show $s\in\mathcal A$.
		Indeed, we take open neighborhoods $V_n$ of each $s_n$ as above. But, since we can find a uniform estimate constant $c$ for the reverse isoperimetric inequalities over a fixed compact domain in $\mathcal S$ that contains the closure of this sequence, we may further assume that the diameter of $V_n$ is uniformly bounded below and thus $s_\infty \in V_n$ for sufficiently large $n$.
		Using Proposition \ref{proper_unobstructedness_Fuk_trick_prop} similarly concludes that $s_\infty\in\mathcal   A$. Hence, $\mathcal A$ is closed.
		
		To sum up, $\mathcal A$ is a non-empty subset of $\mathcal S$ that is both open and closed. Because $\mathcal S$ is connected, we must have $\mathcal A=\mathcal S$. The proof is now complete.
	\end{proof}

	\begin{rmk}\label{reverse_rmk_proof}
		The difficulty in generalizing Duval's argument happens only when trying to include Lagrangian \textit{intersections}.
		Specifically, in the work of Chass\'e-Hicks-Nho \cite{chasse2023reverse}, a Lagrangian intersection, say $L_1\cap L_2$, is required to be \emph{locally standard} in the sense of \cite[Definition II.1]{chasse2023reverse}.
		However, in our context, we focus solely on isotopies of smooth Lagrangians, not their intersections. Therefore, there is no essential difficulty in generalizing Duval's argument in our case.
	\end{rmk}

	\section{Applications}
	Finally, let's expand on the details for the applications discussed in the introduction.

	\begin{proof}[Proof of Proposition \ref{Application_4mfd_prop}]
		Fix a Lagrangian torus $L$ in $X$. By \cite[Theorem A]{dimitroglou2016lagrangian}, one can find a Lagrangian isotopy $(L_s)_{0\leqslant s\leqslant 1}$ such that $L_0$ is a standard monotone Lagrangian torus and $L_1=L$.
		Remark also that, in the case $X=\mathbb {CP}^2$ or $S^2\times S^2$, by \cite[Theorem C]{dimitroglou2016lagrangian}, we can find a Hamiltonian isotopy to make the Lagrangian isotopy $(L_s)$ placed inside the complement $X\setminus D_\infty$, where $D_\infty$ is the divisor at infinity when $X=\mathbb {CP}^2$ and is the union of two holomorphic lines when $X=S^2\times S^2$.
		Note that $L_0$ is graded in the complement of $D_\infty$. By \cite[Lemma 3.1]{AuTDual}, we can ensure the non-negativity of Maslov indices for nontrivial holomorphic disks.
		Since the monotone condition implies that $L_0$ does not bound nonconstant Maslov-0 disks, we see that those curvature terms $\m_{0,\beta}$ with $\mu(\beta)=0$ are all vanishing and that the formal power series $Q_i$'s are varnishing as well (Definition \ref{W_Q_defn}).
		Thus, $L_0$ is properly unobstructed and so is $L=L_1$ by Theorem \ref{Main_thm_this_paper} (i.e. \ref{analytic_continuation_unobstr_thm}).
		
	\end{proof}

	\begin{proof}[Proof of Proposition \ref{Application_AAK_prop_intro}]
		We need to first recall the context of \cite[3.1, 3.2]{AAK_blowup_toric} or \cite[3.3.2]{Au_Special} for a typical case.
		Let $V$ be a toric manifold over $\mathbb C$. Let $H=\{f=0\}\subset V$ be an appropriate hypersurface with some technical conditions \cite[Definition 3.3]{AAK_blowup_toric}.
		Let $X$ be the blow-up of $V\times \mathbb C$ along $H\times \{0\}$.
		Let $D$ be the proper transform of the toric anticanonical divisor of $V\times \mathbb C$, and we set $X^0=X\setminus D$.
		For example, when $V=(\mathbb C^*)^n$ and $H=\{1+x_1+\cdots +x_n=0\}$, we have $X^0=\{(x_1,\dots, x_n,y,z)\in (\mathbb C^*)^n\times \mathbb C^2 \mid yz=x_1+\cdots +x_n +1\}$ \cite[9.1]{AAK_blowup_toric}.
		
		One can find a Lagrangian fibration $\pi:X^0\to B=\mathbb R^n\times \mathbb R_+$ \cite[Definition 4.4]{AAK_blowup_toric} such that the singular locus is $B^{sing}=\Pi'\times \{\epsilon\}$ where $\Pi'$ is a perturbation of the amoeba of the hypersurface $H$ \cite[(4.4) and Fig. 2]{AAK_blowup_toric}.
		The wall region, consisting of all base points over which the fibers of $\pi: X^0 \to B$ bound Maslov-0 disks, has a non-empty complement in the base \cite[Corollary 5.2]{AAK_blowup_toric}. In other words, the fibers outside of this wall region are tautologically unobstructed.
		Although the wall region may have a non-empty interior in the base manifold \cite[Proposition 5.1]{AAK_blowup_toric}, it follows from Theorem \ref{Main_thm_this_paper} that those fibers over the wall region are still (properly) unobstructed.
	\end{proof}

	\appendix
	
	\section{Kuranishi structures and virtual fundamental chains}
	\label{s_Kuranishi}
	
	In this appendix, let's elaborate on Theorem \ref{A_infinity_algebra_up_to_pseudo_isotopy_thm} or Axiom \ref{axiom_J}.
	Assuming these results, the appendix is not logically necessary for the main result, but we hope our writings below may be a helpful exposition for the basic logic flow of the Kuranishi structure theory and its application.
	The main references are \cite{FOOODiskOne,FOOODiskTwo,FOOO_Kuranishi,FuCyclic}.
	Our primary focus is on \cite[Theorem 21.35]{FOOO_Kuranishi}, which deals with a quite broad and abstract framework. When it comes to the specific moduli spaces in practical applications \cite{FOOODiskOne,FOOODiskTwo}, we aim to present a streamlined outline of the theory of Kuranishi structure.
	We attempt to make definitions and statements concise and precise, but providing complete details of proofs is clearly beyond our ability. Where full detail is not feasible, we will at least provide accurate references to the literature.

	\subsection{Obstruction bundle data}
	The moduli space $\mathcal M:=\mathcal M_{k+1,\beta}(J,L)$ in (\ref{moduli_single}) carries a Kuranishi structure.
	It can be canonically constructed by the so-called obstruction bundle data. 
	
	Let $\mathbf p_a$ be an irreducible component of a fixed $\mathbf p\in\mathcal M$. The linearization of the Cauchy-Riemann operator defines a linear elliptic operator, denoted as $D_{\mathbf p_a} : W_{\mathbf p_a}\to L_{\mathbf p_a}$, where the source and target are appropriate Sobolev spaces of sections of the pullback bundles of $TX$ with boundary values in the pullbacks of $TL$ \cite[(5.1)]{FOOODiskOne}.
	Let $W_{\mathbf p}$ be the Hilbert subspace of the direct sum of $W_{\mathbf p_a}$ with compatibility at all nodal points. Let $L_{\mathbf p}$ be the direct sum of all $L_{\mathbf p_a}$.
	Then, the collection $D_{\mathbf p_a}$ gives rise to a Fredholm operator
	$D_{\mathbf p}: W_{\mathbf p}\to L_{\mathbf p}$ (see \cite[(5.5)]{FOOODiskOne} for the detailed descriptions).
	Furthermore, each evaluation map $\ev_i$ in (\ref{ev_eq}) induces a map $\mathrm{EV}_i: W_{\mathbf p} \to T_{\ev_i(\mathbf p)}L$ defined by taking the fiber over the $i$-th marked point \cite[7.11]{FOOODiskOne}.
	
	Define $\mathcal X := \mathcal X_{k+1,\beta}(J,L)$ to be the set of $[\Sigma,\mathbf u,\mathbf z]$ in almost the same way as $\mathcal M$ (\ref{moduli_single}) but with more flexible conditions: $\mathbf u$ is not required to be $J$-holomorphic, but $\mathbf u$ is of $C^2$ class on each irreducible component. Note that $\mathcal M\subseteq \mathcal X$ is a topological space, but $\mathcal X$ is simply a set. The pair $(\mathcal X,\mathcal M)$ admits the structure of \textit{partial topology} in the following sense:
	There is a collection of subsets $B_\epsilon(\mathbf p) \subseteq \mathcal X$ for all $\mathbf p\in\mathcal M$ and $\epsilon>0$ such that \cite[Proposition 4.3]{FOOODiskOne}:
	
	\begin{itemize}[itemsep=0pt]
		\item The collection $B_\epsilon(\mathbf p)\cap \mathcal M$ forms a basis of the standard stable map topology of $\mathcal M$.
		\item For any $\mathbf q\in B_\epsilon(\mathbf p)\cap \mathcal M$, there exists a $\delta>0$ such that $B_\delta(\mathbf q) \subset B_\epsilon (\mathbf p)$.
		\item The subset $B_\epsilon(\mathbf p)$ is monotone with $\epsilon$, and the intersection of $B_\epsilon(\mathbf p)$ for all $\epsilon$ is $\{\mathbf p\}$.
	\end{itemize}
	
	A subset $\mathscr U\subseteq \mathcal X$ is a \textit{(partial) neighborhood} of $\mathbf p\in\mathcal M$ if $\mathscr U$ is contained in $B_\epsilon(\mathbf p)$ for some $\epsilon$.
	Roughly speaking, these $B_\epsilon(\mathbf p)$ behave much like the usual open neighborhoods in a topological space but with the distinction here that they must contain points in $\mathcal M$.
	
	By \textit{obstruction bundle data} of $\mathcal M$, we mean a rule $\mathscr O$ that assigns to each $\mathbf p\in\mathcal M$ a partial neighborhood $\mathscr U_{\mathbf p}$ and assigns to each $\mathbf p\in\mathcal M$ and $\mathbf x\in\mathscr U_{\mathbf p}$ a finite-dimensional linear subspace $E_{\mathbf p}(\mathbf x)$ within $C^2( u_\mathbf x^*TX\otimes \Lambda^{0,1})$, where $u_{\mathbf x}$ is a representative of $\mathbf x$ in $\mathcal X$, such that \cite[5.1]{FOOODiskOne}
	
	\begin{itemize}[itemsep=0pt]
		\item[$\circ$]
		Each $E_{\mathbf p}(\mathbf x)$ is supported away from nodal points, 
		is invariant under the choice of representative $u_{\mathbf x}$ of $\mathbf x$, and
		is smoothly dependent on $\mathbf x$.
		\item[$\circ$] The following transversality condition holds: the sum of $E_{\mathbf p}(\mathbf p)$ and $\mathrm{Im}(D_{\mathbf p})$ is the whole $L_{\mathbf p}$.
		\item[$\circ$] $E_{\mathbf p}(\mathbf x)$ is semi-continuous on $\mathbf p$ in the sense that if $\mathbf q\in \mathscr U_{\mathbf p}$ and $\mathbf x\in\mathscr U_{\mathbf p}\cap \mathscr U_{\mathbf q}$, $E_{\mathbf q}(\mathbf x)\subseteq E_{\mathbf p}(\mathbf x)$.
	\end{itemize}
	Here the smooth dependence on $\mathbf x$ means that by joining $\mathbf p$ and $\mathbf x$ via minimal geodesics and by taking parallel transports, these $E_{\mathbf p}(\mathbf x)$ can be viewed as a collection of subspaces in the same $C^2(u^*_{\mathbf p}TX\otimes \Lambda^{0,1})$ with support away from nodal points parameterized by $\mathbf x$; then the smooth dependence on $\mathbf x$ makes sense. By shrinking $\mathscr  U_{\mathbf p}$ if necessary, we may assume that the distance between $\mathbf p$ and $\mathbf x$ is controlled by the injectivity radius of a Riemannian metric to make the above parallel transport works.
	Furthermore, we say that a choice of obstruction bundle data $\{E_{\mathbf p}(\mathbf x)\}$ of $\mathcal M=\mathcal M_{k+1,\beta}$ satisfies the \textit{mapping transversality condition for $\ev_i$} (\ref{ev_eq}) if \cite[7.11]{FOOODiskOne}
	
	\begin{itemize}
		\item[$\circ$] for any $\mathbf p=[\Sigma, \mathbf u,\mathbf z]\in\mathcal M$, the restriction of the above $\mathrm{EV}_i$ on $D_{\mathbf p}^{-1}\big( E_{\mathbf p}(\mathbf p)\big)$ is surjective.
	\end{itemize}

	Fix $\mathbf p\in\mathcal M$. By \textit{(local) obstruction bundle data at $\mathbf p$} we mean a collection $\{E_{\mathbf p}(\mathbf x): \mathbf x\in \mathscr U_{\mathbf p} \}$ for some partial neighborhood $\mathscr U_{\mathbf p}$ with all the above conditions except the semi-continuity \cite[5.1]{FOOODiskOne}.

	\begin{thm}\label{Obstruction_bundle_data_single_thm}
		\emph{\cite[11.1]{FOOODiskOne}} \ 
		There exists obstruction bundle data $\mathscr O=\{E_{\mathbf p}(\mathbf x) \}$ of the moduli space $\mathcal M=\mathcal M_{k+1,\beta}(J,L)$.
		We can also require $\mathscr O$ satisfies the mapping transversality condition for $\ev_0$.
	\end{thm}
	
	\begin{proof}[Sketch of proof]
		First, there always exists a choice $\{E^0_{\mathbf p}(\mathbf x)\mid \mathbf x\in\mathscr U_{\mathbf p}^0 \}$ of local obstruction bundle at each $\mathbf p\in\mathcal M$ for some partial neighborhood $\mathscr U_{\mathbf p}^0$. This is basically a consequence of the Fredholm condition of $D_{\mathbf p}$. Indeed, we may first choose $E_{\mathbf p}^0(\mathbf p)$ with the conditions and take parallel transports along minimal geodesics to define $E_{\mathbf p}^0(\mathbf x)$ \cite[11.2]{FOOODiskOne} for $\mathbf x \in\mathscr U_{\mathbf p}^0$.
		Shrinking $\mathscr U_{\mathbf p}^0$ if necessary we can further require that for any $\mathbf q\in \mathscr U_{\mathbf p}^0$ there is a partial neighborhood $\mathscr U_{\mathbf q;\mathbf p}\subseteq \mathscr U_{\mathbf p}^0$ of $\mathbf q$ in which $E_{\mathbf q;\mathbf p}(\mathbf x):=E_{\mathbf p}^0(\mathbf x)$ forms a choice of local obstruction bundle data at $\mathbf q$ \cite[11.4]{FOOODiskOne}.
		Take a compact subset $K(\mathbf p)$ of $\mathscr U_{\mathbf p}^0\cap \mathcal M$ for each $\mathbf p \in \mathcal M$ such that $K(\mathbf p)$ is the closure of an open neighborhood of $\mathbf p$ for the stable map topology in $\mathcal M$.
		Since $\mathcal M$ is compact, we can take a finite subset $\{\mathbf p_1,\dots, \mathbf p_N\}$ of $\mathcal M$ such that the union of $K (\mathbf p_i)$ for $1\leqslant i\leqslant N$ is the whole $\mathcal M$. For any $\mathbf q\in\mathcal M$, we set $I(\mathbf q)=\{i \mid \mathbf q \in K (\mathbf p_i) \}$.
		We can further perturb all (finitely many) $E_{\mathbf p_i}^0(\mathbf p_i)$ such that the induced $E_{\mathbf q; \mathbf p_i}(\mathbf x)$'s satisfy the condition that the sum of vector spaces $E_{\mathbf q; \mathbf p_i}(\mathbf q)$, $i\in I(\mathbf q)$, is a direct sum \cite[11.7]{FOOODiskOne}.
		Now, $\mathscr U_{\mathbf q}:= \bigcap_{i\in I(\mathbf q)} \mathscr U_{\mathbf q; \mathbf p_i}$ is a partial neighborhood at each $\mathbf q \in \mathcal M$.
		Define $E_{\mathbf q}(\mathbf x)=\bigoplus_{i\in I(\mathbf q)} E_{\mathbf q; \mathbf p_i}(\mathbf x)$ with $\mathbf x\in \mathscr U_{\mathbf q}$. One can finally verify that $\{E_{\mathbf q}(\mathbf x) : \mathbf q, \mathbf x\}$ is obstruction bundle data of $\mathcal M$.
	\end{proof}

	\subsection{Kuranishi structure}
	\label{ss_Kuranishi}
	We begin with abstract definitions.
	Abusing the notation, $\mathcal M$ may sometimes also denote an arbitrary separable metrizable topological space.
	
	A \textit{Kuranishi chart} of $\mathcal M$ is a tuple $\mathcal U=(U,\mathcal E, \psi, s)$ such that (i) $U$ is an effective orbifold; 
	(ii) $\mathcal E$ is an orbi-bundle on $U$; (iii) $s$ is a smooth section of $\mathcal E$; (iv) $\psi:s^{-1}(0)\to \mathcal M$ is a homeomorphism onto an open subset in $\mathcal M$; see \cite[3.1]{FOOO_Kuranishi}.
	The $U, \mathcal E, s, \psi$ are respectively called a Kuranishi neighborhood, an obstruction bundle, a Kuranishi map, and a parametrization.
	The \textit{dimension} of $\mathcal U$ means $\mathrm{vdim} \  \mathcal U= \dim  U- \mathrm{rank}\mathcal E$.
	An \textit{orientation} of $\mathcal U$ is a pair of orientations of $U$ and $\mathcal E$, and $\det \mathcal E^* \otimes \det TU$ is called its orientation bundle.
	If $\tilde U$ is an open subset of $U$, then there is an obvious way to define an {open subchart} $\mathcal U|_{\tilde U}$ of $\mathcal U$ by taking the restrictions.
	In general, $\mathcal M$ can be a compact metrizable space.
	
	An \textit{embedding of Kuranishi charts} from $\mathcal U=(U,\mathcal E, \psi, s)$ to $\mathcal U'= (U',\mathcal E', \psi', s')$ is a pair $\Phi=(\varphi, \hat\varphi)$ with the following properties:
	(i) $\varphi: U\to U'$ is an embedding of orbifolds; (ii) $\hat\varphi: \mathcal E\to\mathcal E'$ is an embedding of orbi-bundles over $\varphi$; (iii) $\hat\varphi \circ s = s' \circ \varphi$; (iv) $\psi' \circ \varphi =\psi$ holds on $s^{-1}(0)$; (v) For each $x\in U$ with $s(x)=0$, the derivative $D_{\varphi(x)} s'$ induces an isomorphism between $T_{\varphi(x)} U' / D\varphi_x (T_x U)$ and $\mathcal E'_{\varphi(x)} / \hat \varphi(\mathcal E_x)$; see \cite[3.2]{FOOO_Kuranishi}.

	A \textit{coordinate change} (in the weak sense) from $\mathcal U$ to $\mathcal U'$ is a triple $\Phi=(\tilde U, \varphi, \hat\varphi)$ such that (i) $\tilde U$ is an open subset of $U$; (ii) $(\varphi,\hat\varphi)$ is an embedding of Kuranishi charts from $\mathcal U|_{\tilde U}$ to $\mathcal U'$; if further satisfying $\psi (s^{-1}(0)\cap \tilde U) = \mathrm{Im}(\psi)\cap \mathrm{Im}(\psi')$, it is called a \textit{coordinate change in the strong sense} \cite[3.6]{FOOO_Kuranishi}.

	In our context, the Kuranishi neighborhood, i.e. the orbifold $U$, may allow boundaries and corners. If so, there is the so-called corner structure stratification $U=\bigsqcup_k S_k(U)$ where $S_k(U)$ is the closure of the set of points whose neighborhoods are open neighborhoods of $0$ within $([0,1)^k\times \mathbb R^{n-k}) / \Gamma$ where $\Gamma$ is an isotropy group \cite[4.13]{FOOO_Kuranishi}.
	Note that $S_k^\circ(U):=S_k(U)\setminus S_{k+1}(U)$ is a smooth orbifold without boundaries or corners and is always assumed to be effective \cite[4.14]{FOOO_Kuranishi}. Besides, for $\Phi$, we also require $\varphi(S_k(U))\subseteq S_k(U')$ and $S_k^\circ (U)=\varphi^{-1}(S_k^\circ(U'))$ \cite[(8.2)]{FOOO_Kuranishi}.

	A \textit{Kuranishi structure} $\widehat {\mathcal U}$ of a compact subset $Z$ within $\mathcal M$ is a rule that assigns to each $  p\in Z$ a Kuranishi chart $\mathcal U_{  p}=(U_{  p}, \mathcal E_{  p}, \mathfrak s_{  p}, \psi_{  p})$ (of $\mathcal M$) and to each pair $p\in Z$ and $q\in \mathrm{Im}(\psi_{  p})\cap Z$ a coordinate change $\Phi_{pq}=(U_{  p   q} , \varphi_{  p   q}, \hat \varphi_{pq})$ from $\mathcal U_q$ to $\mathcal U_p$ such that the following compatibility hold: \cite[3.9]{FOOO_Kuranishi}
	
	\begin{itemize}[itemsep=0pt]
		\item $p\in \mathrm{Im}(\psi_{  p})$ and $  q\in \psi_{  q} (s^{-1}_{  q}(0)\cap U_{  p   q})$.
		\item for $r \in  \psi_{  q} (s^{-1}_{  q}(0)\cap U_{pq})\cap Z$, we have $\varphi_{pr}=\varphi_{pq}\circ \varphi_{qr}$ and $\hat \varphi_{pr}=\hat\varphi_{pq}\circ\hat \varphi_{qr}$ on their domains.
		\item $\Phi_{pp}=(U_p,\id,\id)$ and $\dim\widehat{\mathcal U}:=\dim\mathcal U_p=\dim U_p-\mathrm{rank} \ \mathcal E_p$ is independent of $p$.
	\end{itemize}
	Note that $U_{pq}\subseteq  U_q$ and $\varphi_{pq}: U_{pq}\to U_p$, $\widehat\varphi_{pq}: \mathcal E_q|_{U_{pq}}\to \mathcal E_p $ are embeddings of orbifolds, orbi-bundles.
	A Kuranishi structure $\widehat{\mathcal U}$ is called \textit{orientable} if we can choose an orientation of each chart $\mathcal U_p$ preserved by coordinate changes \cite[3.10]{FOOO_Kuranishi}.
	Now, we call the triple $(\mathcal M, Z,\widehat{\mathcal U})$ a \textit{relative K-space} \cite[3.11]{FOOO_Kuranishi}.
	The compact subset $Z$ is called the \textit{support set} of $\widehat{\mathcal U}$ \cite[3.9]{FOOO_Kuranishi}.
	We often consider the case $Z=\mathcal M$, and then $(\mathcal M, \widehat{\mathcal U})$ is called a \textit{K-space}.
	
	A \textit{strongly continuous} map $\widehat{f}=\{f_p\}:\widehat{\mathcal U} \to M$ from a K-space $(\mathcal M, \widehat{\mathcal U})$ to a smooth manifold $M$ is a collection that assigns to each $p\in \mathcal M$ a continuous map $f_{p} :U_{p} \to M$ such that $ f_{p} \circ \varphi_{  p  q}=f_{ q}$. It is called \textit{strongly smooth} if each $f_{ p}$ is smooth and called \textit{weakly submersive} if each $f_{p}$ is a submersion onto the manifold $M$ \cite[3.40]{FOOO_Kuranishi}.


	\begin{thm}\label{Kuranishi_single_moduli_thm}
		\emph{\cite[1.1 or 7.1]{FOOODiskOne}} \  \  
		To a choice $\mathscr O=\{E_{\mathbf p}(\mathbf x) : \mathbf p\in\mathcal M, \ \mathbf x\in\mathscr U_{\mathbf p}\}$ of obstruction bundle data of moduli space $\mathcal M=\mathcal M_{k+1,\beta}(J,L)$, we can associate a Kuranishi structure (with corners) $\widehat {\mathcal U}$ on $\mathcal M$.
		All the evaluation maps $\ev_i: (\mathcal M,\widehat{\mathcal U})\to L$ $(0\leqslant i\leqslant k)$ are strongly smooth \cite[7.10]{FOOODiskOne}.
		Moreover, if $\mathscr O$ satisfies the mapping transversality condition for $\ev_0$, then $\ev_0$ is weakly submersive \cite[7.12]{FOOODiskOne}.
	\end{thm}
	
	\begin{proof}[Sketch of proof]
		First, the (local) obstruction bundle data $\{ E_{\mathbf p}(\mathbf x): \mathbf x\in\mathscr U_{\mathbf p}\}$ at $\mathbf p\in\mathcal M$ determines a Kuranishi chart $\mathcal U_{\mathbf p}=(U_{\mathbf p}, \mathcal E_{\mathbf p}, \mathfrak s_{\mathbf p}, \psi_{\mathbf p})$ at each $\mathbf p$ as follows \cite[Lemma 7.4]{FOOODiskOne}:
		Define the Kuranishi neighborhood $U_{\mathbf p} \subseteq \mathscr U_{\mathbf p}$ to be the subset of $\mathbf x\in\mathscr U_{\mathbf p}$ so that $\bar\partial u_{\mathbf x}$ lies in $E_{\mathbf p}(\mathbf x)$, which is independent of the choice of the representative $u_{\mathbf x}$ of $\mathbf x$.
		Define the obstruction bundle $\mathcal E_{\mathbf p}$ as $\bigsqcup_{\mathbf x\in  U_{\mathbf p}}  E_{\mathbf p}(\mathbf x)$ modulo the automorphisms of $\mathbf x$'s.
		Define the Kuranishi map $\mathfrak s_{\mathbf p} : U_{\mathbf p} \to \mathcal E_{\mathbf p}$ by sending $\mathbf x$ to  $\bar\partial u_{\mathbf x}$.
		
		Second, we construct coordinate changes. Assume $\mathbf q \in \mathscr U_{\mathbf p} \cap \mathcal M$. Let $\mathcal U_{\mathbf p}=(U_{\mathbf p}, \mathcal E_{\mathbf p}, \mathfrak s_{\mathbf p}, \psi_{\mathbf p} )$ and $\mathcal U_{\mathbf q}=(U_{\mathbf q}, \mathcal E_{\mathbf q}, \mathfrak s_{\mathbf q}, \psi_{\mathbf q})$
		be the Kuranishi charts at $\mathbf p$ and $\mathbf q$ as above.
		Consider $U_{\mathbf p \mathbf q} = \mathscr U_{\mathbf p} \cap U_{\mathbf q}   \subseteq \mathscr U_{\mathbf p}\cap \mathscr U_{\mathbf q}$. For any $\mathbf x$ in $U_{\mathbf p\mathbf q}$, the semi-continuity condition implies that $E_{\mathbf q}(\mathbf x)\subseteq E_{\mathbf p}(\mathbf x)$, 
		and the definition of $U_{\mathbf q}$ implies that $\bar\partial u_{\mathbf x}\in E_{\mathbf q}(\mathbf x) \subseteq E_{\mathbf p}(\mathbf x)$.
		By definition again, $\mathbf x \in U_{\mathbf p}$. Thus, $U_{\mathbf p\mathbf q}\subseteq U_{\mathbf p}$, and let $\varphi_{\mathbf p \mathbf q}$ denote the inclusion map.
		Similarly, one can define an inclusion map $\hat\varphi_{\mathbf p\mathbf q}: \mathcal E_{\mathbf p}|_{U_{\mathbf p \mathbf q}} \to \mathcal E_{\mathbf p}$. Then, one can verify that $\Phi_{\mathbf p\mathbf q}:= (U_{\mathbf p \mathbf q} , \varphi_{\mathbf p \mathbf q} , \hat \varphi_{\mathbf p \mathbf q})$ for all $\mathbf p, \mathbf q$ forms a coordinate change from $\mathcal U_{\mathbf p}$.
		
		The construction of Kuranishi charts and coordinate changes has been purely set-theoretic so far. However, the smoothness of obstruction bundle data enables the use of gluing analysis and exponential decay estimates \cite{FOOOTech,FOOOExponentialDecay} to further verify that the \textit{same} set-theoretic data $\{\mathcal U_{\mathbf p}, \Phi_{\mathbf p\mathbf q}\}$ is of $C^\infty$ class. Detailing this process is beyond our current scope. 
		For a comprehensive explanation, we direct readers to \cite[Sec 9, 10]{FOOODiskOne} for showing it is of the $C^n$ class and to \cite[Sec 12]{FOOODiskOne} for the $C^\infty$ class.
		Finally, the evaluation maps $\ev_i$ (\ref{ev_eq}) is naturally extended from $\mathcal M$ to $\mathcal X$. Given that the Kuranishi neighborhood $U_{\mathbf p}$ is by definition a subset of $\mathcal X$, this extension allows $\ev_i$ to naturally induce a corresponding map $\ev_{i,\mathbf p}: U_{\mathbf p} \to L$. We can also verify the desired conditions for them \cite[7.10, 7.12]{FOOODiskOne}.
	\end{proof}
	
	So far, given obstruction bundle data (from Fredholm theory), we can get a Kuranishi structure on the \textit{concrete} moduli spaces. Thereafter, what remains is mostly \textit{abstract} perturbation theories.

	\subsection{CF-perturbations on a Kuranishi chart}  
	\label{ss_CF_chart}
	Fix a Kuranishi chart $\mathcal U=(U,\mathcal E, \psi,s)$ of $\mathcal M$.
	By an \textit{orbifold chart} of the orbi-bundle $\mathcal E\to U$ we mean a quintuple $\mathfrak V=(V,E,\Gamma,\phi,\hat\phi)$ where (i) $\Gamma$ is a finite isotropy group acting smoothly and effectively on a manifold $V$ and acting linearly on a finite-dimensional vector space $E$, (ii) $\phi : V \to U$ is a $\Gamma$-equivariant map that induces a homeomorphism from $V / \Gamma$ onto an open subset $\phi(V)$ of the orbifold $U$, (iii) $\hat\phi: V\times E\to U$ is a $\Gamma$-equivariant map that covers $\phi$ and induces a homeomorphism from $(V\times E)/\Gamma$ onto an open subset of $\mathcal E$ \cite[23.18]{FOOO_Kuranishi}.

	By CF-perturbation $\mathcal S= (W, \hat W, \omega, \{\mathfrak s^\epsilon\}_{\epsilon\in (0,\epsilon_0]} )$ of $\mathcal U$ on an orbifold chart $\mathfrak V$, we mean \cite[7.4]{FOOO_Kuranishi}:
	\begin{itemize}[itemsep=0pt]
		\item 
		$\hat W$ is a finite-dimensional oriented vector space on which $\Gamma$ acts linearly, and $W$ is a $\Gamma$-invariant open neighborhood of $0$ in $\hat W$.
		
		\item  $\mathfrak s^\epsilon : V \times W \to E$ is smooth, $\Gamma$-equivariant, and smoothly dependent on $\epsilon \in (0,\epsilon_0]$.
		\item $\lim_{\epsilon \to 0} \mathfrak s^\epsilon (y, \xi) = s(y)$ in compact $C^1$-topology on $V\times W$.
		\item Let $\vol_W$ be a volume form on $W$. Let $\omega=|\omega| \vol_W$ be a $\Gamma$-invariant compactly-supported smooth differential form on $W$ such that $|\omega|$ is a non-negative function and $\int_W \omega =1$.
	\end{itemize}
	
	For two such $\mathcal S_1$ and $\mathcal S_2$ on the same orbifold chart $\mathfrak V$, we introduce a \textit{relation} by declaring $\mathcal S_1,\mathcal S_2$ are related
	if there exists ``a common enlargement'' $\mathcal S_0$ in the sense that for $i=1,2$, there exists a $\Gamma$-equivariant linear projection $\Pi_i:\hat W_0\to \hat W_i$ such that $\Pi_i(\hat W_0)=\hat W_i$, $\Pi_i !(\omega_0)=\omega_i$ (for the usual pushout or integration along fibers), and $\mathfrak s^\epsilon_i(y, \Pi_i(\xi))=\mathfrak s_0^\epsilon(y,\xi)$. We then take the smallest \textit{equivalence relation} that contains this relation, which is equivalently realized by the explicit `zig-zag' in \cite[7.6]{FOOO_Kuranishi}.
	For a fixed $\epsilon$, we set $\mathcal S^\epsilon = (W, \hat W, \omega, \mathfrak s^\epsilon)$.

	We say a CF-perturbation on $\mathfrak V$, $\mathcal S=(W,\hat W,\omega, \{\mathfrak s^\epsilon\})$, is \textit{transversal to 0} if for all $\epsilon$, the map $\mathfrak s^\epsilon:V\times W\to E$ is transversal to $0\in E$ in a neighborhood of the support of $\omega$ \cite[7.9]{FOOO_Kuranishi}. In particular, $(\mathfrak s^\epsilon)^{-1}(0)$ is a smooth submanifold of $V\times W$.
	Let $f: V \to N$ be a $\Gamma$-equivariant smooth submersion into a manifold $N$.
	We say $f$ is \textit{strongly submersive with respect to $(\mathfrak V, \mathcal S)$} if $\mathcal S$ is transversal to 0 in the above sense and the induced map $f_\epsilon: (\mathfrak s^{\epsilon})^{-1}(0) \to V \to N$ is a smooth submersion.
	For later use, given a smooth map $g:M\to N$ between manifolds, we say $f$ is \textit{strongly transversal to $g$ with respect to $(\mathfrak V, \mathcal S)$} if $\mathcal S$ is transversal to $0$ and the $f_\epsilon$ is transversal to $g$ \cite[7.9]{FOOO_Kuranishi}.
	
	In the above setting, let $h$ be a smooth compactly-supported differential form on $V/\Gamma \cong \phi(V)$ with the lift form $\tilde h$ on $V$.
	Now, the \textit{pushout} of $h$ with respect to $(f,\mathcal S^\epsilon)$ \cite[7.11]{FOOO_Kuranishi} is a smooth differential form $f!(h; \mathcal S^\epsilon)$ on $M$ for each $\epsilon$ such that for any $\rho\in \Omega^*(M)$, we have
	\begin{equation}
		\label{pushout_single_orbifold_chart_eq}
		\#\Gamma \cdot \int_M \rho \wedge f!(h;\mathcal S^\epsilon) = \int_{(\mathfrak s^\epsilon)^{-1}(0)} f^*\rho \wedge \tilde h \wedge \omega 
	\end{equation}
	This form is not affected if we change $\mathcal S$ to any other representatives in its equivalence class \cite[7.13]{FOOO_Kuranishi}.

	For two orbifold charts $\mathfrak V,\mathfrak V'$ of the orbi-bundle $\mathcal E\to U$ with $\phi(V)\subseteq \phi'(V')$ in $U$, an \textit{orbifold coordinate change} from $\mathfrak V$ to $\mathfrak V'$ is a triple $\Phi=(h,\tilde \varphi, \breve\varphi)$
	where $h:\Gamma \to \Gamma'$ is an injective group homomorphism, $\tilde\varphi:V\to V'$ is an $h$-equivariant smooth \textit{open} embedding that induces the inclusion $V/\Gamma \cong \phi(V)\subseteq \phi'(V')\cong V'/\Gamma'$, $\breve{\varphi}: V\times E\to E'$ is a smooth map that induces $(V\times E)/\Gamma \to (V'\times E')/\Gamma'$ \cite[6.5 or 23.23]{FOOO_Kuranishi}.
	Now, provided a CF-perturbation $\mathcal S'=(W,\hat W,\omega,\{\mathfrak s^{\prime\epsilon}\})$ on $\mathfrak V'$, the above $\Phi$ induces a \textit{pullback CF-perturbation} $\Phi^*\mathcal S'=(W,\hat W,\omega, \{\mathfrak s^\epsilon\})$ on $\mathfrak V$ where $W,\hat W,\omega$ are unchanged and $\mathfrak s^\epsilon$ is defined by requiring \cite[7.7]{FOOO_Kuranishi}
	\begin{equation}
		\label{orbifold_coordinate_change_eq}
		\breve\varphi\big(y, \mathfrak s^\epsilon (y,\xi)\big)=\mathfrak s^{\prime\epsilon}(\tilde\varphi(y),\xi)
	\end{equation}
	This does not affect the properties of being transversal to 0 and being strongly submersive \cite[7.14]{FOOO_Kuranishi}.

	A \textit{CF-perturbation} on the Kuranishi chart $\mathcal U$ is the set of the collections $\mathfrak S=\{(\mathfrak V_r,\mathcal S_r)\}$, where $(\mathfrak V_r,\mathcal S_r)$ is as above and $\mathfrak V_r$'s cover $U$, such that the pullback CF-perturbations induced by orbifold coordinate changes are compatible modulo the aforementioned equivalence relations
	\cite[7.16, 7.20]{FOOO_Kuranishi}.
	Similarly, we say that $\mathfrak S$ is transversal to $0$ if $\mathcal S_r$ is \textit{transversal to 0} for each $r$. We say a smooth submersion $f: U\to N$ is \textit{strongly submersive with respect to $\mathfrak S$} if the restriction of $f$ on $\phi_r(V_r)\subseteq U$ is strongly submersive with respect to $\mathcal S_r$ for each $r$.
	Given a smooth map $g:M\to N$, we say $f$ is \textit{strongly transversal to $g$ with respect to $(\mathcal U,\mathfrak S)$} if for each $r$, $f$ is strongly transversal to $g$ with respect to $(\mathfrak V_r,\mathcal S_r)$. \cite[7.26]{FOOO_Kuranishi}.

	We write $\mathscr{CF}^{\mathcal U}(U)$ for the set of CF-perturbations on $U$. By requiring $\mathfrak S$ must be transversal to 0 or by requiring $\mathfrak S$ must make $f$ strongly submersive with respect to $\mathfrak S$, we obtain subsets, denoted by $\mathscr{CF}^{\mathcal U}_{\pitchfork 0}(U)$ or $\mathscr{CF}^{\mathcal U}_{f\pitchfork }(U)$, of $\mathscr{CF}^{\mathcal U}(U)$.
	One can also replace $U$ by its open subset $\Omega\subseteq U$ here.

	\begin{prop}
		The rule $\Omega\mapsto \mathscr {CF}^{\mathcal U}(\Omega)$ defines a sheaf of sets $\mathscr{CF}^{\mathcal U}$ on $U$ \cite[7.22]{FOOO_Kuranishi}.
		Moreover, $\Omega\mapsto \mathscr {CF}^{\mathcal U}_{\pitchfork 0} (\Omega)$ and $\Omega\mapsto \mathscr {CF}^{\mathcal U}_{f\pitchfork}(\Omega)$ define subsheaves $\mathscr {CF}^{\mathcal U}_{\pitchfork 0}$ and $\mathscr {CF}^{\mathcal U}_{f\pitchfork}$ respectively \cite[7.26]{FOOO_Kuranishi}.
	\end{prop}

	Fix a smooth submersion $f:U \to M$.
	Let $h$ be a compactly-supported differential form on $U$.
	For $\mathfrak S\in \mathscr{CF}_{f\pitchfork}(\mathrm{supp}(h))$ and a fixed $\epsilon$, we can define \textit{the pushout of $h$ by $f$ with respect to $\mathfrak S^\epsilon$} by
	\begin{equation}
		\label{pushout_single_chart_eq}
		f!(h;\mathfrak S^\epsilon) = \sum \ f!(\chi_r h ; \mathcal S_r^\epsilon)
	\end{equation}
	where $\{\chi_r\}$ is a smooth partition of unity subordinate to the covering $\{V_r/\Gamma_r\cong \phi_r(V_r)\}$ of $U$.
	Besides, it is independent of the choice of the $\{\chi_r\}$ and $\deg f!(h;\mathfrak S^\epsilon)=\deg h +\dim M -\dim \mathcal U$ \cite[7.32]{FOOO_Kuranishi}.

	\subsection{CF-perturbations on a Kuranishi structure}
	
	Let $\Phi=(\tilde U, \varphi, \hat\varphi)$ be an embedding of Kuranishi charts from $\mathcal U=(U, \mathcal E, \psi, s)$ to $\mathcal U'=(U',\mathcal E',\psi',s')$. 
	Recall that $\tilde U\subseteq U$, but we may assume $\tilde U=U$. Then, $\Phi$ can be viewed as an embedding of the orbi-bundles $(\varphi,\hat\varphi): (U,\mathcal E)\to (U',\mathcal E')$.
	
	Fix a CF-perturbation $\mathfrak S' \in \mathscr {CF}^{\mathcal U'}(U')$.
	We aim to investigate the conditions under which this $\mathfrak S'$ can be ``pulled back'' through $\Phi$ to $\mathscr{CF}^{\mathcal U}$.
	Let $\mathfrak S' \{(\mathfrak V'_r, \mathcal S'_r)\}$ as before, where $\mathfrak V'_r=(V'_r,E'_r,\Gamma'_r, \phi'_r,\hat\phi'_r)$ is an orbifold chart of $(U',\mathcal E')$ and $\mathcal S'_r=(W'_r,\hat W'_r,\omega'_r, \{\mathfrak s_r^{\prime\epsilon}\})$ is a local CF-perturbation on it.
	Without loss of generality, we may assume that for each $r$, there is an orbifold chart $\mathfrak V_r=(V_r,E_r,\Gamma_r, \phi_r,\hat\phi_r)$ such that $\phi_r(V_r) =\varphi^{-1}(\phi'_r(V'_r))$. Besides, $(\varphi,\hat\varphi)|_{V_r}$ can be equivalently viewed as a triple $(h_r,\tilde \varphi_r, \breve\varphi_r)$, similar to the orbifold coordinate change previously discussed near (\ref{orbifold_coordinate_change_eq}), yet distinct in that $\varphi$ is merely an embedding, not necessarily an open map.
	Therefore, we need the extra \textit{restrictable} condition on $\mathcal S'_r$, namely, requiring that $\mathfrak s^{\prime \epsilon}_r (\tilde\varphi_r(y),\xi)$ is always contained in the image of $v\mapsto \breve\varphi_r (\tilde\varphi_r(y),v)$ \cite[7.42]{FOOO_Kuranishi}. If this holds, then we can define the {pullback CF-perturbation} $\mathcal S_r$ on $\mathfrak V_r$ similarly as before.
	
	Put it another way, the above restrictable condition actually corresponds to a subsheaf $\mathscr {CF}^{\mathcal U\rhd \mathcal U'}$ of $\mathscr {CF}^{\mathcal U'}$, and to each $\mathfrak S'$ in it we can associate a \textit{pullback CF-perturbation} $\mathfrak S=\Phi^*\mathfrak S'$ in $\mathscr {CF}^{\mathcal U}$ \cite[7.44]{FOOO_Kuranishi}.
	After gaining the idea of how the CF-perturbations of individual Kuranishi charts are connected through embeddings of these charts, we can proceed to establish the following definition.

	A \textit{CF-perturbation} of a Kuranishi structure $\widehat{\mathcal U}=\{\mathcal U_p, \Phi_{pq}\}$ is a collection $\widehat{\mathfrak S}=\{\mathfrak S_p\}$ consisting of a CF-perturbation $\mathfrak S_p$ of each chart $\mathcal U_p$ such that they are compatible with the pullbacks in the sense that $\Phi_{pq}^*(\mathfrak S_p)=\mathfrak S_q$\footnote{The left side makes sense only if $\mathfrak S_p$ is in $\mathscr{CF}^{\mathcal U_q \rhd \mathcal U_p}$\cite[7.44(2)]{FOOO_Kuranishi}; removing \cite[9.1(2)]{FOOO_Kuranishi} should be unambiguous.} on the domains of both sides \cite[9.1]{FOOO_Kuranishi}.
	For any fixed $\epsilon$, we similarly set $\widehat{\mathfrak S}^\epsilon=\{\mathfrak S_p^\epsilon\}$.

	\subsection{Perturbed smooth correspondence}
	To utilize the CF-perturbations, we must first establish their existence; this often involves modifying a given Kuranishi structure by ``thickening'' it.
	
	A \textit{strict KK-embedding} $\widehat{\Phi}=\{\Phi_p\}$ from a Kuranishi structure $\widehat{\mathcal U}=\{ \mathcal U_p=(U_p,\mathcal E_p, \psi_p, s_p), \Phi_{pq}=(U_{pq}, \varphi_{pq}, \hat\varphi_{pq}) \}$ to another $\widehat{\mathcal U}^+$ with the same support set $Z\subseteq \mathcal M$ consists of a collection of embeddings of Kuranishi charts $\Phi_p=(\varphi_p, \widehat{\varphi}_p): \mathcal U_p \to \mathcal U_p^+$ for each $p\in Z$ such that (i) $\mathrm{Im}(\psi_p)\subseteq \mathrm{Im}(\psi_p^+)$; (ii) $U_{pq}\subseteq \varphi_q^{-1}(U^+_{pq})$ if $q\in\mathrm{Im}(\psi_p)$; (iii) $\Phi_p\circ \Phi_{pq} = \Phi^+_{pq} \circ \Phi_q$.
	Moreover, we say that $\widehat{\Phi}$ is \textit{open} if $\dim U_p=\dim U_p^+$ for each $p$; we say that $\widehat{\mathcal U}$ is an \textit{open substructure} of $\widehat{\mathcal U^+}$ if there exists an open KK-embedding $\widehat{\mathcal U}\to\widehat{\mathcal U^+}$.
	A \textit{KK-embedding} is a strict KK-embedding from an open substructure of the source to the target \cite[3.19]{FOOO_Kuranishi}.
	
	A \textit{strict thickening} of $\widehat{\mathcal U}$ is a pair $(\widehat{\mathcal U^+}, \widehat{\Phi})$ of a Kuranishi structure $\widehat{\mathcal U^+}$ and a strict KK-embedding $\widehat{\Phi}: \widehat{\mathcal U}\to\widehat{\mathcal U^+}$ with the following extra condition:
	Note that each point $p$ in the support set $Z$ must be contained in the open subset $\psi_p(s_p^{-1}(0))\cap \psi_p^+( (s_p^+)^{-1}(0))$ of $\mathcal M$. We require there is a neighborhood $O_p$ of $p$ (within the metrizable space $\mathcal M$) such that for each $q\in O_p\cap Z$, there exists a neighborhood $W_p(q)$ of the point $o_p(q):=\psi_p^{-1}(q)$ (within the orbifold $U_p$) such that (i) $\varphi_p(W_p(q))\subseteq \varphi_{pq}^+(U_{pq}^+)$ inside $U_p^+$ and (ii) for any point $z\in U_{pq}^+$, $\hat\varphi_p(E_p)|_z \subseteq \hat\varphi_{pq}^+(E_q^+)|_z$ whenever both sides are defined.
	Finally, we define a \textit{thickening} of $\widehat{\mathcal U}$ to be a strict thickening of an open substructure of it \cite[5.3]{FOOO_Kuranishi}.
	Roughly speaking, a thickening is a special sort of KK-embedding with nice technical properties.

	Let $(\mathcal M,\widehat{\mathcal U})$ be a K-space where $\widehat{\mathcal U}=\{\mathcal U_p, \Phi_{pq}\}$.
	Let $\widehat{f}=\{f_p\}: \widehat{\mathcal U}\to M$ be a strongly smooth map.
	Given a KK-embedding $\widehat{\Phi}=\{\Phi_p\}:\widehat{\mathcal U'}\to\widehat{\mathcal U}$, then the new collection $\widehat{f}\circ \widehat{\Phi}:=\{f_p\circ \Phi_p\}: \widehat{\mathcal U'}\to M$ is strongly continuous, strongly smooth, or weakly submersive whenever so is $\widehat f$ \cite[3.44]{FOOO_Kuranishi}. We call $\widehat{f}\circ \widehat{\Phi}$ the \textit{pullback} of $\widehat{f}$ via $\widehat{\Phi}$.
	Given a CF-perturbation $\widehat{\mathfrak S}=\{\mathfrak S_p\}$ of $\widehat{\mathcal U}$, a strongly smooth map $\widehat{f}=\{f_p\}$ as above is said to be \textit{strictly strongly submersive with respect to $\widehat{\mathfrak S}$} if each $f_p:U_p\to M$ is strongly submersive with respect to $\mathfrak S_p$; see (\ref{orbifold_coordinate_change_eq}).
	We call it \textit{strongly submersive with respect to $\widehat{\mathfrak S}$} if its restriction to an open substructure $\widehat{\mathcal U}_0$ of $\widehat{\mathcal U}$ is strictly so \cite[9.2]{FOOO_Kuranishi}.
	\begin{equation*}
		\Scale[0.85]{\xymatrix{
				& (\mathcal M, \widehat{\mathcal U}) \ar[dl]^{\widehat{f_s}} \ar[dr]_{\widehat{f_t}}  \\
				M_s & & M_t
		}}
	\end{equation*}
	By a \textit{smooth correspondence}, we mean a quadruple $\mathfrak X=( \mathcal M, \widehat{\mathcal U},  \widehat{f_s},\widehat{f_t})$ where $\widehat{f_s}: (\mathcal M, \widehat{\mathcal U})\to M_s$ is a strongly smooth map into a manifold $M_s$ and $\widehat{f_t}: (\mathcal M, \widehat{\mathcal U})\to M_t$ is a both strongly smooth and weakly submersive map into another manifold $M_t$ \cite[7.1]{FOOO_Kuranishi}.
	A CF-perturbation of smooth correspondence is a smooth correspondence and $\widehat{\mathfrak S}=\{\mathfrak S_p\}$ is a CF-perturbation of $\widehat{\mathcal U}$ with respect to which $\widehat{f_t}$ is strongly submersive \cite[9.25]{FOOO_Kuranishi}.
	In this setting, we call the tuple $\tilde{\mathfrak X}=(\mathfrak X,\widehat{\mathfrak S})=( \mathcal M, \widehat{\mathcal U},  \widehat{f_s},\widehat{f_t}; \widehat{\mathfrak S})$ a \textit{perturbed smooth correspondence} \cite[10.19]{FOOO_Kuranishi}.
	Roughly speaking, the following result tells that we can always perturb a smooth correspondence up to thickening of Kuranishi structures.
	
	\begin{lem}
		\label{CF_exist_lem}
		For each smooth correspondence $\mathfrak X=(\mathcal M, \widehat{\mathcal U}, \widehat{f_s},\widehat{f_t})$, there exists a perturbed smooth correspondence $\tilde{\mathfrak X}=(\mathcal M, \widehat{\mathcal U^+}, \widehat{f^+_s},\widehat{f^+_t}, \widehat{\mathfrak S})$ and a KK-embedding $\widehat \Phi: \widehat{\mathcal U}\to\widehat{\mathcal U^+}$ such that $\widehat f_s=\widehat {f_s^+}\circ \widehat \Phi$, $\widehat f_t=\widehat {f_t^+}\circ \widehat \Phi$, and $(\widehat{\mathcal U^+},\widehat{\Phi})$ is a thickening of $\widehat{\mathcal U}$ \cite[9.26]{FOOO_Kuranishi}.
	\end{lem}
	
	\begin{proof}[Sketch of proof]
		A Kuranishi structure has uncountably many charts and thus are hard to work with.
		We can introduce the notion of good coordinate system that serves a role akin to a locally finite sub-atlas in the manifold theory.
		A \textit{good coordinate system} $\widetriangle {\mathcal U}$ of $\mathcal M$ consists of a collection of Kuranishi charts $\mathcal U_{\mathfrak p}= (U_{\mathfrak p}, \mathcal E_{\mathfrak p}, \psi_{\mathfrak p}, s_{\mathfrak p})$, indexed by a \textit{finite} partially ordered set $(\mathfrak P, \leqslant)$, together with a collection of coordinate changes in the strong sense $\Phi_{\mathfrak p \mathfrak q} = (U_{\mathfrak p \mathfrak q} , \varphi_{\mathfrak p \mathfrak q}, \hat \varphi_{\mathfrak p \mathfrak q} ) : \mathcal U_{\mathfrak q}\to \mathcal U_{\mathfrak p}$, with $\mathfrak q\leqslant \mathfrak p$, such that the conditions similar to the case of a Kuranishi structure hold \cite[3.15]{FOOO_Kuranishi}.
		We can similarly introduce the notions of strongly smooth or weakly submersive maps $\widetriangle{f}$ \cite[3.43]{FOOO_Kuranishi}, CF-perturbations $\widetriangle{\mathfrak S}$\cite[7.49]{FOOO_Kuranishi}, strongly submersive maps \cite[7.50]{FOOO_Kuranishi} of $\widetriangle{\mathcal U}$.
		Similar to the above KK-embeddings, we can also introduce \textit{KG-embeddings} \cite[3.30]{FOOO_Kuranishi}, \textit{GG-embeddings} \cite[3.24]{FOOO_Kuranishi}, \textit{GK-embedding} \cite[5.6]{FOOO_Kuranishi}. Here `K' and `G' stands for Kuranishi structures and good coordinate systems respectively.

		For our fixed Kuranishi structure $\widehat{\mathcal U}$ of $\mathcal M$, there exists a good coordinate system $\widetriangle{\mathcal U}$ of $\mathcal M$ and a KG-embedding $\widehat{\Phi}: \widehat{\mathcal U} \to \widetriangle{\mathcal U}$ \cite[1.1, 3.35, or 9.10]{FOOO_Kuranishi}.
		Given $\widetriangle{\mathcal U}$, there also exists a Kuranishi structure $\widehat{\mathcal U}^+$ and a GK-embedding $\widetriangle{\mathcal U} \to\widehat{\mathcal U}^+$ \cite[5.26, 6.30]{FOOO_Kuranishi}, and we can further show that the composition $\widehat{\mathcal U} \to \widetriangle{\mathcal U} \to \widehat{\mathcal U}^+$ is a thickening \cite[5.28]{FOOO_Kuranishi}.
		Only for a good coordinate system, the CF-perturbation $\widetriangle{\mathfrak S}$ always exists, and we may further achieve the desired strongly submersive condition \cite[7.51]{FOOO_Kuranishi}. Using this, the above-mentioned GK-embedding $\widetriangle{\mathcal U}\to \widehat{\mathcal U}^+$ can induce a new CF-perturbation $\widehat{\mathfrak S}^+$ of $\widehat{\mathcal U}^+$ that is compatible with $\widetriangle{\mathfrak S}$ and preserves the desired strongly submersive condition as well \cite[9.9]{FOOO_Kuranishi}.
	\end{proof}

	\subsection{Correspondence map}
	Many concepts for differential forms on manifolds can be extended to the Kuranishi structure setting as follows \cite[7.70 \& 7.71]{FOOO_Kuranishi}:
	\begin{itemize}[itemsep=0pt]
		\item A \textit{(smooth) differential $k$-form} $\widehat{h}=\{h_p\}$ on a Kuranishi structure $\widehat{\mathcal U}$ of $\mathcal M$ is a collection of smooth different $k$-forms $h_p$ on the orbifolds $U_p$ such that $\varphi_{pq}^*h_p=h_q$. For this $\widehat{h}$, the collection $d\widehat{h}=\{dh_{p}\}$ is a differential $(k+1)$-form on $\widehat{\mathcal U}$, called the \textit{exterior derivative} of $\widehat{h}$.
		\item  If $h$ is a differential $k$-form on a manifold $M$ and $\widehat{f}=\{ f_{p}\}: (X, \widehat{\mathcal U})\to M$ is a strongly smooth map, then $\widehat{f}^*h=\{f^*_{p}h\}$ defines a differential $k$-form on $\widehat{\mathcal U}$, called the \textit{pullback} of $h$ by $\widehat{f}$.
		\item  If $\widehat{h}^i=\{h^i_{ p}\}$ are differential $k_i$-forms on $\widehat{\mathcal U}$ for $i=1,2$, then $\widehat{h}^1\wedge\widehat{h}^2=\{h^1_{ p}\wedge h^2_{ p}\}$ is a differential $(k_1+k_2)$-form on $\widehat{\mathcal U}$, called the \textit{wedge product}.
	\end{itemize}

	By Lemma \ref{CF_exist_lem} above, replacing it by a thickening if necessary, we may assume that there exists a CF-perturbation $\widehat{\mathfrak S}$ on $\mathfrak X=(\mathcal M, \widehat{\mathcal U}, \widehat{f_s},\widehat{f_t})$ such that $\widehat{f_t}$ is strongly submersive with respect to $\widehat{\mathfrak S}$, i.e. $\tilde{\mathfrak X}=(\mathfrak X, \widehat{\mathfrak S})$ is a perturbed smooth correspondence.
	Then, we have: 
	\begin{thm}
		\label{corr_thm}
		There is a canonical construction that assigns to $\tilde{\mathfrak X}$ a linear map
		\begin{equation}
			\label{Corr_eq}
			\Corr^\epsilon_{\tilde{\mathfrak X}}:= \widehat f_t ! \Big( (\widehat{f_s})^*h ; \widehat{\mathfrak S^\epsilon}\Big) : \Omega^*(M_s)\to \Omega^*(M_t)
		\end{equation}
		for each sufficiently small $\epsilon$. Here `canonical' means the map (\ref{Corr_eq}) depends only on $\tilde{ \mathfrak X}$ and $\epsilon$ \cite[9.25]{FOOO_Kuranishi}.
		Moreover, replacing $\widehat{\mathcal U}$ by a thickening $\widehat{\mathcal U}^+$ of it does not affect the map (\ref{Corr_eq}).
	\end{thm}
	
	\begin{proof}[Sketch of proof]
		Since $\widehat h:= (\widehat{f_s})^*h $ can be verified as a differential form on $\widehat{\mathcal U}$, it suffices to define the pushout $\widehat f_t ! \Big(  \widehat h ; \widehat{\mathfrak S^\epsilon}\Big)$. This will be a globalized version of the pushout of a single Kuranishi chart as established in (\ref{pushout_single_chart_eq}).
		We use the tool of good coordinate systems again.
		Recall that as in the proof of Lemma \ref{CF_exist_lem}, there is a good coordinate system $\widetriangle{\mathcal U}$ and a KG-embedding $\widehat{\Phi}:\widehat{\mathcal U} \to \widetriangle{\mathcal U}$, and there always exists a CF-perturbation $\widetriangle{\mathfrak S}$ on $\widetriangle{\mathcal U}$ \cite[3.35]{FOOO_Kuranishi}.
		Moreover, we can properly choose $\widetriangle{\mathfrak S}$ so that it is compatible with the given $\widehat{\mathfrak S}$ through the above KG-embedding $\widehat{\Phi}$ and that the induced map $\widetriangle{f_t}$ is strongly submersive with respect to $\widetriangle{\mathfrak S}$ as well \cite[9.10]{FOOO_Kuranishi}.
		The advantage of using a good coordinate system is that it only consists finitely many Kuranishi charts and always admits a version of partition of unity \cite[7.68]{FOOO_Kuranishi}; for this, we can generalize (\ref{pushout_single_chart_eq}) to first establish the pushout 
		$\widetriangle{f_t}!(\widetriangle{h}; \widetriangle{\mathfrak S}^\epsilon)$
		with respect to a good coordinate system $\widetriangle{\mathfrak S}$ \cite[7.79]{FOOO_Kuranishi}.
		Then, it remains to use \cite[5.17, 5.20]{FOOO_Kuranishi} to show that this pushout does not depend on the choices of good coordinate system $\widetriangle{\mathcal U}$ and the related data $\widetriangle{\mathfrak S}$, $\widehat{\Phi}$ \cite[9.16]{FOOO_Kuranishi}.
	\end{proof}
	\vspace{-0.3em}
	The correspondence maps are characterized by two primary properties: Stokes' formula and the correspondence formula. These will be briefly reviewed below.
	
	\subsection{Stokes' Formula}
	\label{ss_stokes}
	Recall the stratification $S_k(U)\subseteq U$ for an orbifold with corners $U$ in \ref{ss_Kuranishi}.
	We can further define the normalized (codimension-$k$) corner $\widehat{S_k}(U)$ that consists of pairs $(x,A)$ where $x\in S_k(U)$ and $A\subseteq \{1,2,\dots, m\}$ if $x\in S_m^\circ (U)$. Forgetting $A$ defines the normalization map $\widehat{S_k}(U)\to S_k(U)$.
	For $k=1$, we set $\partial U=\widehat{S_1}(U)$, called the \textit{normalized boundary} \cite[8.4]{FOOO_Kuranishi}.

	Let $(\mathcal M, \widehat{\mathcal U})$ be a K-space where $\widehat{\mathcal U}=\{\mathcal U_p,\Phi_{pq}\}$ with $\mathcal U_p=(U_p,\mathcal E_p, \psi_p,s_p)$ and $\Phi_{pq}=(U_{pq},\varphi_{pq},\hat\varphi_{pq})$ as before.
	By the \textit{corner structure stratification} of $(\mathcal M, \widehat{\mathcal U})$ we mean a collection of subsets given by
	$S_k(\mathcal M, \widehat{\mathcal U})=\{p\in \mathcal M \mid o_p \in S_k(U_p)\}$ where $k\geqslant 1$ and $o_p$ is the point determined by $s_p(o_p)=0$ and $\psi_p(o_p)=p$ \cite[4.16]{FOOO_Kuranishi}.
	Similar to the case of orbifold, we can define the \textit{normalized (codimension-$k$) corner} $\widehat{S_k}(\mathcal M, \widehat{\mathcal U})$ with \textit{normalization maps} $\pi_k: \widehat{S_k}(\mathcal M, \widehat{\mathcal U})\to {S_k}(\mathcal M, \widehat{\mathcal U})$ \cite[24.17]{FOOO_Kuranishi}.
	Each $\widehat{S_k}(\mathcal M,\widehat{\mathcal U})$ is a K-space.
	For $k=1$, we call $\widehat{S_1}(\mathcal M,\widehat{\mathcal U})$ the \textit{normalized boundary} of $(\mathcal M, \widehat{\mathcal U})$, denoted by $(\partial \mathcal M,  \widehat{\mathcal U}_\partial)$, with the normalization map $\pi_1:\partial \mathcal M \to S_1(\mathcal M,\widehat{\mathcal U})$ of the underlying topological spaces \cite[8.8]{FOOO_Kuranishi}.

	A CF-perturbation $\widehat{\mathfrak S}$ on $\widehat{\mathcal U}$ induces a CF-perturbation $\widehat{\mathfrak S_\partial}$ on $\widehat{\mathcal U}_\partial$;
	besides, given $\widehat{f}: \widehat{\mathcal U} \to M$ is strongly submersive with respect to $\widehat{\mathfrak S}$, there is an induced $\widehat{f_\partial}: \widehat{\mathcal U_\partial} \to M$ that is strongly submersive with respect to $\widehat{\mathfrak S_\partial}$ \cite[9.27]{FOOO_Kuranishi}.
	In particular, a perturbed smooth correspondence $\tilde{\mathfrak X}=(\mathfrak X,\widehat{\mathfrak S})$ induces a perturbed smooth correspondence, denoted by $\tilde{\mathfrak X}_\partial=(\mathfrak X_\partial, \widehat{\mathfrak S_\partial})$.
	Similar notions also apply to relative K-spaces.
	Now, the Stokes' formula now states that \cite[9.28 \& 9.29]{FOOO_Kuranishi}:
	
	\begin{thm}
		For small $\epsilon>0$, $d\Big( \  \widehat f ! (\widehat{h}; \widehat{\mathfrak S}^\epsilon ) \Big)=\widehat{f}! \big(d\widehat{h}; \widehat{\mathfrak S}^\epsilon \big)+(-1)^* \widehat{f_\partial}! \big(  \widehat{h_\partial} ; \widehat{\mathfrak S_\partial^\epsilon} \big)$. In terms of (\ref{Corr_eq}),
		\begin{equation}
			\label{stokes_eq}
			d\circ \Corr_{\tilde{\mathfrak X}}^\epsilon = \Corr_{\tilde{\mathfrak X}}^\epsilon  \circ \  d + (-1)^* \Corr_{\tilde{\mathfrak X}_\partial}^\epsilon
		\end{equation}
	\end{thm}
	
	\begin{proof}[Sketch of proof]
		The strategy is initially reducing the problem to a version for good coordinate systems \cite[8.11]{FOOO_Kuranishi}. This reduction was almost the same as the process of showing Theorem \ref{corr_thm}.
		As before, a good coordinate system takes advantage of partition of unity, further reducing the problem to a version within a single Kuranishi chart \cite[8.12]{FOOO_Kuranishi}. This then becomes the usual Stokes' formula.
	\end{proof}

	\subsection{Composition formula}
	There are natural constructions of direct products of Kuranishi structures and fiber products of (strongly) smooth maps, if the involved maps are (weakly) submersive.
	For clarity, we focus only on K-spaces, although the case of relative K-spaces can be addressed similarly.

	Let $(X, \widehat{\mathcal U})$ be a K-space and $\widehat{f}=\{f_p\}: \widehat{\mathcal U} \to N$ be a strongly smooth map.
	Let $g:M\to N$ be a smooth map between smooth manifolds.
	Take a CF-perturbation $\widehat{\mathfrak S}=\{\mathfrak S_p\}$ on $\widehat{\mathcal U}$.
	
	\begin{itemize}[itemsep=0pt]
		\item The \textit{underlying continuous map} of $\widehat{f}$ refers to the map $f: \widehat{\mathcal U} \to N$ defined by $f(p)=f_p(p)$. We can show that it is continuous \cite[3.40 (2)]{FOOO_Kuranishi}.
		\item We say $\widehat{f}$ is \textit{weakly transversal} to $g$ on $X$ if for each $p\in X$ and Kuranishi chart $\mathcal U_p=(U_p,\mathcal E_p,\psi_p,s_p)$ and for each $q\in M$, we have that $df_p (T_x U_p) +dg(T_y M)=T_zN$ for any $(x,y)\in U_p\times M$ with $z=f_p(x)=g(y)$ \cite[4.2]{FOOO_Kuranishi}.
		We first have the (topological) fiber product $X\times_N M$ or $X \times_{(f,g)} M$.
		\item
		There is naturally the \textit{fiber product} of $\widehat{f}:\widehat{\mathcal U}\to N$ and $g:M\to N$ with the fiber product Kuranishi structure on $X\times_N M$.
		The Kuranishi neighborhood\footnote{There seems to be a typo error for the terminology in \cite[4.6]{FOOO_Kuranishi}: the tuple $\mathcal U_p=(U_p,\mathcal E_p, \psi_p, s_p)$ should be called a \textit{Kuranishi chart} instead of \textit{Kuranishi neighborhood}, since the latter refers to the orbifold $U_p$ in $\mathcal U_p$ \cite[3.1]{FOOO_Kuranishi}.} takes the form $U_p\times_N M=U_p\times_{(f_p,g)} M$ \cite[4.6]{FOOO_Kuranishi}.
		Finally, we denote the resulting K-space by $(X, \widehat{\mathcal U}) \times_N M$ or $(X, \widehat{\mathcal U})\times_{(f,g)} M$ \cite[4.9(1)]{FOOO_Kuranishi}.
		\item We say $\widehat{f}$ is \textit{strongly transversal to $g$ with respect to $\widehat{\mathfrak S}$} if for the CF-perturbation $\mathfrak S_p$ on each Kuranishi chart $\mathcal U_p$, the map $f_p$ is strongly transversal to $g$ with respect to $\mathfrak S_p$ \cite[7.50 or 9.2]{FOOO_Kuranishi}.
		In this case, there is naturally a \textit{fiber product CF-perturbation} $\widehat{\mathfrak S} \times_{(f,g)}   N=\{ \mathfrak S_p \times_{(f,g)} N\}$ \cite[10.12]{FOOO_Kuranishi}.
	\end{itemize}

	Let $(X_1,\widehat{\mathcal U_1})$ and $(X_2,\widehat{\mathcal U_2})$ be K-spaces. Suppose $\widehat{f_i}=\{(f_i)_p\}:(X_i,\widehat{\mathcal U_i})\to N$ are strongly smooth maps.
	Suppose $\widehat{\mathfrak S_i}$ is a CF-perturbation of $\widehat{\mathcal U_i}$ for $i=1,2$.

	\begin{itemize}[itemsep=0pt]
		\item  
		There is an obvious way to define the \textit{direct product} Kuranishi structure, and we denote the resulting K-space by $(X_1 \times X_2,  \widehat{\mathcal U_1}\times \widehat{\mathcal U_2})$ \cite[4.1]{FOOO_Kuranishi}. There is also an obvious way to define the direct product CF-perturbation $\widehat{\mathfrak S_1}\times \widehat{\mathfrak S_2}$ on $\widehat{\mathcal U_1}\times \widehat{\mathcal U_2}$ \cite[10.6]{FOOO_Kuranishi}.
		
		\item
		We say $\widehat{f_1}$ and $\widehat{f_2}$ are \textit{weakly transversal} if $\widehat{f_1}\times \widehat{f_2}: \widehat{\mathcal U_1}\times \widehat{\mathcal U_2}\to N\times N$ is weakly transversal to the diagonal map $\Delta_N:N \to N\times N$.
		The \textit{fiber product} of $\widehat{f_1}$ and $\widehat{f_2}$ is defined as the fiber product of $\widehat{f_1}\times \widehat{f_2}$ and $\Delta_N$ in the sense mentioned above.
		Let $f_1, f_2$ be the underlying continuous maps.
		Denote the consequent K-space by $(X_1, \widehat{\mathcal U_1})\times_{(f_1,f_2)} (X_2, \widehat{\mathcal U_2} )$ or $(X_1,  \widehat{\mathcal U_1})\times_N (X_2, \widehat{\mathcal U_2} )$ \cite[4.9(2)]{FOOO_Kuranishi}.
		Similarly, the Kuranishi neighborhood takes the form $(U_{p_1}\times U_{p_2}) \times_N M=U_{p_1}\times_N U_{p_2} = U_{p_1}\times_{(f_1)_{p_1}, (f_2)_{p_2}} U_{p_2}$.
		\item 
		We say $\widehat{f_1}$ and $\widehat{f_2}$ are \textit{strongly transversal with respect to $\widehat{\mathfrak S_1}$ and $\widehat{\mathfrak S_2}$} if $\widehat{f_1}\times \widehat{f_2}$ is strongly transversal to the diagonal $\Delta_N$ with respect to $\widehat{\mathfrak S_1}\times \widehat{\mathfrak S_2}$. In this situation, there is naturally the \textit{fiber product CF-perturbation} of the above-mentioned fiber product Kuranishi structure, denoted as $\widehat{\mathfrak S_1}\times_{(\widehat{f_1}, \widehat{f_2})} \widehat{\mathfrak S_2}$ \cite[10.13]{FOOO_Kuranishi}.
	\end{itemize}

	Let $\mathfrak X_{12}=(X_{12}, \widehat{\mathcal U_{12}}, \widehat{f_1}, \widehat{f_2})$ and $\mathfrak X_{23}=(X_{23}, \widehat{\mathcal U_{23}}, \widehat{g_2},\widehat{g_3})$ be smooth correspondences. In particular, $\widehat{f_2}$ and $\widehat{g_3}$ are weakly submersive.
	Assume that $\widehat{f_2}$ and $\widehat{g_2}$ send into the same manifold $M_2$. 
	Let $f_2$ and $g_2$ be the underlying continuous maps. We consider the (topological) fiber product $X_{13}:=X_{12}\times_{(f_2,g_2)} X_{23}$.
	Then, it is equipped with the fiber product Kuranishi structure, denoted as $\widehat{\mathcal U_{13}}$.
	Let $p=(p_1,p_2)$ be an arbitrary point in $X_{13}$.
	By the natural projection maps of Kuranishi neighborhoods $U_{p_1}\times_{M_2} U_{p_2}\to U_{p_i}$, we can construct strongly smooth maps $\widehat{h_1}: \widehat{\mathcal U_{13}} \to M_1$ and $\widehat{h_3}: \widehat{\mathcal U_{13}} \to M_3$, where $\widehat{h_3}$ is further weakly submersive \cite[10.18]{FOOO_Kuranishi}. By definition, this means $\mathfrak X_{13}:= ( X_{13}, \widehat{\mathcal U_{13}}, \widehat{h_1},\widehat{h_3})$ is a smooth correspondence.

	Besides, suppose $\tilde{\mathfrak X}_{12}=(\mathfrak X_{12}, \widehat{\mathfrak S_{12}})$ and $\tilde{\mathfrak X}_{23}=(\mathfrak X_{23}, \widehat{\mathfrak S_{23}})$ are perturbed smooth correspondences.
	Then, the fiber product CF-perturbation $\widehat{\mathfrak S_{13}}$ of $\widehat{\mathfrak S_{12}}$ and $\widehat{\mathfrak S_{23}}$ gives rise to a perturbed smooth correspondence $\tilde{\mathfrak X}_{13}=(\mathfrak X_{13}, \widehat{\mathfrak S_{13}})=(X_{13}, \widehat{\mathcal U_{13}}, \widehat{h_1},\widehat{h_3}, \widehat{\mathfrak S_{13}})$ \cite[10.19]{FOOO_Kuranishi}.
	See the following diagram.
	\begin{equation}
		\label{composition_formula_diagram}
		\Scale[0.95]{\xymatrix{
				& & \mathfrak X_{13} \ar@{-->}[dl] \ar@{-->}[dr] 
				\ar@/^1pc/[ddrr]^{\widehat{h_3}}
				\ar@/_1pc/[ddll]_{\widehat{h_1}} \\
				& \mathfrak X_{12} \ar[dl]^{\widehat f_1} \ar[dr]_{\widehat f_2} & & \mathfrak X_{23} \ar[dl]^{\widehat g_2} \ar[dr]_{\widehat g_3} \\
				M_1 & & M_2 & & M_3}
		}
	\end{equation}
	
	\begin{thm}[Composition Formula]
		\label{composition_formula_thm}
		$\Corr^\epsilon_{\tilde{\mathfrak X}_{13}}=\Corr^\epsilon_{\tilde{\mathfrak X}_{23}}\circ \Corr^\epsilon_{\tilde{\mathfrak X}_{12}}$ for sufficiently small $\epsilon$ \cite[10.21]{FOOO_Kuranishi}
	\end{thm}

	\subsection{$A_\infty$ correspondence}
	At this stage, we studied individual moduli spaces $\mathcal M_{k+1,\beta}:=\mathcal M_{k+1,\beta}(J,L)$ in (\ref{moduli_single}), and we also need to study all of them simultaneously as a moduli space system $\mathbb M(J,L)=\{\mathcal M_{k+1,\beta}\}$ in (\ref{moduli_space_system_eq}).
	As topological spaces, $\mathcal M_{k,\beta}$ can be always defined, and there are obvious topological embeddings from various $\mathcal M_{k_1+1,\beta_1}\times_{(\ev_i,\ev_0)} \mathcal M_{k_2+1,\beta_2}$ into $\mathcal M_{k+1,\beta}$, where $k_1+k_2=k+1$ and $\beta_1+\beta_2=\beta$.
	The Gromov compactness implies that for any $E_0>0$, the set of $\beta\in H_2(X,L)$ such that some $\mathcal M_{k+1,\beta}$ is non-empty and $E(\beta)\leqslant E_0$ is finite.
	Moreover, we know $\mathcal M_{k+1,\beta}=\varnothing$ if $E(\beta)<0$.
	There exists a partial order ``$\prec$'' on the set of pairs $(k,\beta)$ defined by declaring $(k,\beta)\prec (k',\beta')$ when $E(\beta)<E(\beta')$ or $\omega(\beta)=\omega(\beta')$, $k<k'$. Then, we have $(k_i,\beta_i)\prec (k,\beta)$ here.

	Given $\mathbb M(J,L)=\{\mathcal M_{k+1,\beta}\}$, we say that a collection $\mathbb {MU} (J,L):=\{(\mathcal M_{k+1,\beta}, \widehat{\mathcal U}_{k,\beta})\}$ of K-spaces with corners is an \textit{$A_\infty$ correspondence} (over $L$) \cite[21.9]{FOOO_Kuranishi} if \footnote{An $A_\infty$ correspondence is also called a `tree-like K-system' in \cite{FOOO_Kuranishi,FOOODiskTwo}. We omit (I), (II), and (VIII) from \cite[21.6]{FOOO_Kuranishi} to maintain a focus, since they are not directly related to the Kuranishi structures. }
	
	\begin{itemize}[itemsep=0pt]
		\item 
		For $\beta\in H_2(X,L)$ and $k\geqslant 0$, $\widehat{\mathcal U}_{k,\beta}$ is an oriented Kuranishi structure on $\mathcal M_{k+1,\beta}$ such that the evaluation map $\ev_i: \widehat{\mathcal U}_{k,\beta}\to L$ is strongly smooth for $0\leqslant i \leqslant k$ and that $\ev_0$ is also weakly submersive. Its dimension is $\mu(\beta)+\dim L + k-2$.
		\item 
		There is a natural isomorphism of K-spaces between the normalized boundary $(\partial \mathcal M_{k+1,\beta},  (\widehat{\mathcal U}_{k,\beta} )_\partial )$ and the following disjoint union of the fiber product K-spaces: (see \cite[21.6]{FOOO_Kuranishi} for the sign $*$)
	\end{itemize}
	\begin{equation}
		\label{tree_like_moduli_eq}
		\bigsqcup_{k_1+k_2=k+1, \ \beta_1+\beta_2=\beta, \ 1\leqslant i\leqslant k_2}  (-1)^* \big( \mathcal M_{k_1+1,\beta_1}, \widehat{\mathcal U}_{k_1,\beta_1} \big) \times_{(\ev_i,\ev_0)} 
		\big(
		\mathcal M_{k_2+1,\beta_2} , \widehat{\mathcal U}_{k_2,\beta_2}
		\big)
	\end{equation}
	\begin{itemize}
		\item [] We also require the natural analogs for the normalized corners $\widehat{S}_m (\mathcal M_{k+1,\beta} , \widehat{\mathcal U}_{k,\beta})$.\footnote{Note that $\widehat{S}_1(\mathcal M_{k+1,\beta}, \widehat{\mathcal U}_{k,\beta})$ is the same as the normalized boundary $(\partial \mathcal M_{k+1,\beta},  (\widehat{\mathcal U}_{k,\beta} )_\partial )$. The (XI) in \cite[21.6]{FOOO_Kuranishi} is actually a necessary condition of (X), so we choose to omit it for simplicity.}
	\end{itemize}

	\begin{thm}
		\emph{\cite[Theorem 2.16]{FOOODiskTwo}} \
		Given the moduli system $\mathbb M(J,L)=\{\mathcal M_{k+1,\beta}\}$, there exists an $A_\infty$ correspondence $\mathbb {MU}(J,L):= \{(\mathcal M_{k+1,\beta}, \widehat{\mathcal U}_{k,\beta})\}$.
	\end{thm}
	
	\begin{proof}[Sketch of proof]
		Recall that for fixed $k$ and $\beta$, we can individually construct the obstruction bundle data $\mathscr O_{k,\beta}=\{E^{k,\beta}_{\mathbf p}(\mathbf x) \}$ of $\mathcal M_{k+1,\beta}$ by Theorem \ref{Obstruction_bundle_data_single_thm} and then use it to obtain a Kuranishi structure $\widehat{\mathcal U}_{k,\beta}$ by Theorem \ref{Kuranishi_single_moduli_thm}. We just need to include a few extra points; recall that the choices of $E_{\mathbf p}^{k,\beta}(\mathbf x)$ are quite flexible.
		Specifically, we aim to construct $\mathscr O_{k,\beta}$'s inductively on the above order $\prec$. If $\{E^{k_i,\beta_i}_{\mathbf p_i}(\mathbf x_i)\}$ for $i=1,2$ have been specified on $\mathcal M_{k_i,\beta_i}$ in (\ref{tree_like_moduli_eq}), then we recall all $(k_i,\beta_i)\prec (k,\beta)$. By constructing them inductively, we can require that $E^{k,\beta}_{\mathbf p}(\mathbf x)= E^{k_1,\beta_1}_{\mathbf p_1}(\mathbf x_1) \oplus E^{k_2,\beta_2}_{\mathbf p_2}(\mathbf x_2)$ near $\partial\mathcal M_{k+1,\beta}$ as well as similar conditions near normalized corners \cite[5.4]{FOOODiskTwo}.
		With this modification, one can finally verify that the consequent system of Kuranishi structures $\widehat{\mathcal U}_{k,\beta}$ satisfies the desired conditions \cite[5.3]{FOOODiskTwo}.
	\end{proof}

	\subsection{Perturbed $A_\infty$ correspondence}
	
	Given an $A_\infty$ correspondence $\mathbb {MU}(J,L)=\{\mathcal M_{k+1,\beta}, \widehat{\mathcal U}_{k,\beta}\}$ as above, each $\mathfrak X_{k,\beta}:= \big(\mathcal M_{k+1,\beta}, \widehat{\mathcal U}_{k,\beta}, (\ev_1,\dots, \ev_k), \ev_0\big)$ is a smooth correspondence.
	By Lemma \ref{CF_exist_lem}, there is a thickening $\widehat{\mathcal U}_{k,\beta}^+$ of $\widehat{\mathcal U}_{k,\beta}$ and a CF-perturbation $\widehat{\mathfrak S}_{k,\beta}$ of $\widehat{\mathcal U}_{k,\beta}^+$ such that $\ev_0$ is strongly submersive with respect to $\widehat{\mathfrak S}_{k,\beta}$. In other words, 
	$
	\tilde {\mathfrak X}_{k,\beta}= (\mathcal M_{k+1,\beta}, \widehat{\mathcal U}^+_{k,\beta}, (\ev_1,\dots, \ev_k),\ev_0, \widehat{\mathfrak S}_{k,\beta})
	$
	is a perturbed smooth correspondence.
	In this setting, if we further require that the restriction of $\widehat{\mathfrak S}_{k,\beta}$ to the normalized boundary is equivalent to the fiber product of $\widehat{\mathfrak S}_{k_1,\beta_1}$ and $\widehat{\mathfrak S}_{k_2,\beta_2}$ with respect to (\ref{tree_like_moduli_eq}) and natural analogous conditions hold for the normalized corners (cf. \cite[22.3 (4) \& (7)]{FOOO_Kuranishi}, then we say the collection $\widetilde{\mathbb {MU}}(J,L)=\{(\mathcal M_{k+1,\beta}, \widehat{\mathcal U}^+_{k,\beta}, \widehat{\mathfrak S}_{k,\beta}) \}$ a \textit{perturbed $A_\infty$ correspondence}.\footnote{This is an analog of perturbed smooth correspondence but is not specified in \cite{FOOO_Kuranishi}. We formulate it for clarity.}
	
	\begin{thm}
		\label{perturbed_A_infinity_correspondence_thm}
		Given an $A_\infty$ correspondence $\mathbb {MU}(J,L):= \{(\mathcal M_{k+1,\beta}, \widehat{\mathcal U}_{k,\beta})\}$, there exists a perturbed $A_\infty$ correspondence $\widetilde{\mathbb {MU}}(J,L):= \{(\mathcal M_{k+1,\beta}, \widehat{\mathcal U}^+_{k,\beta}, \widehat{\mathfrak S}_{k,\beta} )\}$.
	\end{thm}
	
	\begin{proof}[Sketch of proof]
		We can construct it inductively. The key is the following lemma \cite[17.78, 17.81]{FOOO_Kuranishi}:
		
		Let $(\mathcal M, \widehat{\mathcal U})$ be a K-space; let $\widehat{\mathcal U}^+_{\partial}$ be a Kuranishi structure that is a thickening of the normalized boundary $(\partial\mathcal M, \widehat{\mathcal U}_\partial)$ and has similar conditions for the normalized corners.
		Assume $\widehat{\mathfrak S}^+_\partial$ is a CF-perturbation of $\widehat{\mathcal U}^+_\partial$ with corresponding conditions.
		Then, there exists an extended Kuranishi structure $\widehat{\mathcal U}^+$ on $\mathcal M$ and a CF-perturbation $\widehat{\mathfrak S}^+$ of it. Moreover, if $\widehat{f}:\widehat{\mathcal U}\to N$ is strongly smooth and weakly submersive, then we can choose $\widehat{\mathfrak S}^+$ so that $\widehat{f}: \widehat{\mathcal U}^+\to N$ is strongly submersive with respect to $\widehat{\mathfrak S}^+$.
		
		The basic idea of this result is as follows. The boundary $\partial N$ of a manifold has a collared neighborhood in the form $\partial N\times [0,\epsilon)$.
		A similar theory holds for K-spaces \cite[17.38, 17.40]{FOOO_Kuranishi} and is functorial as adding and forgetting the collared structures induce equivalences of categories \cite[17.43, 17.46]{FOOO_Kuranishi}. Using this can extend $\widehat{\mathcal U}_\partial^+$ to an open neighborhood of $\partial \mathcal M$ in $\mathcal M$ \cite[17.58, 17.73]{FOOO_Kuranishi}. Lastly, a rather easy lemma tells that if a Kuranishi structure (resp. CF-perturbation) is defined on an open neighborhood of a compact subset, then we can extend it without changing it in a slightly smaller neighborhood of the compact subset \cite[17.75, 17.77]{FOOO_Kuranishi}.
	\end{proof}
	
	Naively speaking, the $A_\infty$ operator $\check \m_{k,\beta}$ mentioned in the introduction would be defined as
	\begin{equation}
		\label{m_k_beta_epsilon_eq}
		\check \m_{k,\beta}(h_1,\dots, h_k)=\ev_0!\big(\ev_1^*h_1\wedge \cdots \wedge \ev_k^* h_k; \widehat{\mathfrak S}_{k,\beta}^\epsilon \big)=\Corr_{\tilde{\mathfrak X}_{k,\beta}}^\epsilon (h_1, \dots,  h_k)
	\end{equation}
	for $h_1,\dots, h_k\in\Omega^*(L)$ and sufficiently small $\epsilon$.
	But, a critical problem is that for any $(k,\beta)$, there is a small number $\epsilon_{k,\beta}>0$ so that the above formula makes sense only when $0<\epsilon \leqslant \epsilon_{k,\beta}$.
	To offer some insight initially, suppose the infimum of these $\epsilon_{k,\beta}$ was positive, then we could see the $A_\infty$ formula as follows.
	Applying the Stokes' formula (\ref{stokes_eq}), we first have that for any $(k,\beta) \neq (1,0)$,
	\begin{align*}
		\textstyle
		d \circ \check \m_{k,\beta} (h_1,\dots, h_k)
		-
		\sum_{i=1}^k \check \m_{k,\beta} (h_1,\dots, h_{i-1}, dh_i, h_{i+1}, \dots, h_k) =  \Corr^\epsilon_{\partial \tilde{\mathfrak X}_{k,\beta}} (h_1,\dots, h_k)
	\end{align*}
	By the boundary description (\ref{tree_like_moduli_eq}) and the composition formula (Theorem \ref{composition_formula_thm}), the right side is given by $\Corr^\epsilon_{  \tilde {\mathfrak X}_{k_1,\beta_1}} \big(h_1,\dots, \Corr^\epsilon_{ \tilde {\mathfrak X}_{k_2,\beta_2}}(h_i,\dots, h_j),\dots, h_k \big)=\check \m_{k_1,\beta_1}(h_1,\dots, \check \m_{k_2,\beta_2} (h_i,\dots, h_j), \dots, h_k)$. Putting them together, the collection of multi-linear maps $\check \m_{k,\beta}$ forms an $A_\infty$ structure; cf. \cite[Section 22.4]{FOOO_Kuranishi}.
	
	Nevertheless, the infimum could be zero in general, and we need to impose energy cuts, say, as in \cite[22.3]{FOOO_Kuranishi}. 
	Let's finally explain the proof of \cite[21.35 (1)]{FOOO_Kuranishi}, i.e. Theorem \ref{A_infinity_algebra_up_to_pseudo_isotopy_thm}, as follows.
	
	\vspace{0.3em}
	
	\begin{proof}[Sketch of proof of Theorem \ref{A_infinity_algebra_up_to_pseudo_isotopy_thm}]
		Consider the set $\mathcal G$ of all pairs $(k,\beta)$ in $\mathbb N\times H_2(X,L)$ such that the topological space $\mathcal M_{k+1,\beta}(J,L)$ is non-empty.
		Then, $\mathcal G$ is countable by Gromov's compactness.
		Take a strictly increasing sequence of subsets $\{\mathcal G_n: n=1,2,\dots \}$ such that their union is the whole $\mathcal G$.
		May assume $\epsilon_n=\inf_{\mathcal G_n} \{\epsilon_{k,\beta} \}$ is strictly decreasing.
		For any $n$, we define operators $\check \m_{k,\beta}^{(n)}$ as (\ref{m_k_beta_epsilon_eq}) for $(k,\beta)\in\mathcal G_n$ and $\epsilon=\epsilon_n$.
		By using the Stokes formula \ref{stokes_eq} and the composition formula \ref{composition_formula_thm}, each $\check \m^{(n)}$ defines a \textit{partial} $A_\infty$ structure. However, it is not initially clear whether this can be promoted to a complete $A_\infty$ structure, as there are algebraic obstructions in general, cf \cite[Th 4.5.1]{FOOOBookOne}.
		To overcome this, we assert that $\check \m^{(n)}$ is (partially) pseudo-isotopic to $\check \m^{(n+1)}$. Indeed, for various $(k,\beta)$, we consider the trivial product spaces $\oi_s \times \mathcal M_{k+1,\beta}$; we take the scaled CF-perturbations $\widehat{\mathfrak S}_{k,\beta}^{\epsilon_n\epsilon}$ and $\widehat{\mathfrak S}_{k,\beta}^{\epsilon_{n+1}\epsilon}$ at the two ends $s=0,1$. They should be regarded as distinct data, and we can extend them to the product spaces by using the lemma mentioned in the proof of Theorem \ref{perturbed_A_infinity_correspondence_thm}. The extended CF-perturbations can be used to obtain partial pseudo-isotopies between $\check \m^{(n)}$ and $\check \m^{(n+1)}$, cf \cite[22.21]{FOOO_Kuranishi}. By the algebraic promotion lemma \cite[22.8]{FOOO_Kuranishi}, the existence of the (partial) pseudo-isotopies eliminates parts of the algebraic obstructions to promote $\check \m^{(n)}$. Finally, we can perform similar arguments between $\check \m^{(n)}$ and $\check \m^{(n+k)}$ for $k\gg 1$, which will eliminates all algebraic obstructions as $\bigcup_k\mathcal G_{n+k}=\mathcal G$.
		To sum up, for any fixed $n$, the $\check \m^{(n)}$ can be promoted to a complete $A_\infty$ structure, and each pair of them is pseudo-isotopic to each other.
	\end{proof}

	\bibliographystyle{abbrv}
	
	\bibliography{mybib_unobs}

\end{document}